\pdfoutput=1
\RequirePackage{ifpdf}
\ifpdf 
\documentclass[pdftex]{sigma}
\else
\documentclass{sigma}
\fi

\numberwithin{equation}{section}

\newtheorem{Theorem}{Theorem}[section]
\newtheorem{Corollary}[Theorem]{Corollary}
\newtheorem{Lemma}[Theorem]{Lemma}
\newtheorem{Proposition}[Theorem]{Proposition}
\newtheorem{Algorithm}[Theorem]{Algorithm}
 { \theoremstyle{definition}
\newtheorem{Definition}[Theorem]{Definition}
\newtheorem{Example}[Theorem]{Example}}

\usepackage{savesym}

\usepackage{mathabx}
\usepackage{tikz}
\usepackage[all,cmtip]{xy}
\usepackage{mleftright}
\usepackage{booktabs}
\usepackage{enumitem}
\usepackage{mathdots}
\usepackage{multicol}

\newcommand{\Z}{\mathbb{Z}}
\newcommand{\Q}{\mathbb{Q}}
\newcommand{\R}{\mathbb{R}}
\newcommand{\C}{\mathbb{C}}

\newcommand{\M}{\mathcal{M}}
\newcommand{\A}{\mathbb{A}}

\newcommand{\mO}{\mathcal{O}}

\newcommand{\eps}{\varepsilon}
\newcommand{\ph}{\varphi}

\newcommand{\pr}{\mathbb{P}}

\newcommand{\on}{\operatorname}
\newcommand{\ol}{\overline}

\newcommand{\ub}{\underbar}
\newcommand{\mfk}{\mathfrak}

\newcommand{\mat}{\left(\begin{array}}
\newcommand{\tam}{\end{array}\right)}

\newcommand{\hilb}{\text{\textnormal{-Hilb}}\,}
\newcommand{\mR}{\mathcal{R}}
\newcommand{\Hirz}{\text{\textnormal{Hirz}}}
\newcommand{\dP}{\text{\textnormal{dP}}}
\newcommand{\Gig}{\text{$G$-\textnormal{ig}}}

\allowdisplaybreaks

\begin{document}
\allowdisplaybreaks

\newcommand{\arXivNumber}{1908.05748}

\renewcommand{\PaperNumber}{106}

\FirstPageHeading

\ShortArticleName{Walls for~$G$-Hilb via Reid's Recipe}

\ArticleName{Walls for~$\boldsymbol{G}$-Hilb via Reid's Recipe}

\Author{Ben WORMLEIGHTON}

\AuthorNameForHeading{B.~Wormleighton}

\Address{Department of Mathematics and Statistics, Washington University in~St.~Louis,\\ MO 63130, USA}
\Email{\href{mailto:email@address}{benw@wustl.edu}}
\URLaddress{\url{https://sites.google.com/view/benw/}}

\ArticleDates{Received November 14, 2019, in~final form October 14, 2020; Published online October 24, 2020}

\Abstract{The three-dimensional McKay correspondence seeks to relate the geometry of crepant resolutions of Gorenstein $3$-fold quotient singularities $\mathbb{A}^3/G$ with the representation theory of the group $G$. The first crepant resolution studied in~depth was the $G$-Hilbert scheme $G\text{-Hilb}\,\mathbb{A}^3$, which is also a~moduli space of $\theta$-stable representations of the McKay quiver associated to $G$. As~the stability parameter $\theta$ varies, we obtain many other crepant resolutions. In~this paper we focus on the case where $G$ is abelian, and compute explicit inequalities for the chamber of the stability space defining $G\text{-Hilb}\,\mathbb{A}^3$ in~terms of a marking of exceptional subvarieties of $G\text{-Hilb}\,\mathbb{A}^3$ called Reid's recipe. We~further show which of these inequalities define walls. This procedure depends only on the combinatorics of the exceptional fibre and has applications to the birational geometry of other crepant resolutions.}


\Classification{14E16; 14M25; 16G20}

\section{Introduction} \label{sec:intro}

Let $G\subset\on{SL}_n(\C)$ be a~finite subgroup. When $n=2$ there is a~famous ADE classification of such subgroups that matches the classification of Du Val or modality zero singularities by taking a~subgroup $G$ to the quotient singularity $0\in\A^2/G$. This observation and the surrounding deep interactions of the geometry of $\A^2/G$ and its resolutions, and the representation theory of $G$ are known as the two-dimensional McKay correspondence~\cite{ArtinVerdier85,BatyrevDais96,DenefLoeser02,ItoNakajima00,ItoNakamura96,McKay81,Reid}. In~this case, the unique minimal or crepant resolution has a~modular interpretation as the $G$-Hilbert scheme $G\hilb\A^2$. The moduli space $G\hilb M$ for~$M$ a~variety and $G\subset\on{Aut}(M)$ a~finite subgroup parameterises $G$-clusters in~$M$: zero-dimensional $G$-invariant subschemes $M$ of $\A^2$ with $H^0(\mO_Z)\cong\C[G]$ as $G$-modules. This was generalised to three dimensions for finite abelian subgroups of $\on{SL}_3(\C)$ by Nakamura~\cite{Nakamura01} who showed that $G\hilb\A^3$ is a~crepant resolution of~$\A^3/G$ and then to all subgroups $G$ by the celebrated work of Bridgeland--King--Reid~\cite{BridgelandKingReid01}. They moreover established an equivalence of categories
\begin{gather} \label{eqn:BKR} 
D^b\big(G\hilb\A^3\big)\simeq D^b_G\big(\A^3\big),
\end{gather}
which also holds if $G\hilb\A^3$ is replaced by any projective crepant resolution of $\A^3/G$. Compare also the results of Bridgeland~\cite{Bridgeland02}.

Using the GIT approach of King~\cite{King94} to constructing moduli of quiver representations the \mbox{$G$-Hilbert} scheme can also be realised as a~moduli space of $\theta$-stable quiver representations $\M_\theta(Q,\ub{d})$, where $Q$ is the McKay quiver of $G$ and $\ub{d}=(d_i)$ is a~given dimension vector. In~this situation the stability parameter $\theta$ lives in~the stability space
\begin{gather*}
\Theta:=\Bigg\{\theta\in\on{Hom}_\Z\big(\Z^{Q_0},\Q\big)\colon \sum_{i\in Q_0} d_i\theta(i)=0\Bigg\},
\end{gather*}
where $Q_0$ is the set of vertices of $Q$. By definition, the vertices of the McKay quiver biject with the irreducible representations $\on{Irr}(G)$ of $G$ and so one can view $\Z^{Q_0}$ as the abelian group underlying the representation ring of $G$. As~$\theta$ varies, it is possible that one obtains many different crepant resolutions of $\A^3/G$; in~the case that $G$ is abelian, Craw--Ishii~\cite{CrawIshii04} show that all projective crepant resolutions arise in~this way. The stability space $\Theta$ has a~wall-and-chamber structure such that the moduli space $\M_\theta(Q,\ub{d})$ is constant so long as $\theta$ remains inside a~given chamber. We~denote the moduli space $\M_\mfk{C}:=\M_\theta(Q,\ub{d})$ for any generic $\theta$ in~a~chamber $\mfk{C}$. Denote the chamber corresponding to $G\hilb\A^3$ by $\mfk{C}_0$. The positive orthant
\begin{gather*}
\Theta^+:=\big\{\theta\in\Theta\colon \theta(\rho)>0\text{ for all nontrivial $\rho\in\on{Irr}(G)$}\big\}
\end{gather*}
lies inside $\mfk{C}_0$ however it is not usually equal to it. The primary purpose of this paper is to provide explicit combinatorial inequalities defining $\mfk{C}_0$ and identify precisely which of these define walls of~$\mfk{C}_0$. We~remark that such equations were computed for a~group of order $11$ in \cite[Example~9.6]{CrawIshii04}.

Assume that $G$ is abelian. In~this context \cite[Theorem~9.5]{CrawIshii04} gives an abstract description of such inequalities, however it is difficult to perform explicit calculations or deduce general statements from their presentation. One can view some of the results herein as a~combinatorialisation of \cite[Theorem~9.5]{CrawIshii04}, which turn out to be very amenable to applications.

We briefly outline the context and notation of \cite{CrawIshii04} that we will also use. For a~chamber $\mfk{C}\subset\Theta$ the equivalence from (\ref{eqn:BKR}) induces an isomorphism $\ph_\mfk{C}\colon K_0(\M_\mfk{C})\to K_G\big(\A^3\big)=\on{Rep}(G)$. Here $K_0(\M_\mfk{C})$ denotes the $K$-group of sheaves supported on the preimage of the $G$-orbit for the origin under the resolution $\M_\mfk{C}\to\A^3/G$. Walls in~$\Theta$ are cut out by hyperplanes $\big(\sum_i\alpha_i\cdot\theta(\chi_i)=0\big)$ for some characters $\chi_i\in\on{Irr}(G)$ and integers $\alpha_i\in\Z$, though in~general not all $\theta$ on such a~hyperplane will be non-generic. The inequalities in~\cite{CrawIshii04} have three different forms, each coming from exceptional subvarieties. Firstly, each exceptional curve $C\subset G\hilb\A^3$ gives an inequality of the form
\begin{gather*}
\theta\big(\ph_{\mfk{C}_0}(\mO_C)\big)>0.
\end{gather*}
The characters appearing in~these inequalities are packaged in~collections of monomials associated to exceptional curves that were named by Nakamura~\cite{Nakamura01} in~a~different context as $G$\textit{-igsaw pieces}. Our first result is to pin down which characters lie in~$G$-igsaw pieces. In~general there are several $G$-igsaw pieces corresponding to a~single curve $C$; the union of all the pieces that do not include the trivial character is the set of characters that appear in~the inequality for~$C$. We~call this union the \emph{total $G$-igsaw piece}.

As $G$ is abelian the singularity $\A^3/G$ and its crepant resolutions are toric. There is a~method of marking the exceptional subvarieties of $G$-Hilb~-- the edges and vertices in~the triangulation~-- by characters of $G$ known as ``Reid's recipe''. This was used to explicitly describe the McKay correspondence in~the classical terms of providing a~basis of $H^*\big(G\hilb\A^3,\Z\big)$ indexed by characters by Craw~\cite{Craw05}. It~was later categorified by Logvinenko~\cite{Logvinenko08}, Cautis--Logvinenko~\cite{CautisLogvinenko09}, and Cautis--Craw--Logvinenko~\cite{CautisCrawLogvinenko17}, who expressed the locus labelled by a~character $\chi$ in~terms of the support of an object associated to $\chi$ in~the derived category of $G$-Hilb. We~will discuss this in~more detail in~Sections~\ref{sec:rr} and \ref{sec:taut}.

\begin{Theorem}[Algorithm \ref{alg:up}] \label{thm:pri} There is a~combinatorial procedure that we call the unlocking procedure for computing the characters appearing in~the total $G$-igsaw piece for an exceptional curve in~$G\hilb\A^3$. The input of this procedure is the data of Reid's recipe and the combinatorics of the exceptional fibre.
\end{Theorem}

To briefly illustrate how the procedure works, we consider the example of $G=\frac{1}{30}(25,2,3)$. This notation means that $G$ is the subgroup of $\on{SL}_3(\C)$ generated by
\begin{gather*}
g=\begin{pmatrix}
\eps^{25} \\
& \eps^2 \\
& & \eps^3
\end{pmatrix},
\end{gather*}
where $\eps$ is a~primitive $30$th root of unity. Crepant resolutions correspond to triangulations of the simplex at height $1$~-- the ``junior simplex''~-- with vertices in~the lattice $\Z^3+\Z\cdot\big(\frac{25}{30},\frac{2}{30},\frac{3}{30}\big)$. The~triangulation for~$G$-Hilb is shown in~Fig.~\ref{fig:30t} along with Reid's recipe. We~often denote the character $\chi_a$ defined by $\chi_a(g)=\eps^a$ by the integer $a$.

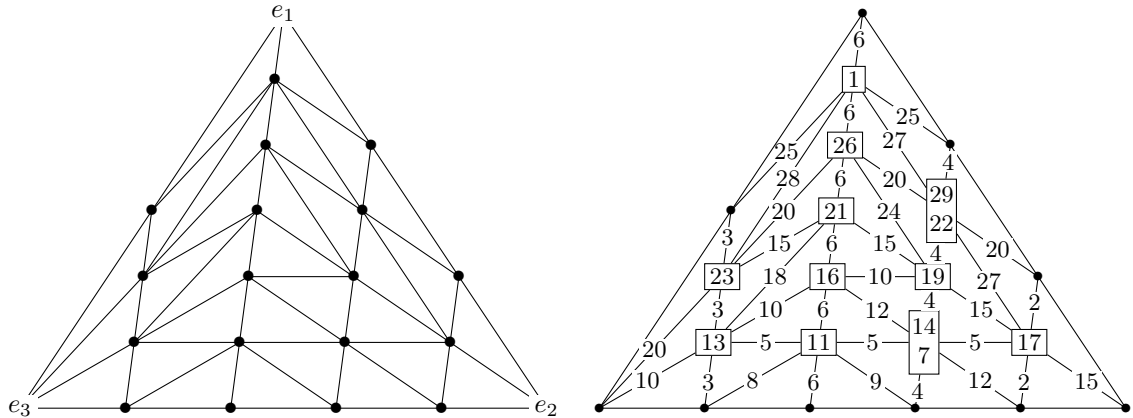
\begin{figure}[h]\centering
\begin{tikzpicture}[scale=0.35] \label{fig:30t}
\node (e1) at (0,9){};
\node (e2) at (10,-6){};
\node (e3) at (-10,-6){};
\node (a1) at (-2/6,9-5/2){$\bullet$};
\node (a2) at (-4/6,9-10/2){$\bullet$};
\node (a3) at (-6/6,9-15/2){$\bullet$};
\node (a4) at (-8/6,9-20/2){$\bullet$};
\node (a5) at (-10/6,9-25/2){$\bullet$};

\draw(e3.center) to (a2.center);
\draw(e3.center) to (a4.center);
\draw(e2.center) to (a3.center);

\node (xy1) at (10/3,4){$\bullet$};
\node (xy2) at (20/3,-1){$\bullet$};

\node (xz1) at (-5,1.5){$\bullet$};

\node (yz1) at (-10+4,-6){$\bullet$};
\node (yz2) at (-10+8,-6){$\bullet$};
\node (yz3) at (-10+12,-6){$\bullet$};
\node (yz4) at (-10+16,-6){$\bullet$};

\draw (e1.center) to (yz2.center);

\node (iz1) at (-16/3,-1){$\bullet$};
\node (iz2) at (-17/3,-3.5){$\bullet$};

\draw (xz1.center) to (yz1.center);
\draw (xz1.center) to (a1.center);
\draw (iz1.center) to (a1.center);

\draw (iz1.center) to (a3.center);
\draw (iz2.center) to (a3.center);

\draw (iz2.center) to (a5.center);
\draw (yz1.center) to (a5.center);

\node (c1) at (-1+11/3,-1){$\bullet$};
\node (c2) at (-1+22/3,-3.5){$\bullet$};

\node (dp1) at (9/3,1.5){$\bullet$};
\node (dp2) at (7/3,-3.5){$\bullet$};

\draw (xy2.center) to (yz4.center);
\draw (xy1.center) to (yz3.center);
\draw (a5.center) to (c2.center);
\draw (a5.center) to (yz3.center);
\draw (a4.center) to (yz4.center);
\draw (a4.center) to (c1.center);
\draw (a2.center) to (xy2.center);
\draw (a1.center) to (xy1.center);
\draw (a1.center) to (c2.center);
\draw (a2.center) to (c1.center);

\draw (e1.center) to (e2.center) to (e3.center) to (e1.center);

\small
\node[fill=white,inner sep=2pt] (e1) at (0,9){$e_1$};
\node[fill=white,inner sep=2pt] (e2) at (10,-6){$e_2$};
\node[fill=white,inner sep=2pt] (e3) at (-10,-6){$e_3$};

\footnotesize
\node (e1) at (0+22,9){$\bullet$};
\node (e2) at (10+22,-6){$\bullet$};
\node (e3) at (-10+22,-6){$\bullet$};

\node[draw, fill=white,inner sep=2pt] (a1) at (-2/6+22,9-5/2){$1$};
\node[draw, fill=white,inner sep=2pt] (a2) at (-4/6+22,9-10/2){$26$};
\node[draw, fill=white,inner sep=2pt] (a3) at (-6/6+22,9-15/2){$21$};
\node[draw, fill=white,inner sep=2pt] (a4) at (-8/6+22,9-20/2){$16$};
\node[draw, fill=white,inner sep=2pt] (a5) at (-10/6+22,9-25/2){$11$};

\node (xy1) at (10/3+22,4){$\bullet$};
\node (xy2) at (20/3+22,-1){$\bullet$};

\node (xz1) at (-5+22,1.5){$\bullet$};

\node (yz1) at (-10+4+22,-6){$\bullet$};
\node (yz2) at (-10+8+22,-6){$\bullet$};
\node (yz3) at (-10+12+22,-6){$\bullet$};
\node (yz4) at (-10+16+22,-6){$\bullet$};

\draw (e1.center) to node[fill=white,inner sep=1pt] {$6$} (a1) to node[fill=white,inner sep=1pt] {$6$} (a2) to node[fill=white,inner sep=1pt] {$6$} (a3) to node[fill=white,inner sep=1pt] {$6$} (a4) to node[fill=white,inner sep=1pt] {$6$} (a5) to node[fill=white,inner sep=1pt] {$6$} (yz2.center);

\node[draw, fill=white,inner sep=2pt] (iz1) at (-16/3+22,-1){$23$};
\node[draw, fill=white,inner sep=2pt] (iz2) at (-17/3+22,-3.5){$13$};

\draw (xz1.center) to node[fill=white,inner sep=1pt] {$3$} (iz1) to node[fill=white,inner sep=1pt] {$3$} (iz2) to node[fill=white,inner sep=1pt] {$3$} (yz1.center);
\draw (xz1.center) to node[fill=white,inner sep=1pt] {$25$} (a1);
\draw (iz1) to node[fill=white,inner sep=1pt] {$28$} (a1);

\draw (iz1) to node[fill=white,inner sep=1pt] {$15$} (a3);
\draw (iz2) to node[fill=white,inner sep=1pt] {$18$} (a3);

\draw (iz2) to node[fill=white,inner sep=1pt] {$5$} (a5);
\draw (yz1.center) to node[fill=white,inner sep=1pt] {$8$} (a5);
\draw (e1.center) to (e2.center) to (e3.center) to (e1.center);

\node[draw, fill=white,inner sep=2pt] (c1) at (-1+22+11/3,-1){$19$};
\node[draw, fill=white,inner sep=2pt] (c2) at (-1+22+22/3,-3.5){$17$};

\node[draw, fill=white,inner sep=1pt] (dp1) at (9/3+22,1.5){$\begin{matrix} 29 \\ 22\end{matrix}$};
\node[draw, fill=white,inner sep=1pt] (dp2) at (7/3+22,-3.5){$\begin{matrix} 14 \\ 7\end{matrix}$};

\draw(e3.center) to node[fill=white,inner sep=1pt] {$20$} (iz1) to node[fill=white,inner sep=1pt] {$20$} (a2);
\draw(e3.center) to node[fill=white,inner sep=1pt] {$10$} (iz2) to node[fill=white,inner sep=1pt] {$10$} (a4);
\draw(e2.center) to node[fill=white,inner sep=1pt] {$15$} (c2) to node[fill=white,inner sep=1pt] {$15$} (c1) to node[fill=white,inner sep=1pt] {$15$} (a3);

\draw (xy2.center) to node[fill=white,inner sep=1pt] {$2$} (c2) to node[fill=white,inner sep=1pt] {$2$} (yz4.center);
\draw (xy1.center) to node[fill=white,inner sep=1pt] {$4$} (dp1) to node[fill=white,inner sep=1pt] {$4$} (c1) to node[fill=white,inner sep=1pt] {$4$} (dp2) to node[fill=white,inner sep=1pt] {$4$} (yz3.center);
\draw (a5) to node[fill=white,inner sep=1pt] {$5$} (dp2) to node[fill=white,inner sep=1pt] {$5$} (c2);
\draw (a5) to node[fill=white,inner sep=1pt] {$9$} (yz3.center);
\draw (a4) to node[fill=white,inner sep=1pt] {$12$} (dp2) to node[fill=white,inner sep=1pt] {$12$} (yz4.center);
\draw (a4) to node[fill=white,inner sep=1pt] {$10$} (c1);
\draw (a2) to node[fill=white,inner sep=1pt] {$20$} (dp1) to node[fill=white,inner sep=1pt] {$20$} (xy2.center);
\draw (a1) to node[fill=white,inner sep=1pt] {$25$} (xy1.center);
\draw (a1) to node[fill=white,inner sep=1pt] {$27$} (dp1) to node[fill=white,inner sep=1pt] {$27$} (c2);
\draw (a2) to node[fill=white,inner sep=1pt] {$24$} (c1);

\end{tikzpicture}
\caption{$G$-Hilb and Reid's recipe for~$G=\frac{1}{30}(25,2,3)$.}
\end{figure}

Throughout this paper we use the convention of ordering vertices as shown in~Fig.~\ref{fig:30t}. We~will demonstrate the unlocking procedure for the curve $C$ shown on the left side of Fig.~\ref{fig:unlock_ex} marked with the character $5$; that is, the character taking $g\mapsto\eps^5$. On the right side of Fig.~\ref{fig:unlock_ex} we~illustrate the unlocking procedure. Roughly, we consider all the curves (or~edges) marked with~$5$, add one character marking each divisor containing two such curves (or~vertices between two edges marked with $5$), and finally add the characters appearing in~$G$-igsaw pieces for certain curves cohabiting a~divisor with a~curve marked with $5$. In~this case, the only such extra curve is marked with $9$ and the $G$-igsaw piece for this curve consists just of the character $9$ itself. Some recursion will be necessary in~general to compute the smaller $G$-igsaw pieces of such curves. It~follows that the total $G$-igsaw piece for~$C$ has characters 5, 7, 9, 11. Observe that this total $G$-igsaw piece only picked one of the two characters 7, 14 marking a~divisor containing two $5$-curves. We~will elaborate in~Section~\ref{sec:comb_def} how the unlocking procedure identifies which of the two characters should be added.

Following~\cite{Wilson92}, walls inside $\Theta$ are of various types denoted \texttt{0}-\texttt{III} depending on the birational geometry of the moduli spaces near the wall. Walls of Type \texttt{I} correspond to flops induced by curves. By \cite[Theorem~9.12]{CrawIshii04}, every flop in~a~single exceptional $(-1,-1)$-curve can be realised by a~wall-crossing of Type \texttt{I} directly from $\mfk{C}_0$, which is very much not true for other resolutions; see \cite[Example~9.13]{CrawIshii04}. Walls of Type \texttt{III} -- that contract a~divisor to a~curve~-- arise from exceptional $(-2,0)$-curves corresponding to certain ``boundary'' edges in~the triangulation for~$G$-Hilb. We~will define this terminology in~Section~\ref{sec:rr}. The final possibility is that an exceptional curve is a~$(1,-3)$-curve that, if the corresponding inequality was irredundant, would produce a~wall of Type~\texttt{II} where by definition a~divisor is contracted to a~point. \cite[Proposition~3.8]{CrawIshii04} shows that there are in~fact no Type \texttt{II} walls in~$\Theta$.

We denote the set of characters appearing in~the total $G$-igsaw piece for an exceptional curve $C$ by $\Gig(C)$. The following result is implied by \cite[Corollary~5.2, Proposition~9.7 and Theorem~9.12]{CrawIshii04} and Theorem~\ref{thm:pri}.

\begin{figure}[t]\centering
\begin{tikzpicture}[scale=0.2]
\footnotesize
\node (e1) at (0,9){$\bullet$};
\node (e2) at (10,-6){$\bullet$};
\node (e3) at (-10,-6){$\bullet$};

\draw (e1.center) to (e2.center) to (e3.center) to (e1.center);

\node (a5) at (-10/6,9-25/2){$\bullet$};
\node (iz2) at (-17/3,-3.5){$\bullet$};

\draw (iz2.center) to node[fill=white,inner sep=1pt] {$\mathbf{5}$} (a5.center);

\node (e1) at (0+22,9){$\bullet$};
\node (e2) at (10+22,-6){$\bullet$};
\node (e3) at (-10+22,-6){$\bullet$};

\draw (e1.center) to (e2.center) to (e3.center) to (e1.center);

\node[draw, fill=white,inner sep=2pt] (a5) at (-10/6+22,9-25/2){$11$};
\node (yz3) at (-10+12+22,-6){$\bullet$};
\node (iz2) at (-17/3+22,-3.5){$\bullet$};
\node (c2) at (-1+22/3+22,-3.5){$\bullet$};
\node[draw, fill=white,inner sep=2pt] (dp2) at (7/3+22,-3.5){$7$};

\draw (iz2.center) to node[fill=white,inner sep=1pt] {$\mathbf{5}$} (a5);
\draw (a5) to node[fill=white,inner sep=1pt] {$5$} (dp2) to node[fill=white,inner sep=1pt] {$5$} (c2.center);
\draw (a5) to node[fill=white,inner sep=1pt] {$9$} (yz3.center);
\end{tikzpicture}
\caption{Unlocking for a~$5$-curve.}\label{fig:unlock_ex}
\end{figure}
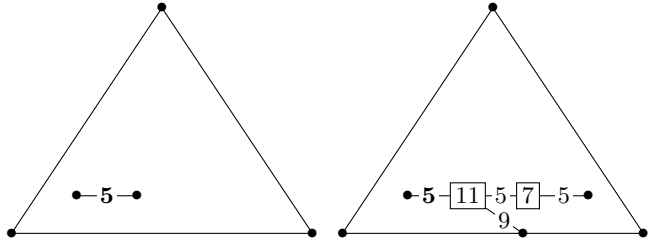

\begin{Proposition}[Propositions~\ref{prop:i} and \ref{prop:II}] \label{prop:curves} Suppose $C\subset G\hilb\mathbb{A}^3$ is an exceptional curve marked with character $\chi$ by Reid's recipe. The inequality corresponding to $C$ is given by
\begin{gather*}
\theta(\ph_{\mfk{C}_0}(\mO_C))=\sum_{\chi\in\Gig(C)}\deg (\mR_\chi|_C)\theta(\chi)>0.
\end{gather*}
If $C$ is a~$(-1,-1)$-curve then the necessary inequality corresponding to $C$ that defines a~Type \texttt{I} wall of $\mfk{C}_0$ is given by
\begin{gather*}
\theta\big(\ph_{\mfk{C}_0}(\mO_C)\big)=\sum_{\chi\in\Gig(C)}\theta(\chi)>0.
\end{gather*}
If $C$ is a~$(1,-3)$-curve then the inequality corresponding to $C$ is given by
\begin{gather*}
\theta\big(\ph_{\mfk{C}_0}(\mO_C)\big)=2\cdot\theta\big(\chi^{\otimes 2}\big)+\sum_{\chi\in\Gig(C)\setminus\{\chi^{\otimes 2}\}}\theta(\chi)=0.
\end{gather*}
In all three cases $\Gig(C)$ is computed by the unlocking procedure.
\end{Proposition}

We can also use the unlocking procedure to compute inequalities that do not come from exceptional curves. The other two kinds of inequality come from exceptional divisors. For each character $\psi$ marking a~divisor, we obtain an inequality $\theta(\psi)>0$. The second kind of inequality coming from divisors is more complicated.

\begin{Proposition}[Proposition~\ref{prop:rigid}] \label{prop:rigid1} Suppose $D'$ is a~(not necessarily prime) exceptional divisor in~$G\hilb\A^3$. Then any $\theta\in\mfk{C}_0$ satisfies
\begin{gather*}
\theta\big(\ph_\mfk{C_0}(\omega_{D'}^\vee)\big)=\sum_{C\subset D'}\sum_{\chi\in\Gig(C)}\theta(\chi)>0,
\end{gather*}
where $C$ ranges over exceptional curves inside $D'$.
\end{Proposition}

As a~result of Propositions \ref{prop:curves} and \ref{prop:rigid1} we can immediately deduce the conclusion \cite[Proposition~3.8]{CrawIshii04} for~$\mfk{C}_0$.

\begin{Corollary}[Corollary~\ref{cor:ii}] $\mfk{C}_0$ has no Type \texttt{II} walls.
\end{Corollary}

We can similarly reprove \cite[Theorem~9.12]{CrawIshii04} by combinatorial means.

\begin{Proposition}[Proposition~\ref{prop:-1nec}] Each flop in~a~$(-1,-1)$-curve in~$G\hilb\A^3$ is induced by a~wall-crossing from $\mfk{C}_0$.
\end{Proposition}

We can use these formulae to show exactly which inequalities are necessary to define $\mfk{C}_0$.

\begin{Theorem}[Theorem~\ref{thm:main}] Suppose $G\subset\on{SL}_3(\C)$ is a~finite abelian subgroup. The walls of the chamber $\mfk{C}_0$ for~$G\hilb\A^3$ and their types are as follows:
\begin{itemize}\itemsep=0pt
\item a~Type \texttt{I} wall for each exceptional $(-1,-1)$-curve,
\item a~Type \texttt{III} wall for each generalised long side,
\item a~Type \texttt{0} wall for each irreducible exceptional divisor,
\item the remaining walls are of Type \texttt{0} coming from divisors as in~Proposition~{\rm \ref{prop:rigid1}}. We~discuss which of these are necessary and how to reconstruct the divisor $D'$ in~Section~{\rm \ref{sec:sum}}.
\end{itemize}
\end{Theorem}

We will define the term ``generalised long side'' in~Definition \ref{def:gls}, which is an entirely combinatorial notion. We~complete the description of wall-crossing behaviours from~\cite{Wilson92} by calling a~wall Type \texttt{0} if the corresponding contraction is a~isomorphism.

This explicit and malleable description of the walls for~$\mfk{C}_0$ has applications to studying the birational geometry of other crepant resolutions of $\A^3/G$. In~forthcoming work~\cite{ItoWormleighton}, the author and Y.~Ito use this description of $\mfk{C}_0$ to study the geometry of another Hilbert scheme-like resolution introduced in~\cite{IshiiItoNolla13} called the ``iterated $G$-Hilbert scheme'' or ``Hilb of Hilb''.

\section[Resolutions of A3/G]{Resolutions of $\boldsymbol{\A^\mathbf{3}/G}$}

\subsection{Setup}

Let $G\subset\on{SL}_3(\C)$ be a~finite abelian subgroup. We~will assume that $G$ is cyclic for notational simplicity during this preliminary section however we will shed this assumption from Section~\ref{sec:comp} onwards. We~will denote by $\frac{1}{r}(a,b,c)$ the cyclic subgroup of $\on{SL}_3(\C)$ generated by the matrix
\begin{gather*}
g=\begin{pmatrix}
\eps^a \\
& \eps^b \\
& & \eps^c
\end{pmatrix},
\end{gather*}
where $\eps$ is a~primitive $r$th root of unity. The first resolution of the quotient singularity $\A^3/G$ to be studied was the $G$-Hilbert scheme $G\hilb\A^3$, the fine moduli space of $G$-clusters: zero-dimensional $G$-invariant subschemes $Z\subset\A^3$ with $H^0(\mO_Z)\cong\C[G]$ as $G$-modules. $G$-Hilb was shown to be smooth for abelian $G$ by Nakamura~\cite{Nakamura01}, and subsequently shown to be smooth~-- and hence a~resolution~-- for all finite $G\subset\on{SL}_3(\C)$ by Bridgeland--King--Reid~\cite{BridgelandKingReid01}.

From the work of King~\cite{King94}, Ito--Nakajima~\cite{ItoNakajima00}, and Craw~\cite{Craw01} one can reinterpret $G$-Hilb as a~moduli space of quiver representations. The quiver in~question is the \textit{McKay quiver} with vertices indexed by irreducible representations of $G$ and the number of arrows between $\rho$ and $\rho'$ defined to be
\begin{gather*}
\on{dim}\on{Hom}_G(\rho'\otimes\rho_{\text{std}},\rho),
\end{gather*}
where $\rho_\text{std}$ is the standard representation of $G$ acting on~$\C^3$. We~choose the dimension vector $\ub{d}=(\on{dim}{\rho})_\rho$ and a~stability parameter $\theta\in\Theta$ as defined above. We~define $\M_\theta:=\M_\theta(Q,\ub{d})$ to be the fine moduli space of $\theta$-stable representations of the McKay quiver with dimension vector~$\ub{d}$ subject to certain relations coming from the associated preprojective algebra. When $\theta(\rho)>0$ for all nontrivial $\rho$ one has that $\M_\theta(Q,\ub{d})=G\hilb\A^3$.

It was apparent from~\cite{BridgelandKingReid01} that their smoothness result and equivalence of categories (\ref{eqn:BKR}) holds for any generic $\theta$ and so one obtains potentially many different resolutions of the form $\M_\theta(Q,\ub{d})$ and also the corresponding equivalences of categories
\begin{gather*}
\Phi_\theta\colon\ D^b(\M_\theta)\to D^b_G\big(\A^3\big).
\end{gather*}

As discussed above, the stability space $\Theta$ has a~wall-and-chamber structure such that any~$\theta$,~$\theta'$ from the same open chamber $\mfk{C}\subset\Theta$ produce isomorphic moduli spaces: $\M_\theta\cong\M_{\theta'}$. For~simplicity we denote by $\M_\mfk{C}$ and $\Phi_\mfk{C}$ the moduli space and equivalence of categories for any generic $\theta\in\mfk{C}$.

When $G$ is abelian, each resolution $\M_\mfk{C}$ is toric. Fix the lattice $N=\Z^3+\Z\cdot\big(\frac{a}{r},\frac{b}{r},\frac{c}{r}\big)$. The singularity $\A^3/G$ is described by the cone $\sigma=\on{Cone}(e_1,e_2,e_3)$ inside $N_\R=N\otimes_Z\R=\R^3_{\langle x_1,x_2,x_3\rangle}$ and crepant resolutions correspond to triangulations of the slice $\sigma\cap(x_1+x_2+x_3=1)$~-- usually referred to as the ``junior simplex''~-- such that the vertices of each triangle lies in~$N$, and each triangle is smooth (its vertices form a~$\Z$-basis of $N$). In~pictures we will always draw only the slice to produce two-dimensional pictures.

\subsection{Reid's recipe} \label{sec:rr}

Let us focus on the case $\M_\mfk{C}=G\hilb\A^3$. We~will denote the universal $G$-cluster by $\mathcal{Z}$ and the chamber of $\Theta$ corresponding to $G$-Hilb by $\mfk{C}_0$. Craw--Reid~\cite{CrawReid02} present an entertaining algorithm to construct the triangulation for~$G$-Hilb that, after commenting on some of the salient features, we will use without comment.

We call the triangulation for~$G$-Hilb the \textit{Craw--Reid triangulation}. It~divides the junior simplex into so-called ``regular triangles'' of equal side length that fall into one of two cases:
\begin{itemize}\itemsep=0pt
\item \textit{corner triangles}, which have one of the vertices $e_1$, $e_2$, $e_3$ of the junior simplex as a~vertex,
\item \textit{meeting of champions}, for which none of the vertices of the junior simplex are vertices.
\end{itemize}
Craw--Reid show that there is at most one meeting of champions triangle (possibly of side length zero, in~which case it is a~point). After dividing the junior simplex into such triangles, one subdivides them further into smooth triangles as depicted in~Fig.~\ref{fig:reg}: the resulting unimodal triangulation describes the resolution $G$-Hilb. We~call line segments on the boundary of a regular triangle \emph{boundary edges} and the corresponding toric curves \emph{boundary curves}.

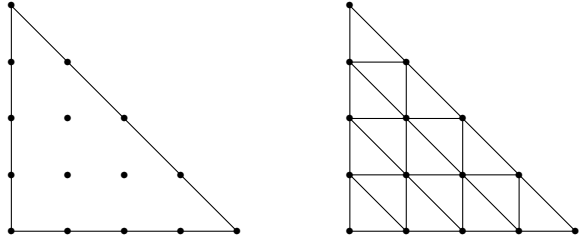
\begin{figure}[h]\centering
\begin{tikzpicture}[scale=0.75]
\tiny
\node (a0) at (0,0){$\bullet$};
\node (a1) at (1,0){$\bullet$};
\node (a2) at (2,0){$\bullet$};
\node (a3) at (3,0){$\bullet$};
\node (a4) at (4,0){$\bullet$};
\node (b0) at (0,1){$\bullet$};
\node (c0) at (0,2){$\bullet$};
\node (d0) at (0,3){$\bullet$};
\node (e0) at (0,4){$\bullet$};
\node (b1) at (1,1){$\bullet$};
\node (b2) at (2,1){$\bullet$};
\node (b3) at (3,1){$\bullet$};
\node (c1) at (1,2){$\bullet$};
\node (c2) at (2,2){$\bullet$};
\node (d1) at (1,3){$\bullet$};

\draw (a0.center) to (a4.center) to (e0.center) to (a0.center);

\node (a0) at (6,0){$\bullet$};
\node (a1) at (7,0){$\bullet$};
\node (a2) at (8,0){$\bullet$};
\node (a3) at (9,0){$\bullet$};
\node (a4) at (10,0){$\bullet$};
\node (b0) at (6,1){$\bullet$};
\node (c0) at (6,2){$\bullet$};
\node (d0) at (6,3){$\bullet$};
\node (e0) at (6,4){$\bullet$};
\node (b1) at (7,1){$\bullet$};
\node (b2) at (8,1){$\bullet$};
\node (b3) at (9,1){$\bullet$};
\node (c1) at (7,2){$\bullet$};
\node (c2) at (8,2){$\bullet$};
\node (d1) at (7,3){$\bullet$};

\draw (a0.center) to (a4.center) to (e0.center) to (a0.center);
\draw (a1.center) to (b0.center);
\draw (a2.center) to (c0.center);
\draw (a3.center) to (d0.center);
\draw (b0.center) to (b3.center);
\draw (c0.center) to (c2.center);
\draw (d0.center) to (d1.center);
\draw (a1.center) to (d1.center);
\draw (a2.center) to (c2.center);
\draw (a3.center) to (b3.center);
\end{tikzpicture}
\caption{A regular triangle and its triangulation.}\label{fig:reg}
\end{figure}

In early versions of the McKay correspondence~\cite{Reid} one of the chief aims was to supply a~bijection from irreducible characters of $G$ to a~basis of cohomology on a~crepant resolution. This was explicitly computed for~$G$-Hilb by Craw~\cite{Craw05} when $G$ is abelian using ``Reid's recipe'': a~labelling of exceptional subvarieties by characters of $G$. Reid's recipe is one of the main tools we will use to compute walls and so we will describe it in~some detail.

An exceptional curve $C$ in~$G$-Hilb corresponds to an edge in~the Craw--Reid triangulation, which in~turn corresponds to a~two-dimensional cone in~the fan for~$G$-Hilb. A primitive normal vector $(\alpha,\beta,\gamma)$ to this cone defines a~$G$-invariant ratio of monomials
\begin{gather*}
x^\alpha y^\beta z^\gamma=m_1/m_2,
\end{gather*}
where $x$, $y$, $z$ are eigencoordinates on~$\C^3$ for~$G$. Mark the curve $C$ with the character by which~$G$ acts on~$m_1$ (or~$m_2$). We~define the $\chi$\textit{-chain} to be the collection of all exceptional curves (or~edges in~the Craw--Reid triangulation) marked by the character $\chi$. We~say that a~triangle in~the Craw--Reid triangulation is a~$\chi$\textit{-triangle} if one of its edges is marked with the character~$\chi$.

After marking all curves, there is a~procedure for labelling the compact exceptional divisors, or interior vertices of the triangulation. Let $D$ be such a~divisor corresponding to a~vertex $v$. There are three cases:
\begin{itemize}\itemsep=0pt
\item $v$ is trivalent: $D\cong\pr^2$ and the three exceptional curves in~$D$ are all marked with the same character $\chi$. \textit{Mark $D$ with $\chi^{\otimes 2}$}.
\item $v$ is $4$- or $5$-valent, or $6$-valent and not inside a~regular triangle: $D$ is a~Hirzebruch surface blown up in~$\text{valency-}4$ points. There are two pairs of exceptional curves in~$D$ each marked with the same character $\chi$ and $\chi'$. \textit{Mark $D$ with $\chi\otimes\chi'$}.
\item $v$ is $6$-valent and lies in~the interior of a regular triangle: $D$ is a~del Pezzo surface of degree~$6$, and there are three pairs of exceptional curves each marked with the same character $\chi$,~$\chi'$,~$\chi''$. $D$ has two $G$-invariant maps to $\pr^2$, \textit{mark $D$ by the two characters arising from the monomials constituting these two maps}. These two characters $\phi_1$, $\phi_2$ satisfy
\begin{gather*}
\chi\otimes\chi'\otimes\chi''=\phi_1\otimes\phi_2.
\end{gather*}
\end{itemize}
For more detail see \cite[Lemmas~3.1--3.4]{Craw05}. We~will frequently refer to a~curve or a~divisor marked with a~character $\chi$ as a~$\chi$-curve or a~$\chi$-divisor.

\begin{Example} In Fig.~\ref{fig:30t} with $G=\frac{1}{30}(25,2,3)$, the leftmost curve marked with the character~$20$ has normal $(-2,25,0)$ giving the $G$-invariant ratio $y^{25}/x^2$. $G$ acts on the numerator and denominator by the character $\eps\mapsto\eps^{20}$, hence the marking. The divisor marked with $23$ incident to the previous curve marked with $20$ has two pairs of curves with characters $20$ and $3$ and a~fifth curve with character $15$. Thus, the divisor is correctly marked by $20+3=23$.
\end{Example}

We refer to divisors of the first two types~-- that is, all divisors not isomorphic to a~del Pezzo surface of degree $6$~-- as \textit{Hirzebruch divisors}, and to divisors isomorphic to a~del Pezzo surface of degree $6$ as \textit{del Pezzo divisors}. We~ask the reader to have grace on the slight abuse of terminology as $\pr^2$ is also a~del Pezzo surface. For a~character $\chi$ marking a~curve, we denote by $\Hirz(\chi)$ the set of characters marking Hirzebruch divisors in~the interior of the $\chi$-chain and by $\dP(\chi)$ the set of characters marking del Pezzo divisors in~the interior of the $\chi$-chain. We~will often say ``along the $\chi$-chain'' in~place of ``in the interior of the $\chi$-chain''.

\subsection[G-igsaw pieces]{$\boldsymbol{G}$-igsaw pieces}
Consider the $G$-clusters at torus-fixed points of $G$-Hilb, or triangles in~the Craw--Reid triangulation. The ideal defining such a~cluster is a~monomial ideal and one can draw a~Newton polygon in~the hexagonal lattice $\Z^3/\Z\cdot(1,1,1)$ to illustrate the monomial basis. An example of a torus-fixed $G$-cluster for the group $G=\frac{1}{6}(1,2,3)$ is shown in~Fig.~\ref{fig:clust}. Notice that there is exactly one monomial in~each character space for~$G$ as desired.

\begin{figure}[h]\centering
$$\begin{array}{ccccccccccc}
& y \\
yz & & 1 & & x \\
& z & & xz
\end{array}$$
\caption{A $G$-cluster for~$G=\frac{1}{6}(1,2,3)$ corresponding to a~torus-fixed point of $G$-Hilb.}\label{fig:clust}
\end{figure}

The monomial ideal in~$\C[x,y,z]$ defining this cluster is
\begin{gather*}
\big\langle x^2,y^2,z^2,xy\big\rangle.
\end{gather*}

$G$-clusters corresponding to adjacent triangles separated by an exceptional curve $C$ differ by taking a~subset of the monomials basing one $G$-cluster and moving them to other monomials in~the same character space; that is, multiplying by $G$-invariant ratios of monomials. This process was studied in~\cite{Nakamura01} and called a~$G$\textit{-igsaw transformation}.

\begin{Definition} Let $C\subset G\hilb\A^3$ be an exceptional curve corresponding to the common edge of two adjacent triangles $\tau$, $\tau'$. Denote by $Z_\tau$, $Z_{\tau'}$ the two $G$-clusters corresponding to the torus-fixed points of $G\hilb\A^3$ for~$\tau$, $\tau'$. We~call the set of characters labelling monomials in~$Z_\tau$ (or~in~$Z_{\tau'}$) that partake in~the $G$-igsaw transformation the \emph{total $G$\textit-igsaw piece} for~$C$ and denote it by~$\Gig(C)$.
\end{Definition}

We will often also refer to the set of monomials underlying $\Gig(C)$ for one of the triangles either side of $C$ as a~total $G$-igsaw piece. There is a~single monomial that divides all others in~the total $G$-igsaw piece, and this is the monomial in~the $G$-cluster in~the character space for the character marking $C$. Indeed, an alternative definition of a total $G$-igsaw piece for~$C$ is the set of all monomials in~the $G$-graph for one of the triangles adjacent to $C$ that are divisible by the monomial in~this eigenspace.

\begin{Example} We continue the example of $G=\frac{1}{6}(1,2,3)$. The Craw--Reid triangulation and Reid's recipe for this group is shown in~Fig.~\ref{fig:1/6a}. The triangle labelled by $\star$ is the triangle corresponding to the $G$-cluster from Fig.~\ref{fig:clust}.

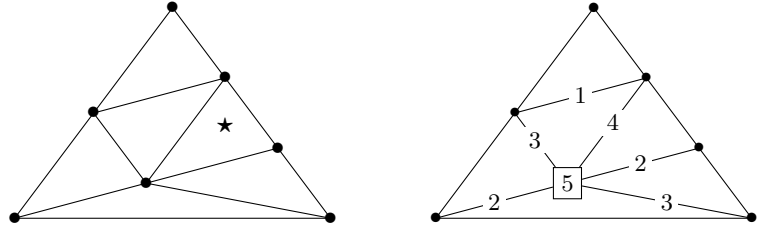
\begin{figure}[h]\centering
\begin{tikzpicture}[scale=0.7]
\node(e1) at (0,2){$\bullet$};
\node(e2) at (3,-2){$\bullet$};
\node(e3) at (-3,-2){$\bullet$};

\node(123) at (-0.5,-4/3){$\bullet$};
\node(240) at (2,2-8/3){$\bullet$};
\node(420) at (1,2-4/3){$\bullet$};
\node(303) at (-1.5,0){$\bullet$};

\node (star) at (1,-0.25){\Large $\star$};

\draw[-] (e1.center) to (e2.center) to (e3.center) to (e1.center);

\draw[-] (e2.center) to (123.center);
\draw[-] (e3.center) to (240.center);
\draw[-] (303.center) to (420.center);
\draw[-] (123.center) to (303.center);
\draw[-] (123.center) to (420.center);

\small

\node(e1) at (0+8,2){$\bullet$};
\node(e2) at (3+8,-2){$\bullet$};
\node(e3) at (-3+8,-2){$\bullet$};

\footnotesize

\node[draw,fill=white,inner sep=3pt] (123) at (-0.5+8,-4/3){$5$};
\node(240) at (2+8,2-8/3){$\bullet$};
\node(420) at (1+8,2-4/3){$\bullet$};
\node(303) at (-1.5+8,0){$\bullet$};

\draw[-] (e2.center) to node[fill=white,inner sep=2pt] {$3$} (123);
\draw[-] (e3.center) to node[fill=white,inner sep=2pt] {$2$} (123) to node[fill=white,inner sep=2pt] {$2$} (240.center);
\draw[-] (303.center) to node[fill=white,inner sep=2pt] {$1$} (420.center);
\draw[-] (123) to node[fill=white,inner sep=2pt] {$3$} (303.center);
\draw[-] (123) to node[fill=white,inner sep=2pt] {$4$} (420.center);

\draw[-] (e1.center) to (e2.center) to (e3.center) to (e1.center);
\end{tikzpicture}
\caption{Triangulation and Reid's recipe for~$\frac{1}{6}(1,2,3)$.}\label{fig:1/6a}
\end{figure}

Passing through the $4$-curve $C$ adjacent to the triangle $\star$ performs a~$G$-igsaw transformation with total $G$-igsaw piece centred on the monomial with character $4$, which in~this case is $xz$. The $G$-igsaw transformation switches $xz$ for~$y^2$~-- since the $G$-invariant ratio for~$C$ is $xz/y^2$~-- producing the new $G$-cluster
\begin{gather*}
\begin{array}{ccccccccccc}
y^2 \\
& y \\
yz & & 1 & & x \\
& z &
\end{array}
\end{gather*}
If we pass through the $2$-curve bordering $\star$ then the total $G$-igsaw piece contains the monomials $y$, $yz$ and produces the $G$-cluster
\begin{gather*}
\begin{array}{ccccccccccc}
& 1 & & x & & x^2\\
 z & & xz & & x^2z
\end{array}
\end{gather*}
\end{Example}

As the two total $G$-igsaw pieces for a~given curve are related by multiplying by $G$-invariant ratios it is clear that they each have the same set of characters represented by their monomials. We~denote the set of characters in~either total $G$-igsaw piece for a~curve $C$ by $\Gig(C)$. For convenience we will also denote by $\chi(m)$ the character by which $G$ acts on a~monomial $m$.

\subsection{Tautological bundles} \label{sec:taut}

The sheaf $\mR=\pi_*\mO_{\mathcal{Z}}$ is locally free with fibre $H^0(\mO_Z)$ above $Z\in G\hilb\A^3$. It~splits into eigensheaves
\begin{gather*}
\mR=\bigoplus_{\chi\in\on{Irr}{G}}\mR_\chi
\end{gather*}
and these summands are called \textit{tautological bundles}. Since $G$ is abelian, the $\mR_\chi$ are line bundles.~\cite{Craw05} gives relations between these line bundles in~the Picard group, which translate to divisibility relations between eigenmonomials. For a~triangle $\tau$ in~the Craw--Reid triangulation, denote the generator of $\mR_\chi$ on the affine piece corresponding to $\tau$ by $r_{\chi,\tau}$. We~usually omit reference to $\tau$ so long as the context is clear

\begin{Theorem}[{\cite[Theorem~6.1]{Craw05}}] \label{thm:craw} The relations between $($generators of$)$ tautological line bundles are described by Reid's recipe in~the following way.
\begin{itemize}\itemsep=0pt
\item If three lines marked with the same character $\chi$ meet at a~vertex marked with $\psi=\chi^{\otimes 2}$ then
\begin{gather*}
r_\chi^2=r_\psi.
\end{gather*}
\item If four or five or six lines consisting of two pairs marked by characters $\chi$, $\chi'$ and zero or one or two extra lines marked with further characters meet at a~vertex marked with $\psi=\chi\otimes\chi'$ then
\begin{gather*}
r_\chi\cdot r_{\chi'}=r_\psi.
\end{gather*}
\item If six lines consisting of three pairs marked by characters $\chi$, $\chi'$, $\chi''$ meet at a~vertex marked with $\phi$, $\phi'$ then
\begin{gather*}
r_\chi\cdot r_{\chi'}\cdot r_{\chi''}=r_\phi\cdot r_{\phi'}.
\end{gather*}
\end{itemize}
\end{Theorem}

The claim is that these relations hold and generate all relations between tautological bundles. We~will make heavy use of these divisibility relations between eigenmonomials to study $G$-igsaw pieces for exceptional curves.

As alluded to, the work of Logvinenko~\cite{Logvinenko08}, Cautis--Logvinenko~\cite{CautisLogvinenko09}, and Craw--Cautis--Log\-vi\-nenko~\cite{CautisCrawLogvinenko17} categorifies Reid's recipe via the tautological bundles. Many of the constructions in~\cite[Sections~3--4]{CautisCrawLogvinenko17} resemble constructions made in~Section~\ref{sec:comp} below, however the computations they make are for the inverse equivalence of~(\ref{eqn:BKR}) to that utilised in~\cite{CrawIshii04} and here.

Evident from~\cite{CautisCrawLogvinenko17,CrawIshii04} and below, characters marking a~divisor or a~single curve are special. They are termed ``essential characters'' and have been further examined in~\cite{CrawItoKarmazyn18,Takahashi}.

\subsection[Abstract inequalities for C_0]{Abstract inequalities for~$\boldsymbol{\mfk{C}_0}$}

In \cite[Section~9]{CrawIshii04} Craw--Ishii provide an abstract description of sufficiently many inequalities to carve out the chamber $\mfk{C}_0$. These inequalities arise from the linearisation map
\begin{gather*}
L_{\mfk{C}_0}\colon\ \Theta\to\on{Pic}(\M_{\mfk{C}_0}),
\end{gather*}
which takes $\theta$ to the ample $\Q$-divisor on~$\M_\theta$ arising from GIT; note that we have identified the Picard groups of different resolutions. By construction $L_{\mfk{C}_0}(\mfk{C}_0)\subset\on{Amp}(\M_{\C_0})$ and so one obtains inequalities
\begin{gather*}
\theta\big(\ph_{\mfk{C}_0}(\mO_C)\big)>0
\end{gather*}
for all exceptional curves $C\subset G$-Hilb. If such an inequality is necessary to define $\mfk{C}_0$, the geometry of $C$ determines the type of the corresponding wall as follows:
\begin{itemize}\itemsep=0pt
\item If $C$ is a~$(-1,-1)$-curve~-- that is, it corresponds to an interior edge inside a~regular triangle~-- then $(\theta(\ph_{\mfk{C}_0}(\mO_C))=0)\cap\ol{\mfk{C}}_0$ is a~Type \texttt{I} wall.
\item If $C$ is a~$(1,-3)$-curve~-- that is, it corresponds to one of the edges incident to a~trivalent vertex~-- then $(\theta(\ph_{\mfk{C}_0}(\mO_C))=0)\cap\ol{\mfk{C}}_0$ is a~Type \texttt{II} wall.
\item If $C$ is contained in~a~Hirzebruch divisor but it is not in~either of the previous cases, then $(\theta(\ph_{\mfk{C}_0}(\mO_C))=0)\cap\ol{\mfk{C}}_0$ is a~Type \texttt{III} wall.
\end{itemize}
One can express the inequality $\theta(\ph_\mfk{C}(\mO_C))>0$ abstractly via \cite[Corollary~5.2]{CrawIshii04}, a~consequence of which is
\begin{gather*}
\theta\big(\ph_{\mfk{C}}(\mO_C)\big)=\sum_\rho\deg (\mR_\rho|_C)\theta(\rho).
\end{gather*}
Any character $\rho$ not in~$\Gig(C)$ has $\mR_\rho|_C=\mO_C$ and so it doesn't appear in~the sum above. It~follows that
\begin{gather*}
\theta\big(\ph_\mfk{C}(\mO_C)\big)=\sum_{\rho\in G\text{-ig}(C)}\deg (\mR_\rho|_C)\theta(\rho).
\end{gather*}

To complete the description by Craw--Ishii, their remaining inequalities~-- which if they are necessary inequalities will define walls of Type \texttt{0}~-- are obtained from divisors in~two ways outlined in~Lemma~\ref{lem:ci5} and Proposition~\ref{prop:rigid}.

\begin{Lemma}[{\cite[Corollary~5.6 and Theorem~9.3]{CrawIshii04}}] \label{lem:ci5} Suppose that $D\subset G\hilb\A^3$ is an irreducible exceptional divisor marked with a~character $\psi$. Then all $\theta\in\mfk{C}_0$ satisfy the inequality
\begin{gather*}
\theta\big(\ph_{\mfk{C}_0}\big(\mR_\psi^{-1}|_D\big)\big)=\theta(\psi)>0.
\end{gather*}
Moreover, the inequalities of this form are necessary and hence define walls of $\mfk{C}_0$.
\end{Lemma}

We call the inequalities coming from Lemma~\ref{lem:ci5} \emph{subsheaf inequalities}. The reason for this is the following. Suppose that $\mO_Z$ is a~$G$-cluster parameterised by a~point in~an irreducible exceptional divisor $D\subset G$-Hilb. Let $\psi$ be a~character marking $D$. It~follows from \cite[Corollary~4.6 and Lemma~9.1]{CrawIshii04} that $r_\psi$ is in~the socle of $H^0(\mO_Z)$ and is constant on~$D$, hence defines a~subsheaf $S=\mO_0\otimes\psi$ of $\mO_Z$ for all $Z\in D$. By definition we must have $\theta(S)=\theta(\psi)>0$ for all $\theta\in\mfk{C}_0$ and \cite[Proposition~9.3]{CrawIshii04} shows that this is indeed a~wall of $\mfk{C}_0$.

There is a~dual version of this using quotients instead of subsheaves. \cite[Lemma~5.7 and Theorem~9.5]{CrawIshii04} shows that
\begin{gather*}
\theta\big(\ph_{\mfk{C}_0}(\omega_{D'})\big)<0
\end{gather*}
for all $\theta\in\mfk{C}_0$ and for every possibly reducible but connected exceptional divisor $D'$. From \cite[Theorem~9.5]{CrawIshii04} this inequality corresponds to evaluating $\theta$ on the minimal rigid quotient $Q$ of $\mO_Z$ for all $Z\in D'$.

\begin{Proposition} \label{prop:rigid} Suppose $D'\subset G\hilb\A^3$ is a~possibly reducible but connected exceptional divisor. Then all $\theta\in\mfk{C}_0$ satisfy the inequality
\begin{gather*}
\sum_{C\subset D'}\sum_{\chi\in\Gig(C)}\theta(\chi)>0.
\end{gather*}
\end{Proposition}

We call the inequalities from Proposition~\ref{prop:rigid} \emph{quotient inequalities}.

\begin{proof} For the sheaf $Q$ to be trivial on~$D'$ it means that $\mR_\rho|_{D'}$ is trivial for all $\rho\subset Q$. Equi\-va\-lently, all torus-invariant $G$-clusters in~$D'$ share the same eigenmonomial $r_\rho$ for each $\rho\subset Q$ or, also equivalently, $\rho\notin\Gig(C)$ for any $C\subset D'$. Reversing the inequality $\theta(Q)<0$ gives $\theta(\C[G]/Q)>0$ and $\C[G]/Q$ contains exactly the characters in~the statement of the result since~$G$ is abelian and so all irreducible representations have multiplicity $1$ in~$\C[G]$.
\end{proof}

The value of Proposition~\ref{prop:rigid} is that it reduces computing inequalities from divisors to computing various total $G$-igsaw pieces, which is the topic of the next section.

\section[Computing characters in~total G-igsaw pieces]{Computing characters in~total $\boldsymbol{G}$-igsaw pieces}\label{sec:comp}

Our motivating question for this section is the following: \textit{how can one determine the characters that appear in~a~total $G$-igsaw piece for a~curve purely from the data of Reid's recipe?} As we shall see, the answer depends somewhat on how $C$ sits inside $G$-Hilb though it is still completely combinatorial.

\subsection{Combinatorial definitions} \label{sec:comb_def}

We start by making some combinatorial definitions. For two points $u,v\in\R^3$ we denote by $[u,v]$, $(u,v)$, $[u,v)$, $(u,v]$ the closed, open, and half-open line segments with endpoints $u$ and $v$. For~example, $u\in [u,v)$ but $v\notin [u,v)$.

Let $C$ be a~$(-1,-1)$-curve marked with $\chi$. Let $\alpha$ be the edge corresponding to $C$ in~the Craw--Reid triangulation. Pick a~point $u_0$ in~the interior of $\alpha$. We~will use $u_0$ to view the $\chi$-chain as an~(embedded) quiver $\Xi_{\chi,u_0}$ as follows. Let the vertices of the quiver be the vertices corresponding to Hirzebruch divisors incident to the $\chi$-chain (including those at the ends if applicable) and include one extra vertex corresponding to $u_0$. Let the edges be the parts of the $\chi$-chain between these divisors with the edge containing $\alpha$ split into two with one on either side of $u_0$. We~orient an edge $\beta\subset\Xi_{\chi,u_0}$ by declaring that the tail is the boundary vertex of $\beta$ closest to $u_0$. Note that the $\chi$-chain is a~tree and so this makes sense. We~show an example of $\Xi_{\chi,u_0}$ for the group $G=\frac{1}{30}(25,2,3)$ with the character $\chi_3$ and the point $\ast$ in~Fig.~\ref{fig:graph}. Note that as an abstract quiver~$\Xi_{\chi,u_0}$ depends only on the edge $\alpha$ or, equivalently, the curve $C$ and so we slightly abuse notation by subsequently denoting it $\Xi_C$.

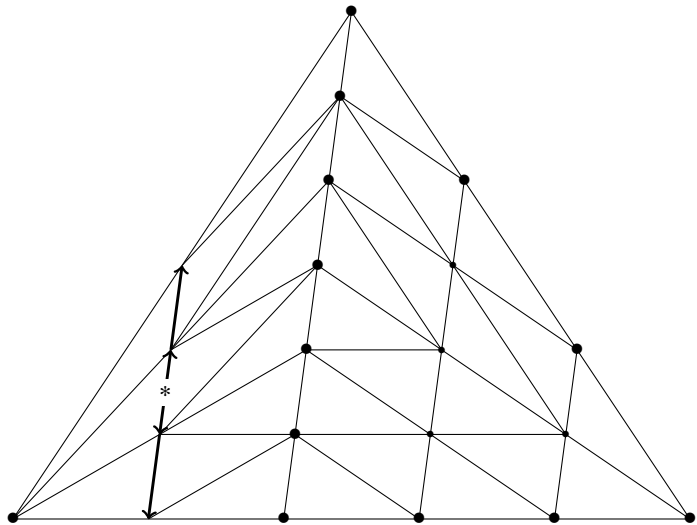
\begin{figure}[h]\centering
\begin{tikzpicture}[scale=0.45]
\node (e1) at (0,9){$\bullet$};
\node (e2) at (10,-6){$\bullet$};
\node (e3) at (-10,-6){$\bullet$};
\node (a1) at (-2/6,9-5/2){$\bullet$};
\node (a2) at (-4/6,9-10/2){$\bullet$};
\node (a3) at (-6/6,9-15/2){$\bullet$};
\node (a4) at (-8/6,9-20/2){$\bullet$};
\node (a5) at (-10/6,9-25/2){$\bullet$};

\draw(e3.center) to (a4.center);
\draw(e2.center) to (a3.center);

\node (xy1) at (10/3,4){$\bullet$};
\node (xy2) at (20/3,-1){$\bullet$};

\node (xz1) at (-5,1.5){};

\node (yz1) at (-10+4,-6){};
\node (yz2) at (-10+8,-6){$\bullet$};
\node (yz3) at (-10+12,-6){$\bullet$};
\node (yz4) at (-10+16,-6){$\bullet$};

\draw (e1.center) to (yz2.center);

\node (iz1) at (-16/3,-1){};
\node (iz2) at (-17/3,-3.5){};
\node[fill=white] (iz15) at (-33/6,-2.25){};

\draw[<-,line width=1.1] (iz1.center) to (iz15.center);
\draw[<-,line width=1.1] (iz2.center) to (iz15.center);

\tiny

\draw (e3.center) to (a2.center);
\draw[<-,line width=1.1] (xz1.center) to (iz1.center);
\draw[<-,line width=1.1] (yz1.center) to (iz2.center);
\draw (xz1.center) to (a1.center);
\draw (iz1.center) to (a1.center);

\draw (iz1.center) to (a3.center);
\draw (iz2.center) to (a3.center);

\draw (iz2.center) to (a5.center);
\draw (yz1.center) to (a5.center);

\node (c1) at (-1+11/3,-1){$\bullet$};
\node (c2) at (-1+22/3,-3.5){$\bullet$};

\node (dp1) at (9/3,1.5){$\bullet$};
\node (dp2) at (7/3,-3.5){$\bullet$};

\draw (xy2.center) to (yz4.center);
\draw (xy1.center) to (yz3.center);
\draw (a5.center) to (c2.center);
\draw (a5.center) to (yz3.center);
\draw (a4.center) to (yz4.center);
\draw (a4.center) to (c1.center);
\draw (a2.center) to (xy2.center);
\draw (a1.center) to (xy1.center);
\draw (a1.center) to (c2.center);
\draw (a2.center) to (c1.center);

\draw (e1.center) to (e2.center);
\draw (yz2.center) to (e2.center);
\draw (e3.center) to (yz2.center);
\draw (e3.center) to (e1.center);
\node[fill=white] (iz15) at (-33/6,-2.25){\large $\ast$};
\end{tikzpicture}
\caption{The embedded quiver $\Xi_{\chi_3,\ast}=\Xi_{C_1}$.}\label{fig:graph}
\end{figure}

\begin{Definition} Let $C$ be a~$\chi$-curve. We~say that a~divisor $D$ along the $\chi$-chain is \emph{downstream} of $C$ if it is a~Hirzebruch divisor. We~say that a~$(-1,-1)$-curve $E$ incident to such a~divi\-sor~$D$ is \emph{downstream} of $C$ if the edge for~$E$ meets the tail of the arrow in~$\Xi_C$ corresponding to the $\chi$-curve in~the same regular triangle as $E$.
\end{Definition}

This is illustrated schematically in~Fig.~\ref{fig:schematic} and concretely on the left side of Fig.~\ref{fig:ahead} for the $3$-curve $C_1$ inside $\frac{1}{30}(25,2,3)$-Hilb whose edge includes the point $\ast$ from Fig.~\ref{fig:graph}. In~Fig.~\ref{fig:schematic} the bold edges indicate the boundary of a regular triangle, hence the central vertex is a~Hirzebruch divisor, and only the two dotted curves are downstream according to the given orientation of the $\chi$-chain. In~Fig.~\ref{fig:ahead} the curves downstream of $C_1$ are dotted.

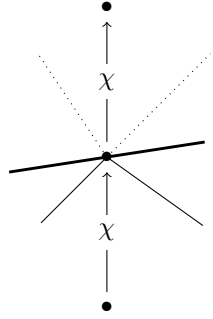
\begin{figure}[h]\centering
\begin{tikzpicture}
\small
\node (a) at (0,2){$\bullet$};
\node (b) at (0,0){$\bullet$};
\node (c) at (0,-2){$\bullet$};
\node (d) at (1.5,1.5){};
\node (e) at (-1,1.5){};
\node (f) at (-1.3,-0.2){};
\node (g) at (1.3,0.2){};
\node (h) at (-1,-1){};
\node (i) at (1.4,-1){};

\draw[<-] (a) to node[fill=white] {$\chi$} (b);
\draw[<-] (b) to node[fill=white] {$\chi$} (c);
\draw[dotted] (b.center) to (d);
\draw[dotted] (b.center) to (e);
\draw (b.center) to (h);
\draw (b.center) to (i);

\draw[line width=1.1pt] (b.center) to (f.center);
\draw[line width=1.1pt] (b.center) to (g.center);
\end{tikzpicture}
\caption{Schematic of curves downstream from a~$(-1,-1)$-curve.}\label{fig:schematic}
\end{figure}

Now suppose that $C$ is a~boundary curve marked with $\chi$. By the construction of the Craw--Reid triangulation there is a~vertex $e_i$ of the junior simplex and an interior vertex $v$ of the junior simplex such that the edge for~$C$ is contained in~the line segment $[e_i,v]$. Let $v_C$ be the vertex furthest from $e_i$ such that this is true. At $v_C$ the $\chi$-chain will change slope. There exists a~vertex~$v_C'$ such that all curves in~the $\chi$-chain between $v_C$ and $v'$ are $(-1,-1)$-curves and hence lie in~the interior of various regular triangles. Note that $v_C=v_C'$ is possible, in~which case there are no $(-1,-1)$-curves along the $\chi$-chain. At $v'_C$ there are two possibilities for the $\chi$-chain: either the $\chi$-chain terminates, or it continues into a~line segment $[v_C',e_j]$ from $v_C'$ to a~vertex $e_j$ of the junior simplex where it then terminates. We~illustrate the situation where the $\chi$-chain terminates at~$e_j$ in~Fig.~\ref{fig:schematic2}. The dashed segment of the $\chi$-chain represents the part between $v_C$ and~$v_C'$, which consists of $(-1,-1)$-curves, and the vertices shown there correspond to Hirzebruch divisors, which occur where the $\chi$-chain passes between two different regular triangles.

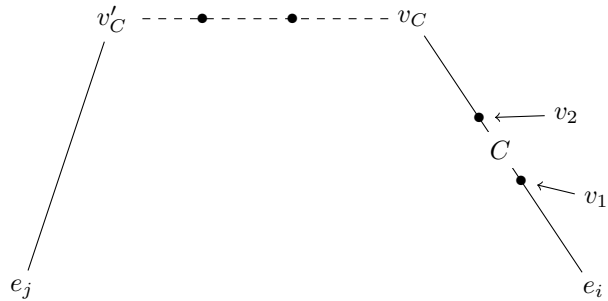
\begin{figure}[h]\centering
\begin{tikzpicture}[scale=0.4]
\small
\node (ei) at (4,-9){$e_i$};
\node (a) at (2-2/5,-6+3/5){$\bullet$};
\node (b) at (0+2/10,-3-3/10){$\bullet$};
\node (c) at (-2,0){$v_C$};
\node (c1) at (-6,0){$\bullet$};
\node (c2) at (-9,0){$\bullet$};
\node (d) at (-12,0){$v'_C$};
\node (ej) at (-15,-9){$e_j$};

\node (l1) at (4.1,-6){$v_1$};
\node (l2) at (3.1,-3.2){$v_2$};

\draw (ei) to (a.center);
\draw (a.center) to node[fill=white] {$C$} (b.center);
\draw (b.center) to (c);

\draw[dashed] (c) to (c1.center) to (c2.center) to (d);

\draw (d) to (ej);

\draw[->] (l1) to (a);
\draw[->] (l2) to (b);
\end{tikzpicture}
\caption{$\chi$-chain for a~boundary curve.}\label{fig:schematic2}
\end{figure}

Similarly to the case where $C$ was a~$(-1,-1)$-curve, we create an embedded quiver $\Xi_C$ supported on the part of the $\chi$-chain between the vertices $v_2$ and $v_C'$. The vertices are the Hirzebruch divisors incident to this part of the $\chi$-chain and the edges are the parts of the $\chi$-chain between these divisors. We~orient the edges by declaring that $v_2$ is the unique source of the quiver and $v_C'$ is the unique sink.

We use the notation of Fig.~\ref{fig:schematic2} in~the following definitions.

\begin{Definition} \label{def:bdry_downstream} Let $C$ be a~boundary curve as depicted in~Fig.~\ref{fig:schematic2}. We~say that a~Hirzebruch divisor $D$ is \emph{downstream} of $C$ if either:
\begin{itemize}\itemsep=0pt
\item the vertex for~$D$ lies in~$[v_2,v_C]$,
\item the vertex for~$D$ lies in~the part of the $\chi$-chain between $v_C$ and $v'_C$ (excluding $v'_C$).
\end{itemize}
Let $E$ be a~$(-1,-1)$-curve marked with a~character $\rho$ and contained in~a~Hirzebruch divisor $D$ downstream of $C$. We~say that $E$ is \emph{downstream} of $C$ if either:
\begin{itemize}\itemsep=0pt
\item the vertex for~$D$ lies in~the line segment $[v_2,v_C]$,
\item the vertex for~$D$ lies between $v_C$ and $v_C'$ and the edge for~$E$ meets the tail of the arrow in~$\Xi_C$ corresponding to the $\chi$-curve in~the same regular triangle as $E$,
\end{itemize}
and if either
\begin{itemize}\itemsep=0pt
\item the $\rho$-chain terminates at~$D$,
\item the edges in~the $\rho$-chain incident to $D$ have different slopes.
\end{itemize}
\end{Definition}

We show a~schematic for the downstream curves relative to $C$ in~Fig.~\ref{fig:schematic3}. The bold arrows represent $\Xi_C$ and the additional edges correspond to sides of the various regular triangles that the $\chi$-chain passes through. The dotted edges are the curves downstream of $C$, and the dashed edges represent two $(-1,-1)$-curves marked with the same character and whose edges have the same slope, hence the dashed curves are not downstream of $C$.

\begin{figure}[h]\centering
\begin{tikzpicture}[scale=0.4]
\small
\node (ei) at (4,-8){$e_i$};
\node (a) at (2-2/5,-6+3/5){$\bullet$};
\node (b) at (0+2/10,-3-3/10){$\bullet$};
\node (c) at (-2,0){$v_C$};
\node (c1) at (-6,0){$\bullet$};
\node (c2) at (-9,0){$\bullet$};
\node (d) at (-12,0){$v'_C$};
\node (ej) at (-15,-8){$e_j$};

\node (l1) at (4.1,-6){$v_1$};
\node (l2) at (3.1,-3.8){$v_2$};

\draw (a.center) to node[fill=white,inner sep=2pt] {$C$} (b.center);
\draw[line width=1.1pt,->] (b) to (c);

\draw[line width=1.1pt,->] (c) to (c1);
\draw[line width=1.1pt,->] (c1) to (c2);
\draw[line width=1.1pt,->] (c2) to (d);

\draw[->] (l1) to (a);
\draw[->] (l2) to (b);

\node (f) at (-6.5,2){};
\node (g) at (-5,-4){};

\draw (f) to (g);

\node (h) at (-8,2){};
\node (i) at (-11,-4){};

\draw (h) to (i);

\node (k) at (0+2/10-2.5,-3-3/10-0.5){};
\node (k') at (0+2/10+2.5,-3-3/10+0.5){};
\node (m) at (0+2/10-2.5,-3-3/10+1){};
\node (m') at (1.5,-3+8/10+1){};

\draw[dashed] (b.center) to (k.center);
\draw[dashed] (b.center) to (k'.center);
\draw[dotted] (b.center) to (m.center);
\draw[dotted] (b.center) to (m'.center);

\node (n) at (-2-2.5,-1){};
\node (o) at (-6-2,-1.75){};
\node (p) at (-9-2,-1.5){};

\draw[dotted] (c) to (n.center);
\draw[dotted] (c1.center) to (o.center);
\draw[dotted] (c2.center) to (p.center);
\end{tikzpicture}
\caption{Schematic of curves downstream of a boundary curve.}\label{fig:schematic3}
\end{figure}
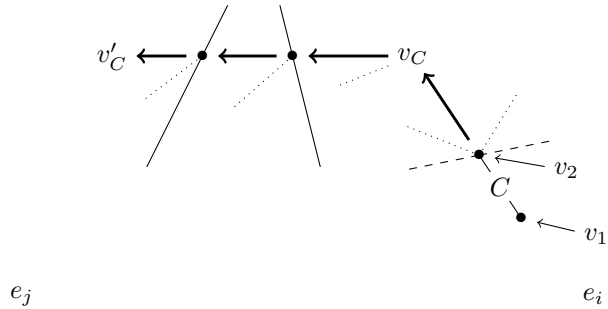

Note that in~Fig.~\ref{fig:schematic2} the divisors for~$v_2$ and $v_C$ are downstream of $C$ but the divisor for~$v_1$ is not. On the right of Fig.~\ref{fig:ahead}, when $G=\frac{1}{30}(25,2,3)$ the divisors $D_2$ and $D_3$ are all the divisors downstream of the $15$-curve $C_2$ whereas $D_1$ is the only divisor along the $15$-chain that is not. We~have bolded the sides of regular triangles in~the triangulation to make it clear that $C_2$ is a~boundary curve, and to clarify which divisors are Hirzebruch divisors. We~also show the curves downstream of $C_2$ dotted.

\begin{figure}[h]\centering
\begin{tikzpicture}[scale=0.36]
\node (e1) at (0,9){$\bullet$};
\node (e2) at (10,-6){$\bullet$};
\node (e3) at (-10,-6){$\bullet$};
\node (a1) at (-2/6,9-5/2){$\bullet$};
\node (a2) at (-4/6,9-10/2){$\bullet$};
\node (a3) at (-6/6,9-15/2){$\bullet$};
\node (a4) at (-8/6,9-20/2){$\bullet$};
\node (a5) at (-10/6,9-25/2){$\bullet$};

\draw[line width=1.1pt](e3.center) to (a4.center);
\draw(e2.center) to (a3.center);

\node (xy1) at (10/3,4){$\bullet$};
\node (xy2) at (20/3,-1){$\bullet$};

\node (xz1) at (-5,1.5){$\bullet$};

\node (yz1) at (-10+4,-6){$\bullet$};
\node (yz2) at (-10+8,-6){$\bullet$};
\node (yz3) at (-10+12,-6){$\bullet$};
\node (yz4) at (-10+16,-6){$\bullet$};

\draw[line width=1.1pt] (e1.center) to (yz2.center);

\node (iz1) at (-16/3,-1){$\bullet$};
\node (iz2) at (-17/3,-3.5){$\bullet$};

\small

\draw (iz1.center) to node[fill=white,inner sep=2pt]{$C_1$} (iz2.center);
\draw[line width=1.1pt](e3.center) to (a2.center);
\draw (xz1.center) to (iz1.center);
\draw (yz1.center) to (iz2.center);
\draw (xz1.center) to (a1.center);
\draw[dotted] (iz1.center) to (a1.center);

\draw (iz1.center) to (a3.center);
\draw (iz2.center) to (a3.center);

\draw[dotted] (iz2.center) to (a5.center);
\draw (yz1.center) to (a5.center);

\node (c1) at (-1+11/3,-1){$\bullet$};
\node (c2) at (-1+22/3,-3.5){$\bullet$};

\node (dp1) at (9/3,1.5){$\bullet$};
\node (dp2) at (7/3,-3.5){$\bullet$};

\draw (xy2.center) to (yz4.center);
\draw (xy1.center) to (yz3.center);
\draw (a5.center) to (c2.center);
\draw (a5.center) to (yz3.center);
\draw (a4.center) to (yz4.center);
\draw (a4.center) to (c1.center);
\draw (a2.center) to (xy2.center);
\draw (a1.center) to (xy1.center);
\draw (a1.center) to (c2.center);
\draw (a2.center) to (c1.center);

\draw (e1.center) to (e2.center);
\draw (yz2.center) to (e2.center);
\draw[line width=1.1pt] (e3.center) to (yz2.center);
\draw[line width=1.1pt] (e3.center) to (e1.center);


\node (e1) at (0+22,9){$\bullet$};
\node (e2) at (10+22,-6){$\bullet$};
\node (e3) at (-10+22,-6){$\bullet$};
\node (a1) at (-2/6+22,9-5/2){$\bullet$};
\node (a2) at (-4/6+22,9-10/2){$\bullet$};
\node[draw,fill=white] (a3) at (-6/6+22,9-15/2){};
\node (a4) at (-8/6+22,9-20/2){$\bullet$};
\node (a5) at (-10/6+22,9-25/2){$\bullet$};

\node[draw,fill=white] (c1) at (-1+11/3+22,-1){};
\node[draw,fill=white] (c2) at (-1+22/3+22,-3.5){};

\draw(e3.center) to (a2.center);
\draw(e3.center) to (a4.center);
\draw[line width=1.1pt] (e2.center) to (c2.center);
\draw[line width=1.1pt] (c1.center) to node[fill=white,inner sep=0pt] {$C_2$} (c2.center);
\draw[line width=1.1pt] (c1.center) to (a3.center);

\node (xy1) at (10/3+22,4){$\bullet$};
\node (xy2) at (20/3+22,-1){$\bullet$};

\node (xz1) at (-5+22,1.5){$\bullet$};

\node (yz1) at (-10+4+22,-6){$\bullet$};
\node (yz2) at (-10+8+22,-6){$\bullet$};
\node (yz3) at (-10+12+22,-6){$\bullet$};
\node (yz4) at (-10+16+22,-6){$\bullet$};

\draw[line width=1.1pt] (e1.center) to (yz2.center);

\node (iz1) at (-16/3+22,-1){$\bullet$};
\node (iz2) at (-17/3+22,-3.5){$\bullet$};

\draw (xz1.center) to (yz1.center);
\draw (xz1.center) to (a1.center);
\draw (iz1.center) to (a1.center);

\draw (iz1.center) to (a3.center);
\draw[dotted] (iz2.center) to (a3.center);

\draw (iz2.center) to (a5.center);
\draw (yz1.center) to (a5.center);

\node (dp1) at (9/3+22,1.5){$\bullet$};
\node (dp2) at (7/3+22,-3.5){$\bullet$};

\draw (xy2.center) to (yz4.center);
\draw (xy1.center) to (yz3.center);
\draw (a5.center) to (c2.center);
\draw (a5.center) to (yz3.center);
\draw (a4.center) to (yz4.center);
\draw[dotted] (a4.center) to (c1.center);
\draw (a2.center) to (xy2.center);
\draw (a1.center) to (xy1.center);
\draw (a1.center) to (c2.center);
\draw[dotted] (a2.center) to (c1.center);

\draw[line width=1.1pt] (e1.center) to (e2.center);
\draw[line width=1.1pt] (e2.center) to (yz2.center);
\draw (e3.center) to (yz2.center);
\draw (e3.center) to (e1.center);

\node[draw,fill=white] (a3) at (-6/6+22,9-15/2){$D_3$};
\node[draw,fill=white] (c1) at (-1+11/3+22,-1){$D_2$};
\node[draw,fill=white] (c2) at (-1+22/3+22,-3.5){$D_1$};
\end{tikzpicture}
\caption{Downstream curves and divisors in~$\frac{1}{30}(25,2,3)$-Hilb.}\label{fig:ahead}
\end{figure}
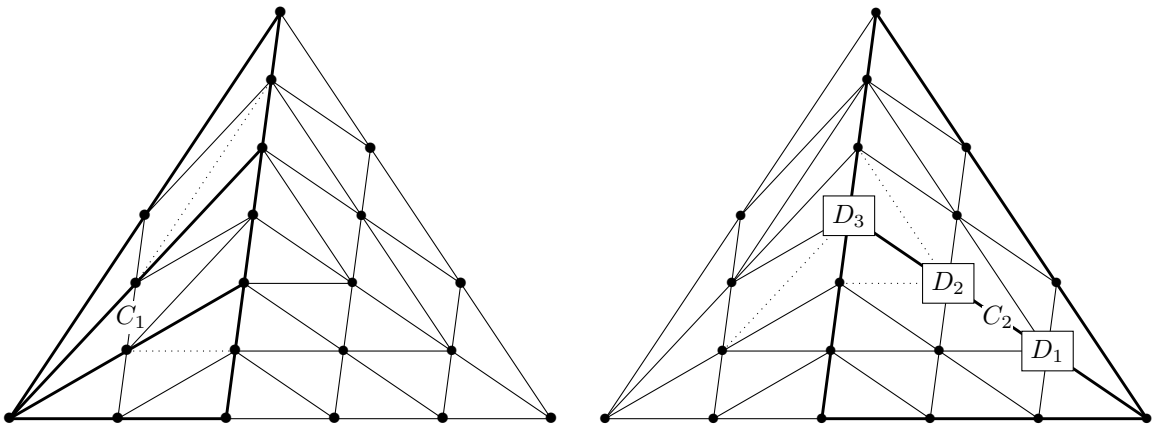

For any exceptional curve $C$ and a~Hirzebruch divisor $D$ downstream of $C$, we denote the set of curves in~$D$ downstream of $C$ by $\mathcal{C}_C(D)$.

We will say that the $\rho$-chain for some character $\rho$ is \emph{broken} at a~vertex $v$ (or~the corresponding divisor) if either the $\rho$-chain terminates at~$v$ or if the edges in~the $\rho$-chain incident to $v$ have different slopes as in~the second part of Definition \ref{def:bdry_downstream}.

Lastly, we define a~character $\chi_\text{dP}(C,D)$ associated to a~$\chi$-curve $C$ and a~del Pezzo divisor $D$ along the $\chi$-chain. Let $\Delta$ be the regular triangle containing the vertex for~$D$, let $v$ be the vertex corresponding to $D$, and let $\alpha$ be the edge of the triangulation corresponding to $C$.

Let $\{p,q,m\}=\{1,2,3\}$. We~denote $x_1=x$, $x_2=y$, $x_3=z$ for convenience. We~assume that $\Delta$ is a~corner triangle with $e_m$ as vertex and one side coming from a~ray out of $e_p$; we will treat the meeting of champions case shortly. Here $\phi_1,\phi_2$ denote the characters marking the del Pezzo divisor at~$v$, and $a$, $b$, $c$, $d$, $e$, $f$ are positive integers coming from the edges in~the Craw--Reid triangulation defining out $\Delta$. More precisely, the two sides incident to $e_m$ have the ratios $x_p^d:x_q^b$ and $x_q^e:x_p^a$ marking them, and the side coming from a~ray out of $e_p$ has ratio $x_m^f:x_q^c$. We~denote by $r=f$ the side length of the regular triangle. Each of the indices $i$, $j$, $k$ ranges from $0,\dots,r$. Consider the local picture for~$p=1$, $q=2$, $m=3$ shown in~Fig.~\ref{fig:gens} adapted from the proof of \cite[Theorem~6.1]{Craw05}, specifically \cite[Fig.~12]{Craw05}, for eigenmonomials near $v$ inside $\Delta$.

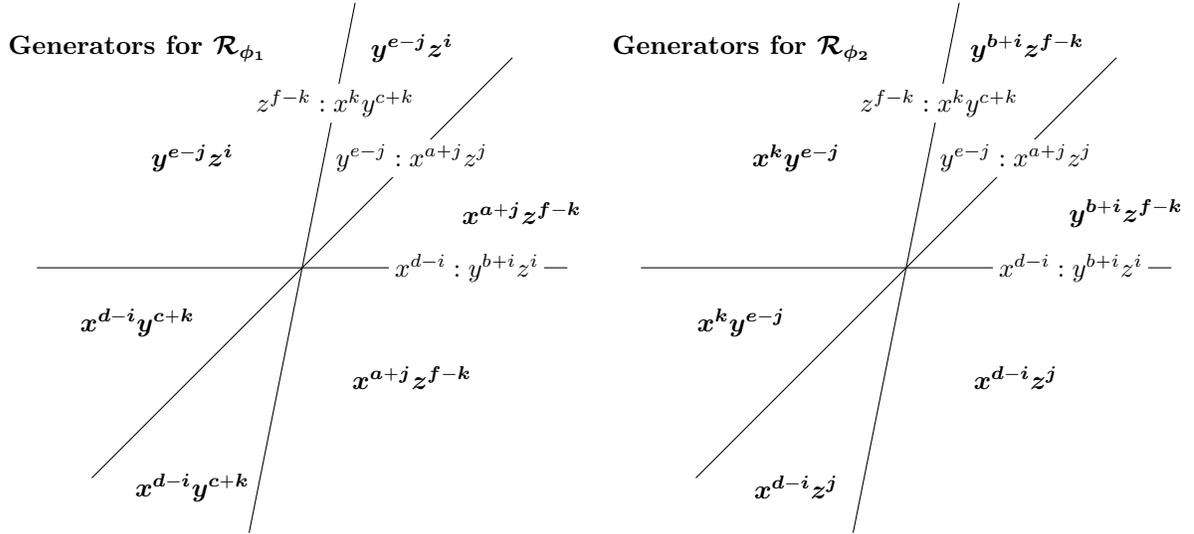
\begin{figure}[h]\centering
\begin{tikzpicture}[scale=0.73]
\small
\node(t) at (-3,4){\textbf{Generators for }$\boldsymbol{\mathcal{R}_{\phi_1}}$};

\node(o) at (0,0){};

\node(a) at (4,4){};
\node(b) at (-5,0){};
\node(c) at (-1,-5){};

\node(d) at (-4,-4){};
\node(e) at (1,5){};
\node(f) at (5,0){};

\draw (o.center) to (a);
\draw (o.center) to (b);
\draw (o.center) to (c);
\draw (o.center) to (d);
\draw (o.center) to (e);
\draw (o.center) to (f);

\node[rectangle, fill=white,inner sep=2pt](l1) at (3,0){$x^{d-i}:y^{b+i}z^i$};
\node[rectangle, fill=white,inner sep=2pt](l2) at (2,2){$y^{e-j}:x^{a+j}z^j$};
\node[rectangle, fill=white,inner sep=2pt](l3) at (0.6,3){$z^{f-k}:x^ky^{c+k}$};

\node(rl1) at (-2,2){$\boldsymbol{y^{e-j}z^i}$};
\node(rl2) at (2,4){$\boldsymbol{y^{e-j}z^i}$};
\node(rl3) at (4,1){$\boldsymbol{x^{a+j}z^{f-k}}$};
\node(rl4) at (2,-2){$\boldsymbol{x^{a+j}z^{f-k}}$};
\node(rl5) at (-2,-4){$\boldsymbol{x^{d-i}y^{c+k}}$};
\node(rl6) at (-3,-1){$\boldsymbol{x^{d-i}y^{c+k}}$};

\node(t) at (8,4){\textbf{Generators for }$\boldsymbol{\mathcal{R}_{\phi_2}}$};

\node(o) at (11,0){};

\node(a) at (15,4){};
\node(b) at (6,0){};
\node(c) at (10,-5){};

\node(d) at (7,-4){};
\node(e) at (12,5){};
\node(f) at (16,0){};

\draw (o.center) to (a);
\draw (o.center) to (b);
\draw (o.center) to (c);
\draw (o.center) to (d);
\draw (o.center) to (e);
\draw (o.center) to (f);

\node[rectangle, fill=white,inner sep=2pt](l1) at (14,0){$x^{d-i}:y^{b+i}z^i$};
\node[rectangle, fill=white,inner sep=2pt](l2) at (13,2){$y^{e-j}:x^{a+j}z^j$};
\node[rectangle, fill=white,inner sep=2pt](l3) at (11.6,3){$z^{f-k}:x^ky^{c+k}$};

\node(rl1) at (9,2){$\boldsymbol{x^ky^{e-j}}$};
\node(rl2) at (13.2,4){$\boldsymbol{y^{b+i}z^{f-k}}$};
\node(rl3) at (15,1){$\boldsymbol{y^{b+i}z^{f-k}}$};
\node(rl4) at (13,-2){$\boldsymbol{x^{d-i}z^j}$};
\node(rl5) at (9,-4){$\boldsymbol{x^{d-i}z^j}$};
\node(rl6) at (8,-1){$\boldsymbol{x^ky^{e-j}}$};
\end{tikzpicture}
\caption{Generators for tautological bundles near $v$.}\label{fig:gens}
\end{figure}

Suppose $\chi=\chi\big(x_p^{d-i}\big)$; that is, if $p=1$, $q=2$, $m=3$ then $\chi$ marks the horizontal chain of curves in~Fig.~\ref{fig:gens}. Following~\cite{Craw05} we denote by $e_ive_j$ the convex part of the junior simplex enclosed by the line segments from $e_i$ to $v$ and from $e_j$ to $v$. Define
\begin{gather*}
\chi_\text{dP}(C,D):=\begin{cases}
\chi\otimes\chi\big(x_m^j\big) & \text{if $\alpha\subset e_pve_q$}, \\
\chi\otimes\chi\big(x_q^{c+k}\big) & \text{if $\alpha\subset e_pve_m$}.
\end{cases}
\end{gather*}
Observe that $\alpha\subset e_qve_m$ is not a~possibility from considering slopes. Similarly, if $\chi=\chi\big(x_q^{e-j}\big)$ define
\begin{gather*}
\chi_\text{dP}(C,D):=\begin{cases}
\chi\otimes\chi\big(x_m^i\big) & \text{if $\alpha\subset e_qve_p$}, \\
\chi\otimes\chi\big(x_p^k\big) & \text{if $\alpha\subset e_qve_m$},
\end{cases}
\end{gather*}
and if $\chi=\chi\big(x_m^{f-k}\big)$ define
\begin{gather*}
\chi_\text{dP}(C,D):=\begin{cases}
\chi\otimes\chi\big(x_p^{a+j}\big) & \text{if $\alpha\subset e_mve_q$}, \\
\chi\otimes\chi\big(x_q^{b+i}\big) & \text{if $\alpha\subset e_mve_p$}.
\end{cases}
\end{gather*}
When $\Delta$ is a~meeting of champions triangle with side ratios $x^d:y^b$, $y^e:z^c$, $x^a:z^f$ we make slight modifications to the above as follows. In~this setting, when $\chi=\chi\big(x^{d-i}\big)$ define
\begin{gather*}
\chi_\text{dP}(C,D):=\begin{cases}
\chi\otimes\chi\big(z^j\big) & \text{if $\alpha\subset e_1ve_2$}, \\
\chi\otimes\chi\big(y^k\big) & \text{if $\alpha\subset e_1ve_3$}.
\end{cases}
\end{gather*}
If $\chi=\chi\big(y^{e-j}\big)$ define
\begin{gather*}
\chi_\text{dP}(C,D):=\begin{cases}
\chi\otimes\chi\big(z^i\big) & \text{if $\alpha\subset e_2ve_1$}, \\
\chi\otimes\chi\big(x^{a+k}\big) & \text{if $\alpha\subset e_2ve_3$},
\end{cases}
\end{gather*}
and if $\chi=\chi\big(z^{f-k}\big)$ define
\begin{gather*}
\chi_\text{dP}(C,D):=\begin{cases}
\chi\otimes\chi\big(x^j\big) & \text{if $\alpha\subset e_3ve_2$}, \\
\chi\otimes\chi\big(y^{b+i}\big) & \text{if $\alpha\subset e_3ve_1$}.
\end{cases}
\end{gather*}

We remark that it follows from Case 4 of the proof of \cite[Theorem~6.1]{Craw05} that $\chi_\text{dP}(C,D)$ is one of the characters marking $D$, and is moreover the unique such character $\phi$ with $r_\chi\,|\, r_\phi$. It~also follows from the construction that $\chi_\text{dP}(C,D)$ takes the same value on~$\chi$-curves in~each of the two connected components of the $\chi$-chain minus the vertex~$v$.

\subsection{Unlocking procedure}

In this subsection we outline the algorithm that we use to compute $\Gig(C)$. We~will spend the remainder of this section proving its validity.

\begin{Algorithm}[unlocking procedure] \label{alg:up} \textnormal{Input:} An exceptional curve $C\subset G\hilb\A^3$ marked with a~character $\chi$ by Reid's recipe.
\begin{enumerate}\itemsep=0pt
\item[\textbf{Ch}] Let $S=\{\chi\}$.
\item[\textbf{dP}] For each del Pezzo divisor $D$ along the $\chi$-chain, add $\chi_\text{dP}(C,D)$ to $S$.
\item[\textbf{H1}] For each Hirzebruch divisor along the $\chi$-chain, add the character marking it to $S$.
\item[\textbf{Re}] For each Hirzebruch divisor $D$ downstream of $C$ and for each $E\in\mathcal{C}_C(D)$, compute $\Gig(E)$ by running the unlocking procedure with $E$ as input.
\item[\textbf{H2}] For each Hirzebruch divisor $D$ downstream of $C$, add the characters in~$\bigcup_{E\in\mathcal{C}_C(D)}\Gig(E)$ to $S$.
\end{enumerate}
\textnormal{Output:} $\Gig(C)=S$.
\end{Algorithm}

We call this the \textit{unlocking procedure} as passing through a~Hirzebruch divisor ``unlocks'' simpler $G$-igsaw puzzles for the curves $E$ downstream of $C$ that one recursively solves in~the step \textbf{\emph{Re}} and then feeds into the total $G$-igsaw piece for~$C$. It~can be visualised as a~flow through the triangulation emanating from the curve $C$ with preferred paths defining its tributaries. We~note that the convoluted definition of $\chi_\text{dP}(C,D)$ is only important for explicit calculations and not for qualitative discussion; the step \textbf{\emph{dP}} states that $\Gig(C)$ contains exactly one of the characters marking each del Pezzo divisor along the $\chi$-chain. We~will often refer to curves downstream of~$C$ as being ``unlocked'' by~$C$.

As an example use case, if $G$-Hilb has a~meeting of champions of side length $0$ with the three champions marked with a~character $\chi$ then for any curve $C$ along the $\chi$-chain the characters in~the $G$-igsaw piece are given by the unlocking procedure applied to the branch of the $\chi$-chain that $C$ lies on, combined with all the characters from (Hirzebruch) divisors along the other two branches of the $\chi$-chain. We~will see an example of this in~Section~\ref{ex:35}.

\subsection{Monomials for divisors}

We will begin by relating the characters marking divisors along the $\chi$-chain to $G$-igsaw pieces for~$\chi$-curves.

\begin{Lemma} \label{lem:div} Suppose $C$ is a~$\chi$-curve. Then $\Gig(C)$ includes exactly one character from each divisor that is along the $\chi$-chain. Moreover, if $D$ is a~del Pezzo divisor along the $\chi$-chain then $\chi_\text{\textnormal{dP}}(C,D)$ is the character marking $D$ that appears in~$\Gig(C)$.
\end{Lemma}

That is, $\Gig(C)$ contains the characters marking each Hirzebruch divisor along the $\chi$-chain and precisely one of the two characters marking each del Pezzo surface along the $\chi$-chain.

\begin{proof} Let $D$ be a~Hirzebruch divisor along the $\chi$-chain marked with a~character $\psi$. Cases~2--3 of the proof of \cite[Theorem~6.1]{Craw05} give that $\mR_\psi$ has degree $1$ on a~given $\chi$-curve, and hence it follows that $r_\chi\,|\, r_\psi$. It~follows that any $G$-igsaw piece featuring $r_\chi$~-- such as a~$G$-igsaw piece for~$C$~-- will also feature each $r_\psi$ and so $\psi\in\Gig(C)$. Now let $D$ be a~del Pezzo divisor along the $\chi$-chain marked with characters $\phi_1$, $\phi_2$. It~follows from Case 4 of the proof of \cite[Theorem~6.1]{Craw05} that exactly one of $\mR_{\phi_1}$, $\mR_{\phi_2}$ has degree $1$ on a~given $\chi$-curve and so $r_\chi\,|\, r_{\phi_1}$ or $r_\chi\,|\, r_{\phi_2}$ but not both. It~follows that exactly one of $\phi_1$, $\phi_2$ lie in~$\Gig(C)$. Let $\phi=\chi_\text{dP}(C,D)\in\{\phi_1,\phi_2\}$. As~noted above, it follows from Case 4 of the proof of \cite[Theorem~6.1]{Craw05} that $r_\chi\,|\, r_\phi$ and so $r_\phi$ is the unique character marking $D$ that appears in~$\Gig(C)$.
\end{proof}

Lemma~\ref{lem:div} gives an effective way of finding the characters in~$\Gig(C)$ coming from divisors. However, this does not usually supply all characters in~$\Gig(C)\setminus\{\chi\}$.

\subsection{Counting characters}

Our method for showing that Algorithm \ref{alg:up} is valid for a~curve $C$ is to source many characters from divisors (as discussed in~the previous subsection) and from curves (coming next) that feature in~$\Gig(C)$ and to then count how many characters are actually in~$\Gig(C)$ to verify that all characters in~the total $G$-igsaw piece have been located. To move towards this second aim we cite a~lemma of Craw--Ishii.

\begin{Lemma}[{\cite[Lemma~9.1]{CrawIshii04}}] \label{lem:ci} A character $\chi$ marks a~torus-invariant compact divisor $D\subset G\hilb\A^3$ iff $r_\chi$ is in~the socle of every $G$-cluster corresponding to a~torus-fixed point in~$D$.
\end{Lemma}

Select a~$(-1,-1)$-curve $C$ marked with $\chi$. This lies in~two del Pezzo divisors from the endpoints of the corresponding line segment. From Lemma~\ref{lem:div} we see that $r_\chi$ divides exactly two of the monomials in~the character spaces labelling these two divisors. Suppose $\tau$ is a~$\chi$-triangle neighbouring $C$. By the shape of the ratios in~Fig.~\ref{fig:gens} we can assume that $r_\chi$ is not a~power of a single variable. The Unique Valley Lemma~\cite[Lemma~3.3]{Nakamura01} of Nakamura implies that $r_\chi$ divides exactly two elements of the socle of the torus-invariant $G$-cluster $Z_\tau$ corresponding to~$\tau$. Lemma~\ref{lem:ci} implies that the elements in~the socle of $Z_\tau$ that $r_\chi$ divides correspond exactly to these two characters labelling the neighbouring del Pezzo divisors. These are the outermost monomials in~the $G$-igsaw piece for~$C$ on~$\tau$, so that knowing them will allow us to count how many characters appear in~$\Gig(C)$.

Using this observation we will first prove the validity of the unlocking procedure for curves inside regular triangles (i.e.~those able to define flops, or $(-1,-1)$-curves) before justifying the procedure for the other exceptional curves.

\subsection[(-1,-1)-curves]{$\boldsymbol{(-1,-1)}$-curves}

We consider four cases covering all $(-1,-1)$-curves in~$G$-Hilb based on the different ratios labelling edges in~Fig.~\ref{fig:gens}. For this subsection denote $x_1=x$, $x_2=y$, $x_3=z$. We~will use indices $\{p,q,r\}=\{1,2,3\}$ to symmetrise the discussion. A $(p,q)$\emph{-triangle} is an $e_p$-corner triangle with one edge coming from a~straight line out of $e_q$.
\begin{itemize}\itemsep=0pt
\item Type \texttt{Ix}: curves in~the interior of a $(r,p)$-triangle with ratios $x_p^{d-i}:x_q^{b+i}x_r^i$.
\item Type \texttt{Iy}: curves in~the interior of an~$(r,p)$-triangle with ratios $x_q^{e-j}:x_p^{a+j}x_r^j$.
\item Type \texttt{Iz}: curves in~the interior of an~$(r,p)$-corner triangle with ratios $x_r^{f-k}:x_p^kx_q^{c+k}$.
\item Type \texttt{Ic}: curves in~the interior of the meeting of champions triangle (if existent).
\end{itemize}

We will treat each of these types of $(-1,-1)$-curves but will specialise to the case $p=1$, $q=2$, $m=3$, which suffices to cover all possibilities by symmetry. Fix a~$(3,1)$-triangle $\Delta$.

\subsubsection{Type \texttt{Iy} curves}

We consider the edges in~the interior of $\Delta$ marked with ratios of the form $y^{e-j}:x^{a+j}z^j$; that is, of Type \texttt{Iy}. The analysis from Section~\ref{sec:comb_def} gives a~precise description of the socle of the $G$-clusters corresponding to nearby torus-invariant points as depicted in~Fig.~\ref{fig:gensy} and hence we identify the total $G$-igsaw pieces for such $\chi$-curves. The only additional calculation required is of the monomials $r_{\phi_0}$ and $r_{\phi_m}$ for the characters at the endpoints. Consider $r_{\phi_m}$. The ratio marking the side of $\Delta$ containing the vertex marked with $\phi_m$ is $z^f:y^c$. It~follows that $y^c\,|\, r_{\phi_m}$ on~$\Delta$ but then it cannot be the case that $r_\chi$ divides $r_{\phi_m}$ since $r_\chi=x^{a+j}z^j$ for some basic triangles in~$\Delta$. It~follows from Section~\ref{sec:rr} that $r_{\phi_m}$ is given by $x^{a+j}y^c$ in~the basic triangle where it is displayed in~Fig.~\ref{fig:gensy}. A similar argument applies to compute $r_{\phi_0}$.

\begin{figure}[h]\centering
\begin{tikzpicture}[scale=0.65]
\node[fill=white,inner sep=2pt](b) at (-5,0){\small$\phi_0$};

\node[fill=white,inner sep=2pt] (o1) at (0,0){\small$\begin{matrix}
\phi_1 \\
\phi_2\end{matrix}$};
\node(a) at (2,2){};
\node(c) at (-0.5,-2.5){};
\node(d) at (-2,-2){};
\node(e) at (0.5,2.5){};

\node (up) at (-3,2){};
\node (down) at (-7,-2){};

\draw (b) to (up);
\draw (b) to (down);

\draw (o1) to (a);
\draw (o1) to (b);
\draw (o1) to (c);
\draw (o1) to (d);
\draw (o1) to (e);

\node[fill=white,inner sep=2pt](o2) at (5,0){\small$\begin{matrix}
\phi_3 \\
\phi_4\end{matrix}$};
\node(a) at (7,2){};
\node(c) at (4.5,-2.5){};
\node(d) at (3,-2){};
\node(e) at (5.5,2.5){};

\draw (o1) to (o2);
\draw (o2) to (a);
\draw (o2) to (c);
\draw (o2) to (d);
\draw (o2) to (e);

\node (ell) at (8.5,0){$\dots$};

\node[fill=white,inner sep=2pt](o3) at (11.5,0){\small$\begin{matrix}
\phi_{m-2} \\
\phi_{m-1}\end{matrix}$};
\node(a) at (13.5,2){};
\node(c) at (11,-2.5){};
\node(d) at (9.5,-2){};
\node(e) at (12,2.5){};

\node(b) at (9,0){};
\node(f) at (16.5,0){\small$\phi_m$};

\node (up) at (16,-2.5){};
\node (down) at (17,2.5){};

\draw (o3) to (a);
\draw (o3) to (b);
\draw (o3) to (c);
\draw (o3) to (d);
\draw (o3) to (e);
\draw (o3) to (f);

\draw (o2) to (ell);
\draw (ell) to (o3);

\draw (f) to (up);
\draw (f) to (down);

\node (l1) at (-3.4,-1.1){\footnotesize$\begin{array}{l}
r_{\phi_0}=y^bz^j \\
r_{\phi_1}=x^{a+j}z^{j+1}\end{array}$};
\node (l2) at (1.9,-1.1){\footnotesize$\begin{array}{l}
r_{\phi_2}=x^{d-1}z^j\\
r_{\phi_3}=x^{a+j}z^{j+2} \end{array}$};
\node (l3) at (6.3,-2.7){\footnotesize$\begin{array}{l}
r_{\phi_4}=x^{d-2}z^j \\
r_{\phi_5}=x^{a+j}z^{j+3}\end{array}$};
\node (l2) at (7.9,-1.1){\footnotesize$\begin{array}{l}
r_{\phi_{m-3}}=x^{d-(f-j-2)}z^j \\
r_{\phi_{m-2}}=x^{a+j}z^{f-1}\end{array}$};
\node (l2) at (13.9,-1.1){\footnotesize$\begin{array}{l}
r_{\phi_{m-1}}=x^{d-(f-j-1)}z^j \\
r_{\phi_m}=x^{a+j}y^c\end{array}$};
\end{tikzpicture}
\caption{Generators of eigenspaces along a~$\chi(x^{a+j}z^j)$-chain inside a~regular triangle.}\label{fig:gensy}
\end{figure}
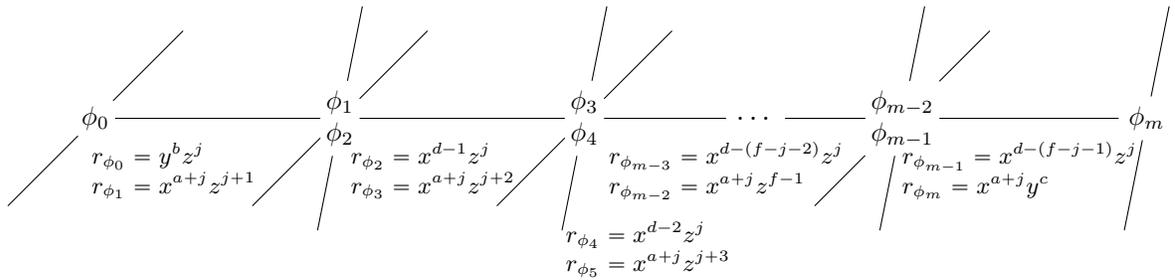

\begin{Lemma} \label{lem:piey} A total $G$-igsaw piece for a~$\chi$-curve of Type \texttt{Iy} on a~$\chi$-triangle chosen so that in~the coordinates used above $r_\chi=x^{a+j}z^j$ is
\begin{gather*}
\begin{array}{ccccccccc}
& & & r_\chi & xr_\chi & \dots & x^{f-i-j-1}r_\chi\\
& & zr_\chi \\
& \iddots \\
z^ir_\chi
\end{array}
\end{gather*}
where the curve corresponds to the edge whose endpoints are the intersection of the $\chi$-chain with the lines marked with $x^{d-i}:y^{b+i}z^i$ and $x^{d-i-1}:y^{b+i+1}z^{i+1}$. Moreover, the $\chi$-chain does not continue outside of this regular triangle. In~particular, $\Hirz(\chi)=\emptyset$.
\end{Lemma}

\begin{proof} The calculation of the total $G$-igsaw piece follows immediately from the description of the eigenmonomials in~Fig.~\ref{fig:gensy}. As~noted the $\chi$-chain cannot continue outside of this regular triangle since neither $r_{\phi_0}$ nor $r_{\phi_m}$ are divisible by $r_\chi$ and so Theorem~\ref{thm:craw} implies that there cannot be two edges marked with $\chi$ incident to either boundary vertex.
\end{proof}

Notice that this means that there are $f-j-1$ characters to account for, excluding $\chi$. But this is exactly the number of del Pezzo surfaces along the $\chi$-chain, each of which contributes one character.

\begin{Corollary} For an exceptional curve $C$ of Type \texttt{Iy} $\Gig(C)$ consists exactly of $\chi$ and the characters $\chi_\text{\textnormal{dP}}(C,D)$ for each del Pezzo divisor $D$ along the $\chi$-chain.
\end{Corollary}

Observe that this is a~situation in~which there is no recursion necessary since $\Hirz(\chi)=\emptyset$. This is one of the base cases that we will reduce to.

\subsubsection{Type \texttt{Ix} curves}

Suppose now that $C$ is a~$\chi$-curve inside $\Delta$ that is marked with the ratio $x^{d-i}:y^{b+i}z^i$; that is, $C$ is of Type \texttt{Ix}. \cite[Theorem~6.1]{Craw05} yields the identities in~Fig.~\ref{fig:gensx} for eigenmonomials on triangles neighbouring the $\chi$-chain, which allow us to completely describe $G$-igsaw pieces inside regular triangles. In~the following we continue the notation of Fig.~\ref{fig:gens} and let $\kappa=r-(i+1)$.

\begin{figure}[h]\centering
\begin{tikzpicture}[scale=0.65]
\node[fill=white,inner sep=2pt](b) at (-5,0){\small$\phi_0$};

\node[fill=white,inner sep=2pt] (o1) at (0,0){\small$\begin{matrix}
\;\;\;\phi_1^1 \\
\phi_1^2\end{matrix}$};
\node(a) at (2,2){};
\node(c) at (-0.5,-2.5){};
\node(d) at (-2,-2){};
\node(e) at (0.5,2.5){};

\node (up) at (-3,2){};
\node (down) at (-7,-2){};

\draw (b) to (up);
\draw (b) to (down);

\draw (o1) to (a);
\draw (o1) to (b);
\draw (o1) to (c);
\draw (o1) to (d);
\draw (o1) to (e);

\node[fill=white,inner sep=2pt](o2) at (5,0){\small$\begin{matrix}
\;\;\;\phi_2^1 \\
\phi_2^2\end{matrix}$};
\node(a) at (7,2){};
\node(c) at (4.5,-2.5){};
\node(d) at (3,-2){};
\node(e) at (5.5,2.5){};

\draw (o1) to (o2);
\draw (o2) to (a);
\draw (o2) to (c);
\draw (o2) to (d);
\draw (o2) to (e);

\node (ell) at (8.5,0){$\dots$};

\node[fill=white,inner sep=2pt](o3) at (11.5,0){\small$\begin{matrix}
\;\;\;\phi_\kappa^1 \\
\phi_\kappa^2\end{matrix}$};
\node(a) at (13.5,2){};
\node(c) at (11,-2.5){};
\node(d) at (9.5,-2){};
\node(e) at (12,2.5){};

\node(b) at (9,0){};
\node[fill=white,inner sep=2pt](f) at (16.5,0){\small$\phi_m$};

\node (up) at (16,-2.5){};
\node (down) at (17,2.5){};

\draw (o3) to (a);
\draw (o3) to (b);
\draw (o3) to (c);
\draw (o3) to (d);
\draw (o3) to (e);
\draw (o3) to (f);

\draw (o2) to (ell);
\draw (ell) to (o3);

\draw (f) to (up);
\draw (f) to (down);

\node (l1) at (-3.5,-1.2){\small$\begin{array}{l}
r_{\phi_0}=x^az^i \\
r_{\phi_1^1}=x^{d-i}y^{c+\kappa}\end{array}$};
\node (l2) at (2.1,-1.8){\small$\begin{array}{l}
r_{\phi_1^2}=x^{d-i}z\\
r_{\phi_2^1}=x^{d-i}y^{c+\kappa-1} \end{array}$};
\node (l3) at (6.9,-2.9){\small$\begin{array}{l}
r_{\phi_2^2}=x^{d-i}z^2 \\
r_{\phi_5}=x^{d-i}y^{c+\kappa-2}\end{array}$};
\node (l2) at (7.9,-1.2){\small$\begin{array}{l}
r_{\phi_{\kappa-1}^2}=x^{d-i}z^{\kappa-2} \\
r_{\phi_\kappa^1}=x^{d-i}y^{c+1}\end{array}$};
\node (l2) at (13.8,-1.2){\small$\begin{array}{l}
r_{\phi_\kappa^2}=x^{d-i}z^{\kappa-1} \\
r_{\phi_m}=x^{d-i}y^c\end{array}$};
\end{tikzpicture}
\caption{Generators of eigenspaces along a~$\chi$-chain inside a~regular triangle.}\label{fig:gensx}
\end{figure}
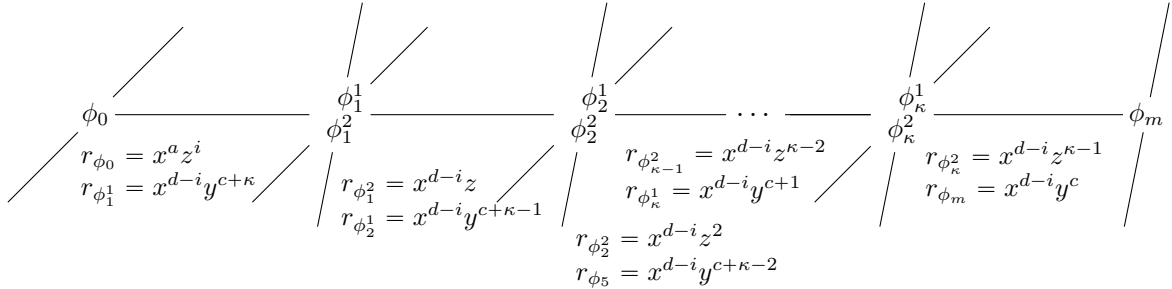

\begin{Lemma} \label{lem:piex} The $G$-igsaw piece for a~$\chi$-curve $C$ of Type \texttt{Ix} on a~$\chi$-triangle chosen so that in~the coordinates used above $r_\chi=x^{d-i}$ is
\begin{gather*}
\begin{array}{ccccccccc}
y^{c+k-1}r_\chi \\
& \ddots \\
& & yr_\chi \\
& & & r_\chi \\
& & zr_\chi \\
& \iddots \\
z^jr_\chi
\end{array}
\end{gather*}
where $C$ corresponds to the edge whose endpoints are the intersection of the $\chi$-chain with the lines marked with $y^{e-j}:x^{a+j}z^j$ and $y^{e-j-1}:x^{a+j+1}z^{j+1}$, and where $i+j+k=r$. Moreover, the $\chi$-chain continues to the right and does not continue to the left of Fig.~{\rm \ref{fig:gensx}}.
\end{Lemma}

\begin{proof} The same argument as for Lemma~\ref{lem:piey} applies, except that $r_\chi$ does divide $r_{\phi_m}$ and so by Theorem~\ref{thm:craw} the $\chi$-chain must continue past the rightmost vertex.
\end{proof}

Notice that the only characters in~any such $G$-igsaw piece that are unaccounted for by divisors along the $\chi$-chain in~the same regular triangle are those for the monomials
\begin{gather*}
yr_\chi,\ \dots,\ y^cr_\chi
\end{gather*}
though $y^cr_\chi=r_{\phi_m}$, which we have seen corresponds to a~Hirzebruch divisor appearing along the $\chi$-chain.

\begin{Lemma} \label{lem:piexb} Suppose $C$ is a~$\chi$-curve of Type \texttt{Ix} such that the $\chi$-chain continues into a~boundary edge of a corner triangle. Then $\Gig(C)$ consists of $\chi$, one character from every del Pezzo divisor along the $\chi$-chain, and the characters marking Hirzebruch divisors along the $\chi$-chain.
\end{Lemma}

This follows since the corner triangle has side length $c$ and so there are exactly $c$ Hirzebruch divisors along the boundary part of the $\chi$-chain that contribute the remaining $c$ characters to the $G$-igsaw piece. We~say that the curves from Lemma~\ref{lem:piexb} are of Type~\texttt{Ixb}. This is the other base case to which the unlocking procedure reduces. Note that the character in~$\Gig(C)$ from a~del Pezzo divisor $D$ along the $\chi$-chain is by definition $\chi_{\text{dP}}(C,D)$.

We consider the remaining possibilities where the $\chi$-chain merges into the interior of an~$e_2$-cor\-ner triangle or the interior of an~$e_1$-corner triangle.

\begin{Lemma} \label{lem:ix'} Suppose $C$ is a~$\chi$-curve of Type \texttt{Ix} and suppose that the $\chi$-chain continues into the interior of an~$e_2$-corner triangle $\Delta'$. Then $\Gig(C)$ consists of $\chi$, one character from each del Pezzo surface along the $\chi$-chain, the character marking the Hirzebruch divisor $D$ between the two regular triangles, and the characters from the total $G$-igsaw piece of the \texttt{Iy} curve also incident to $D$ inside $\Delta'$.
\end{Lemma}

As above, the character in~$\Gig(C)$ from a~del Pezzo divisor $D$ along the $\chi$-chain is $\chi_{\text{dP}}(C,D)$.

\begin{proof} Let $C'$ be the Type \texttt{Iy} curve incident to $D$ in~$\Delta'$. Denote its character by $\chi'$. From Lemma~\ref{lem:piey} the characters in~the total $G$-igsaw piece for~$C'$ are $\chi'$ and one character from each del Pezzo divisor along the $\chi'$-chain inside $\Delta'$. Let the ratios marking sides of $\Delta'$ be
\begin{gather*}
x^{d'}:z^{b'},\quad
z^{e'}:x^{a'},\quad
z^f:y^c,
\end{gather*}
where $a'<d'$ and where $z^f:y^c$ marks the common side with $\Delta$. Let the part of the $\chi$-chain in~$\Delta'$ be marked by the ratio $x^{d'-i'}:z^{b'+i'}y^{i'}$ and so $d'-i'=d-i$. It~follows that the $\chi'$-chain is marked by $z^{e'-j'}:x^{a'+j'}y^{j'}$ where $i'+j'=f$. Using the relation $d'-a'=c$ we see that $a'+j'=d-i$.

Examining the situation explicitly, we see that on the lower $\chi$-triangle neighbouring $C$ one has $r_\chi=x^{d-i}$ and $r_{\chi'}=x^{a'+j'}y^{j'}$ so that $r_\chi\,|\, r_{\chi'}$ near $C$. Moreover, one can see that the zone where $r_{\chi'}$ divides one character from each del Pezzo divisor along the $\chi'$-chain includes this $\chi$-triangle containing $C$ and so these divisibility relations remain. Hence, the $G$-igsaw piece for~$C'$ is contained in~the $G$-igsaw piece for~$C$. The divisibility relations are depicted in~Fig.~\ref{fig:divx'}.

From Lemma~\ref{lem:piex} the total $G$-igsaw piece for~$C$ is missing $c$ characters after counting the characters in~$\Delta$. There are $c-i'$ new characters along the $\chi$-chain corresponding to the del Pezzo divisors along the $\chi$-chain and the boundary Hirzebruch divisor $D$. There are $c-(c-i')-1=i'-1$ divisors along the $\chi'$-chain, making a~contribution of $i'$ characters in~total including $\chi'$ itself. Thus these account for all of the $c$ missing characters.
\end{proof}

\begin{figure}[h]\centering
\begin{tikzpicture}[scale=0.76]
\node (a) at (-4,2){};
\node (v) at (0,2){$\bullet$};
\node (w) at (2,0){$\bullet$};
\node (b) at (3,-1){};
\node (c) at (2,-1){};
\node (d) at (2,1){};
\node (e) at (0,0){};
\node (f) at (3,0){};
\node (g) at (3.2,1.1){};

\draw (c) to (d);
\draw (e) to (f);
\draw (a) to (v.center);
\draw (w.center) to (v.center);
\draw (w.center) to (b);
\draw (v.center) to (g);

\small
\node (lx) at (-4.2,2){$\chi$};
\node (ly) at (2,2.2){$\chi'$};
\node (rel) at (-2,0.75){$\begin{array}{l}
r_\chi=x^{d-i} \\
r_{\chi'}=x^{d-i}y^{j'} \\
r_\chi\,|\, r_{\chi'}\,|\, r_{\phi_i^1}\end{array}$};
\node (l1) at (4,0.8){ $\phi_1^1,\ \phi_1^2$};
\node (ell) at (3.4,-1.2){$\ddots$};
\node (m) at (1,1){};

\draw[->] (l1) to (w);
\draw[->] (ly.south) to (m);

\small

\node (w') at (5,-3){$\bullet$};
\node (b') at (6,-4){};
\node (c') at (5,-4){};
\node (d') at (5,-2){};
\node (e') at (3,-3){};
\node (f') at (6,-3){};

\draw (c') to (d');
\draw (e') to (f');
\draw (ell) to (w'.center);
\draw (w'.center) to (b');

\small
\node (li) at (6,-2){ $\phi_i^1,\ \phi_i^2$};
\draw[->] (li) to (w');
\end{tikzpicture}
\caption{Unlocking for a~Type \texttt{Ix} curve merging into a~$(2,1)$-triangle.}\label{fig:divx'}
\end{figure}
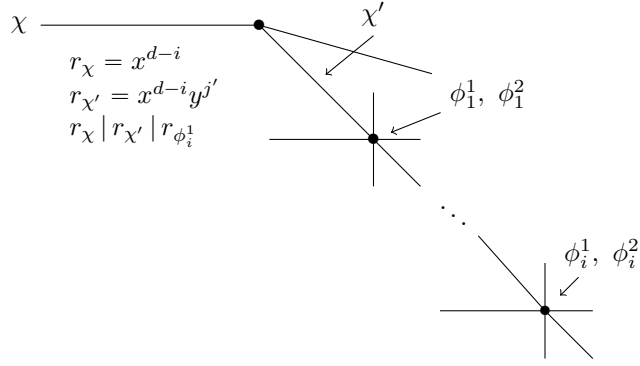

Note that this vindicates the unlocking procedure for such curves, where only one recursion was required to unlock the single Type \texttt{Iy} curve downstream of $C$. The final case to consider is when the $\chi$-chain merges into an $e_1$-corner triangle.

Suppose the $\chi$-chain passes through $n$ $e_1$-corner triangles before entering an $e_2$-corner tri\-angle~$\Delta'$.

Let the ratio $x^{d_m}:y^{b_m}$ mark the edge opposite $e_1$ for the $m$th $e_1$-corner triangle $\Delta_m$ from the left and so $\Delta_m$ has side length $d_m$. Suppose the $\chi$-chain enters $\Delta_m$ at height $i_m$. This means that the $\chi$-chain picks up $d_m-i_m$ divisors from del Pezzo divisors and a~single Hirzebruch divisor inside $\Delta_m$. From analysing local divisibility relations as above, it is clear that $r_\chi$ divides all of the monomials in~the $G$-igsaw pieces for the Type \texttt{Ix} curve incident to the $\chi$-chain and the leftmost Hirzebruch divisor inside each of these regular triangles. See Fig.~\ref{fig:ixc} for a~schematic. We~denote $D_m:=\sum_{q=1}^m d_q$ and $BD_m:=\sum_{q=1}^m(b_q+d_q)$.

\begin{figure}[h]\centering
\begin{tikzpicture}[scale=0.72]
\node (a0) at (0,0){$\bullet$};
\node (a1) at (3,0){$\bullet$};
\node (a2) at (6,0){$\bullet$};
\node (a3) at (9,0){$\bullet$};

\node (e1) at (4,4){};
\node (e2) at (14,-1){};

\node (a-1) at (-1.5,-0.3){$\bullet$};
\node (a0') at (2,2){};
\node (a-1') at (2-4.5,2-0.9){};

\node (b0) at (4-1.5*4,4-1.5*4){};
\node (b1) at (4-1.6*1,4-1.6*4){};
\node (b2) at (4+2*1.6,4-4*1.6){};
\node (b3) at (4+5*1.5,4-4*1.5){};

\node (x0) at (0-1*0.25,-4*0.25){$\bullet$};
\node (x1) at (3+2*0.25,-4*0.25){$\bullet$};
\node (x2) at (6+5*0.25,-4*0.25){$\bullet$};

\node (q0) at (0-1*0.125,-4*0.125){};
\node (q1) at (3+2*0.125,-4*0.125){};
\node (q2) at (6+5*0.125,-4*0.125-0.1){};

\node (p0) at (3,3){};
\node (p1) at (3.8,3.3){};
\node (pd) at (5.4,3){};

\small

\node (C) at (-0.6,0.2){$C$};

\node (ch1) at (-1,-3.5){$\chi(x^{b_1+i-1}z^{f+1})$};
\node (ch2) at (1,-2.7){$\chi(z^{f+1})$};
\node (ch3) at (2.6,-3.5){$\chi(z^{f+d_1+1})$};
\node (ch4) at (5.3,-2.7){$\chi(x^{b_1+i-1}z^{f+d_1+1})$};
\node (ch5) at (8.1,-3.5){$\chi(z^{f+D_{n-1}+1})$};
\node (ch6) at (11.5,-2.7){$\chi(x^{b_{n-1}+i-1}z^{f+D_{n-1}+1})$};

\node (cz0) at (1,4.6){$z^f:y^c$};
\node (cz1) at (4,4.6){$z^{f+d_1}:y^{c-(b_1+d_1)}$};
\node (czd) at (8.2,4.6){$z^{f+D_n}:y^{c-BD_n}$};

\draw (a-1.center) to (a0.center) to (a1.center) to (a2.center) to (a3.center) to (e2.center);

\draw (a0'.center) to (a-1');

\draw (e1.center) to (b0);
\draw (e1.center) to (b1);
\draw (e1.center) to (b2);
\draw (e1.center) to (b3);

\draw (a0.center) to (x0.center);
\draw (a1.center) to (x1.center);
\draw (a2.center) to (x2.center);

\draw[->] (ch1) to (x0);
\draw[->] (ch2) to (q0);
\draw[->] (ch3) to (q1);
\draw[->] (ch4) to (x1);
\draw[->] (ch5) to (q2);
\draw[->] (ch6) to (x2);

\draw[->] (cz0) to (p0);
\draw[->] (cz1) to (p1);
\draw[->] (czd) to (pd);
\end{tikzpicture}
\caption{Unlocking for a~Type \texttt{Ix} curve in~a~series of $e_1$-corner triangles.}\label{fig:ixc}
\end{figure}
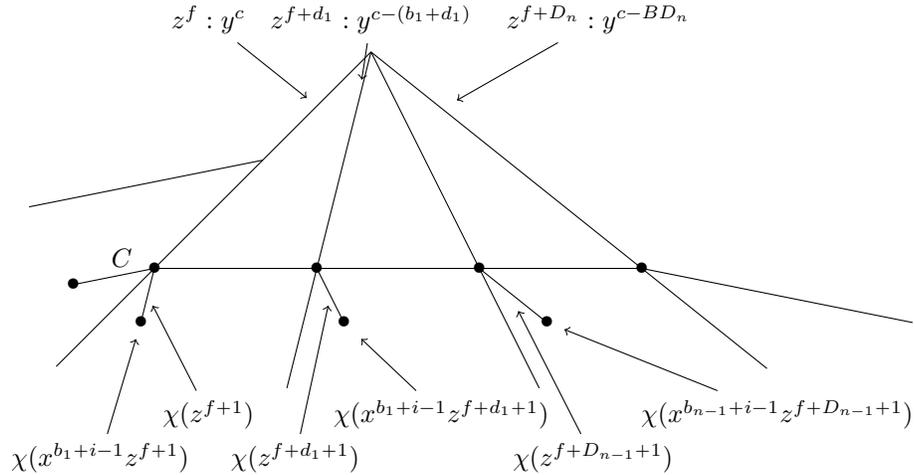

By computing the characters on the nearby del Pezzo divisor, one can tell that these Type \texttt{Ix} curves each have $b_m+i_m$ characters in~their $G$-igsaw pieces, making the total number of cha\-rac\-ters they contribute to the $G$-igsaw piece of $C$
\begin{gather*}
\sum_{q=1}^n(d_q-i_q+b_q+i_q)=\sum_{q=1}^n(b_q+d_q).
\end{gather*}
From \cite[Section~2]{Craw05} the ratios marking the edges from $e_1$ for the $e_1$-corner triangles are of the form
\begin{gather*}
z^{f+\sum_{q=1}^md_q}:y^{c-\sum_{q=1}^m(b_q+d_q)}\qquad\text{for~$m=0,\dots,n$}
\end{gather*}
with the last edge marked by $z^{f+\sum_{q=1}^nd_q}:y^{c-\sum_{q=1}^n(b_q+d_q)}$. In~particular, this means that the $\Delta'$ has side length $c-\sum_{q=1}^n(b_q+d_q)$. Assume the $\chi$-chain continues into a~chain of Type \texttt{Ix} curves in~$\Delta'$. By the same reasoning as for Lemma~\ref{lem:ix'} this produces $c-\sum_{q=1}^n(b_q+d_q)$ new characters in~$\Gig(C)$ coming from $\Delta$. But then we have
\begin{gather*}
\sum_{q=1}^n(b_q+d_q)+c-\sum_{q=1}^n(b_q+d_q)=c
\end{gather*}
characters in~total so far, which exhausts all characters in~$\Gig(C)$ by Lemma~\ref{lem:piex}. Hence in~this case $\Delta'$ is the rightmost regular triangle containing $\chi$-curves. Note that the $\chi$-chain cannot merge into a~chain of Type \texttt{Iy} curves in~$\Delta'$ as such curves cannot escape a~single regular triangle. Also, it is clear from convex geometric considerations that the $\chi$-chain cannot continue into a~chain of Type \texttt{Iz} curves. The only remaining option is that the $\chi$-chain continues into a~chain of boundary curves, thus again producing $c-\sum_{q=1}^n(b_q+d_q)$ new characters from the Hirzebruch divisors along the side of $\Delta'$. In~either case the number of characters coming from $\Delta'$ is exactly the number of characters in~$\Gig(C)$ not accounted for by del Pezzo divisors in~$\Delta$ by Lemma~\ref{lem:piex}. This completes the proof of validity of the unlocking procedure for curves of Type~\texttt{Ix}.

\subsubsection{Type \texttt{Iz} curves}

The third type of curve occurring inside regular triangles is Type \texttt{Iz}: the curves marked by ratios of the form $z^{f-k}:x^ky^{c+k}$ in~the coordinates we have been using for an $e_3$-corner triangle. We~repeat the $G$-igsaw analysis for these curves, represented in~Fig.~\ref{fig:11} with the $\chi=\chi\big(z^{f-k}\big)$-chain dashed.

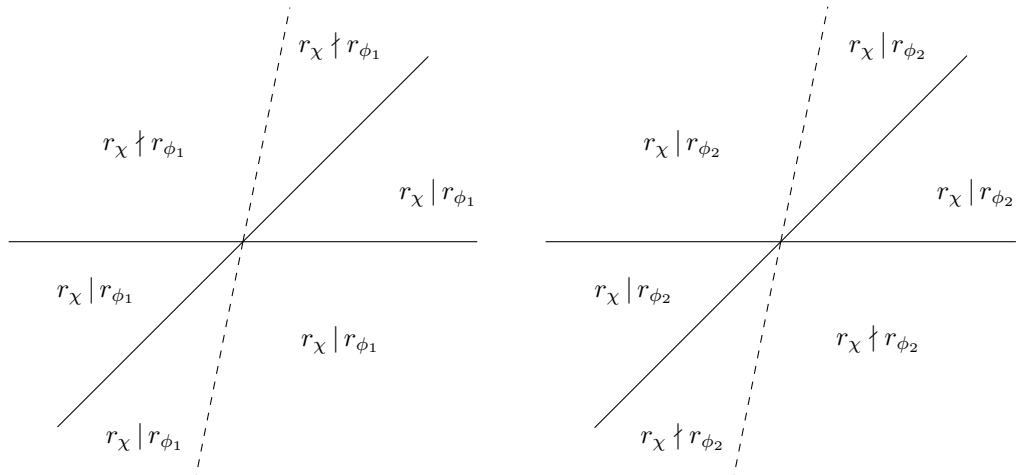
\begin{figure}[h]\centering
\begin{tikzpicture}[scale=0.65]
\node(o) at (0,0){};

\node(a) at (4,4){};
\node(b) at (-5,0){};
\node(c) at (-1,-5){};

\node(d) at (-4,-4){};
\node(e) at (1,5){};
\node(f) at (5,0){};

\draw (o.center) to (a);
\draw (o.center) to (b);
\draw[dashed] (o.center) to (c);
\draw (o.center) to (d);
\draw[dashed] (o.center) to (e);
\draw (o.center) to (f);

\small

\node(rl1) at (-2,2){$r_\chi\not\mid  r_{\phi_1}$};
\node(rl2) at (2,4){$r_\chi \not\mid  r_{\phi_1}$};
\node(rl3) at (4,1){$r_\chi\,|\, r_{\phi_1}$};
\node(rl4) at (2,-2){$r_\chi\,|\, r_{\phi_1}$};
\node(rl5) at (-2,-4){$r_\chi\,|\, r_{\phi_1}$};
\node(rl6) at (-3,-1){$r_\chi\,|\, r_{\phi_1}$};

\node(o) at (11,0){};

\node(a) at (15,4){};
\node(b) at (6,0){};
\node(c) at (10,-5){};

\node(d) at (7,-4){};
\node(e) at (12,5){};
\node(f) at (16,0){};

\draw (o.center) to (a);
\draw (o.center) to (b);
\draw[dashed] (o.center) to (c);
\draw (o.center) to (d);
\draw[dashed] (o.center) to (e);
\draw (o.center) to (f);

\node(rl1) at (9,2){$r_\chi\,|\, r_{\phi_2}$};
\node(rl2) at (13.2,4){$r_\chi\,|\, r_{\phi_2}$};
\node(rl3) at (15,1){$r_\chi\,|\, r_{\phi_2}$};
\node(rl4) at (13,-2){$r_\chi\not\mid  r_{\phi_2}$};
\node(rl5) at (9,-4){$r_\chi\not\mid  r_{\phi_2}$};
\node(rl6) at (8,-1){$r_\chi\,|\, r_{\phi_2}$};
\end{tikzpicture}
\caption{Divisibility relations near $v$.}\label{fig:11}
\end{figure}

As in~all previous cases, exactly one character marking each incident del Pezzo surface has a~monomial divisible by $r_\chi$ and so we can pin down the socle and hence the $G$-igsaw piece for such a~curve.

\begin{Lemma} \label{lem:piez} The $G$-igsaw piece for a~$(-1,-1)$-curve marked with $\chi$ on a~$\chi$-triangle chosen so that in~the coordinates used above $r_\chi=z^{f-k}$ is
\begin{gather*}
\begin{array}{ccccccccc}
y^{b+i-1}r_\chi \\
& \ddots \\
& & yr_\chi \\
& & & r_\chi & xr_\chi & \dots & x^{d-i-k}r_\chi
\end{array}
\end{gather*}
where the curve corresponds to the edge whose endpoints are the intersection of the $\chi$-chain with the lines marked with $x^{d-i}:y^{b+i}z^i$ and $x^{d-i-1}:y^{b+i+1}z^{i+1}$.
\end{Lemma}

This means that there are $b+d-k$ characters in~the $G$-igsaw piece for such a~\texttt{Iz} curve. We~shift notation to match the setup of the final case for Type \texttt{Ix} curves shown in~Fig.~\ref{fig:ixc}. In~particular, we assume our Type \texttt{Iz} curve $C$ lies in~an~$e_1$-corner triangle. Suppose it lies in~the $m$th triangle from the left. From considering local divisibility relations near Hirzebruch divisors along the $\chi$-chain this implies that $C$ unlocks $m-1$ Type~\texttt{Iy} curves to the left and $n-m$ Type~\texttt{Ix} curves to the right. From the calculations for Type~\texttt{Ix} curves, the $n-m$ Type \texttt{Ix} curves each feature $b_q+i_q$ characters in~their $G$-igsaw pieces. From a~similar calculation, one can verify that the Type \texttt{Iy} curves contain $i_q$ characters in~their $G$-igsaw pieces. These unlocked curves thus contribute
\begin{gather*}
\sum_{q=1}^{m-1}i_q+\sum_{q=m+1}^n(b_q+i_q)=\sum_{q=m+1}^nb_q+\sum_{q=1}^ni_q-i_m
\end{gather*}
characters to $\Gig(C)$. The part of the $\chi$-chain in~the $e_3$-corner triangle studied in~the previous case contributes $f-i_0$ characters, and the part in~the $e_2$-corner triangle contributes \mbox{$c-\sum_{q=1}^n(b_q+d_q)$}. If $i_0\not=0$ then we unlock another \texttt{Iy} curve with $i_0$ characters appearing in~its $G$-igsaw piece. If $i_0=0$ then the $\chi$-chain continues along the boundary of an~$e_3$-corner tri\-angle, contributing $f$ characters. In~either case there are $f$ characters coming from the $e_3$-corner triangle. Lastly, there are $\sum_{q=1}^n(d_q-i_q)$ del Pezzo and Hirzebruch divisors along the part of the $\chi$-chain inside $e_1$-corner triangles, giving in~total
\begin{gather*}\underset{\text{$e_3$-corner}}{\underbrace{f}}+\underset{\text{unlocked curves}}{\underbrace{\sum_{q=m+1}^nb_q+\sum_{q=1}^ni_q-i_m}}+\underset{\text{$e_1$-corner}}{\underbrace{\sum_{q=1}^n(d_q-i_q)}}+\underset{\text{$e_2$-corner}}{\underbrace{c-\sum_{q=1}^n(b_q+d_q)}}=f+c-\sum_{q=1}^mb_q-i_m
\end{gather*}
characters. Compare to the quantity $b+d-k$ in~Lemma~\ref{lem:piez}, which in~these coordinates is
\begin{gather*}
c-\sum_{q=1}^m(b_q+d_q)+f+\sum_{q=1}^md_q-i_m=f+c-\sum_{q=1}^mb_q-i_m
\end{gather*}
showing that every character in~$\Gig(C)$ is accounted for.

\subsubsection{Type \texttt{Ic} curves}

As in~\cite{Craw05} the case of curves whose chains meet the interior of a meeting of champions triangle only requires minor notational changes for the arguments above to carry over verbatim. For~brevity we omit it.

\subsection{Boundary curves}

Suppose now that $C$ is a~curve lying on the boundary of a regular triangle. We~will see that the unlocking procedure computes $\Gig(C)$ by a~similar argument to the case of $(-1,-1)$-curves. We~will again use neighbouring divisors to compute the socle and hence the total $G$-igsaw piece for~$C$, and then assess local divisibility relations to evidence that all these characters come from the subvarieties in~the unlocking procedure. We~will sketch the novel elements of the proof below.

Choose coordinates so that $C$ lies along a~straight line from $e_1$. Assume for the moment that the edge for~$C$ is actually incident to $e_1$.

Suppose that $D$ is a~Hirzebruch divisor along the $\chi$-chain. If $D$ is at the boundary of two $e_1$-corner triangles or an $e_1$-corner triangle and a~meeting of champions~-- as shown in~Fig.~\ref{fig:twoe1}~-- then one can check that $r_\chi$ divides the monomials in~the $G$-igsaw pieces for the Type \texttt{Iy} curves~$C_3$ and $C_4$.

\begin{figure}[h]\centering
\begin{tikzpicture}
\node (a) at (0,4){$\bullet$};
\node[fill=white,draw] (b) at (0,2){\small $D$};
\node (c) at (0,0){$\bullet$};
\node (d) at (-2,2){$\bullet$};
\node (e) at (2,2){$\bullet$};
\node (f) at (-2.5,0.8){$\bullet$};
\node (g) at (1.5,0.8){$\bullet$};

\small

\draw (a.center) to node[fill=white,inner sep=2pt] {$\chi$} (b)
to node[fill=white,inner sep=2pt] {$\chi$} (c.center);
\draw[dashed] (b) to node[fill=white,inner sep=2pt] {$C_1$} (d.center);
\draw[dashed] (b) to node[fill=white,inner sep=2pt] {$C_2$} (e.center);
\draw[dashed] (b) to node[fill=white,inner sep=2pt] {$C_3$} (f.center);
\draw[dashed] (b) to node[fill=white,inner sep=2pt] {$C_4$} (g.center);
\end{tikzpicture}
\caption{$D$ bordering two $e_1$-corner triangles or meeting of champions.}\label{fig:twoe1}
\end{figure}
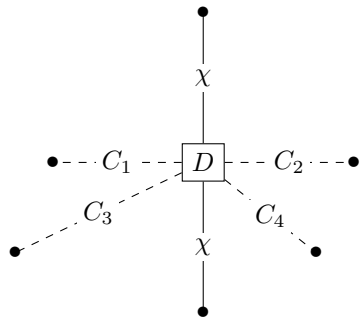

Suppose now that $D$ borders an $e_2$- and an $e_3$-corner triangle, or an $e_1$-corner triangle and an~$e_3$-corner triangle. We~illustrate this situation in~Fig.~\ref{fig:e2e3}, along with some of the ratios marking curves.

\begin{figure}[h]\centering
\begin{tikzpicture}[scale=1.25]
\node (a) at (0,4){$\bullet$};
\node[fill=white,draw] (b) at (0,2){\small $D$};
\node (c) at (0,0){$\bullet$};
\node (d) at (-2,1.5){$\bullet$};
\node (e) at (2.5,1.5){$\bullet$};
\node (f) at (-2.5,0.6){$\bullet$};
\node (g) at (2,0.6){$\bullet$};

\small

\draw (a.center) to node[fill=white,inner sep=2pt] {$\chi$} (b) to node[fill=white,inner sep=2pt] {$\chi$} (c.center);
\draw[dashed] (b) to node[fill=white,inner sep=2pt] {$C_1$} (d.center);
\draw[dashed] (b) to node[fill=white,inner sep=2pt] {$C_2$} (e.center);
\draw[dashed] (b) to node[fill=white,inner sep=2pt] {$C_3$} (f.center);
\draw[dashed] (b) to node[fill=white,inner sep=2pt] {$C_4$} (g.center);

\node (a) at (7,4){$\bullet$};
\node[fill=white,draw] (b) at (7,2){$D$};
\node (c) at (7,0){$\bullet$};
\node (d) at (4.7,1.5){$\bullet$};
\node (e) at (9.5,1.5){$\bullet$};
\node (f) at (4.5,0.6){$\bullet$};
\node (g) at (9,0.6){$\bullet$};

\draw (b) to node[fill=white,inner sep=2pt] {$C_3$} (f.center);
\draw (b) to node[fill=white,inner sep=2pt] {$C_4$} (g.center);
\draw (a.center) to node[fill=white,inner sep=2pt] {$z^f:y^c$} (b) to node[fill=white,inner sep=2pt] {$z^f:y^c$} (c.center);
\draw (b) to node[fill=white,inner sep=2pt] {$x^{d-i}:y^{b+i}z^i$} (d.center);
\draw (b) to node[fill=white,inner sep=2pt] {$x^{d-i}:z^{h+i'}y^{i'}$} (e.center);
\end{tikzpicture}
\caption{$D$ bordering an $e_1$- or an $e_2$-corner triangle and an $e_3$-corner triangle.}\label{fig:e2e3}
\end{figure}
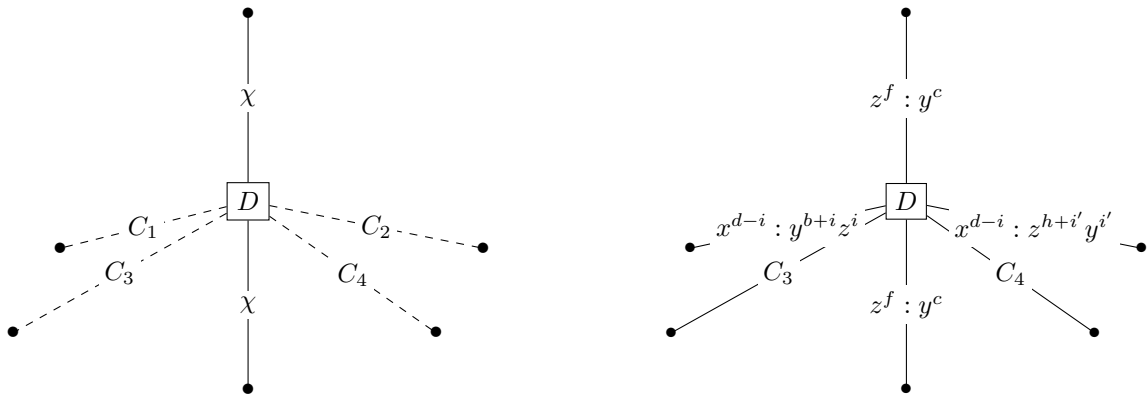

The same argument as in~the previous case gives that $r_\chi$ divides the $G$-igsaw pieces for~$C_3$ and $C_4$.

To treat the remaining two curves $C_1$ and $C_2$ in~each case, we use a~generalised form of \cite[Section~3.3.2]{CrawReid02}: an edge $\ell$ continues in~a~straight line past a~boundary edge $\ell_0$ if and only if the ratio marking $\ell$ features any common variables $x$, $y$, $z$ raised to a~strictly lower exponent than in~the ratio marking $\ell_0$. One can verify this by a~case-by-case analysis using as its base the original result from~\cite{CrawReid02}. This implies that $r_\chi$ divides the $G$-igsaw pieces for ``broken edges'' that do not continue in~a~straight line past the $\chi$-chain and that it does not divide any monomials in~the $G$-igsaw pieces for ``straight edges'' that do continue past the $\chi$-chain. This is captured exactly in~the notion of downstream curves relative to boundary curves in~Definition \ref{def:bdry_downstream}, and hence in~the unlocking procedure.

For the case when $C$ is not incident to $e_1$, consider two boundary curves $C$ and $C'$ shown in~Fig.~\ref{fig:twoc}, where $C$ is closer to $e_1$. One can verify using local divisibility relations that the only difference between the $G$-igsaw piece for~$C$ and for~$C'$ is that the latter loses the characters in~the $G$-igsaw pieces for the dashed curves in~broken chains; that is, exactly the curves that are downstream from $C$ but not from $C'$. By retracing back to the edge incident to $e_1$ the first calculation performed above suffices to compute the $G$-igsaw piece for an arbitrary boundary curve.

\begin{figure}[h]\centering
\begin{tikzpicture}[scale=0.9]
\node (a) at (0,4){$\bullet$};
\node (b) at (0,2){$\bullet$};
\node (c) at (0,0){$\bullet$};
\node (d) at (-2,1.5){$\bullet$};
\node (e) at (3,1.5){$\bullet$};
\node (f) at (-3,0.8){$\bullet$};
\node (g) at (2,0.8){$\bullet$};

\draw (a.center) to node[fill=white] {\small $C$} (b.center) to node[fill=white] {\small $C'$} (c.center);
\draw[dashed] (b.center) to (d.center);
\draw[dashed] (b.center) to (e.center);
\draw[dashed] (b.center) to (f.center);
\draw[dashed] (b.center) to (g.center);
\end{tikzpicture}
\caption{Two boundary curves.}\label{fig:twoc}
\end{figure}
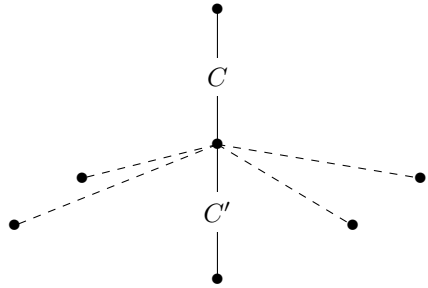

Variations of the arguments above work just as well for the cases not depicted when some of the edges incident to $D$ are also boundary edges of regular triangles. Counting up all these monomials and comparing them with a~socle calculation shows that these are all the characters in~the $G$-igsaw piece for~$C$, which validates the unlocking procedure for boundary curves and hence for all exceptional curves in~$G$-Hilb.

\subsection[Example: G=1/30(25,2,3)]{Example: $\boldsymbol{G =\frac{1}{30}(25,2,3)}$} \label{ex:30}

We will illustrate the unlocking procedure for~$G$-Hilb in~the case that $G=\frac{1}{30}(25,2,3)$. In~the figure below, dashed lines are edges within a~regular triangle and undashed lines are the result of the first stage of the Craw--Reid triangulation.

\begin{figure}[h]\centering
\begin{tikzpicture}[scale=0.59]
\node (e1) at (0,9){$\bullet$};
\node (e2) at (10,-6){$\bullet$};
\node (e3) at (-10,-6){$\bullet$};

\draw (e1.center) to (e2.center) to (e3.center) to (e1.center);

\node (xy1) at (10/3,4){$\bullet$};
\node (xy2) at (20/3,-1){$\bullet$};

\node (xz1) at (-5,1.5){$\bullet$};

\node (yz1) at (-10+4,-6){$\bullet$};
\node (yz2) at (-10+8,-6){$\bullet$};
\node (yz3) at (-10+12,-6){$\bullet$};
\node (yz4) at (-10+16,-6){$\bullet$};

\footnotesize

\node[draw, fill=white,inner sep=2pt] (a1) at (-2/6,9-5/2){$1$};
\node[draw, fill=white,inner sep=2pt] (a2) at (-4/6,9-10/2){$26$};
\node[draw, fill=white,inner sep=2pt] (a3) at (-6/6,9-15/2){$21$};
\node[draw, fill=white,inner sep=2pt] (a4) at (-8/6,9-20/2){$16$};
\node[draw, fill=white,inner sep=2pt] (a5) at (-10/6,9-25/2){$11$};

\draw (e1.center)
to node[fill=white,inner sep=2pt] {$6$} (a1)
to node[fill=white,inner sep=2pt] {$6$} (a2)
to node[fill=white,inner sep=2pt] {$6$} (a3)
to node[fill=white,inner sep=2pt] {$6$} (a4)
to node[fill=white,inner sep=2pt] {$6$} (a5)
to node[fill=white,inner sep=2pt] {$6$} (yz2.center);

\node[draw, fill=white,inner sep=2pt] (iz1) at (-16/3,-1){$23$};
\node[draw, fill=white,inner sep=2pt] (iz2) at (-17/3,-3.5){$13$};

\draw[dashed] (xz1.center)
to node[fill=white,inner sep=2pt] {$3$} (iz1)
to node[fill=white,inner sep=2pt] {$3$} (iz2)
to node[fill=white,inner sep=2pt] {$3$} (yz1.center);
\draw[dashed] (xz1.center)
to node[fill=white,inner sep=2pt] {$25$} (a1);
\draw[dashed] (iz1)
to node[fill=white,inner sep=2pt] {$28$} (a1);

\draw[dashed] (iz1) to node[fill=white,inner sep=2pt] {$15$} (a3);
\draw[dashed] (iz2) to node[fill=white,inner sep=2pt] {$18$} (a3);

\draw[dashed] (iz2) to node[fill=white,inner sep=2pt] {$5$} (a5);
\draw[dashed] (yz1.center) to node[fill=white,inner sep=2pt] {$8$} (a5);

\node[draw, fill=white,inner sep=2pt] (c1) at (-1+11/3,-1){$19$};
\node[draw, fill=white,inner sep=2pt] (c2) at (-1+22/3,-3.5){$17$};

\node[draw, fill=white,inner sep=2pt] (dp1) at (9/3,1.5){$\begin{matrix} 29 \\ 22\end{matrix}$};
\node[draw, fill=white,inner sep=2pt] (dp2) at (7/3,-3.5){$\begin{matrix} 14 \\ 7\end{matrix}$};

\draw(e3.center) to node[fill=white,inner sep=2pt] {$20$} (iz1)
to node[fill=white,inner sep=2pt] {$20$} (a2);
\draw(e3.center) to node[fill=white,inner sep=2pt] {$10$} (iz2)
to node[fill=white,inner sep=2pt] {$10$} (a4);
\draw(e2.center) to node[fill=white,inner sep=2pt] {$15$} (c2)
to node[fill=white,inner sep=2pt] {$15$} (c1)
to node[fill=white,inner sep=2pt] {$15$} (a3);

\draw[dashed] (xy2.center) to node[fill=white,inner sep=2pt] {$2$} (c2) to node[fill=white,inner sep=2pt] {$2$} (yz4.center);
\draw[dashed] (xy1.center) to node[fill=white,inner sep=2pt] {$4$} (dp1) to node[fill=white,inner sep=2pt] {$4$} (c1) to node[fill=white,inner sep=2pt] {$4$} (dp2) to node[fill=white,inner sep=2pt] {$4$} (yz3.center);
\draw[dashed] (a5) to node[fill=white,inner sep=2pt] {$5$} (dp2) to node[fill=white,inner sep=2pt] {$5$} (c2);
\draw[dashed] (a5) to node[fill=white,inner sep=2pt] {$9$} (yz3.center);
\draw[dashed] (a4) to node[fill=white,inner sep=2pt] {$12$} (dp2) to node[fill=white,inner sep=2pt] {$12$} (yz4.center);
\draw[dashed] (a4) to node[fill=white,inner sep=2pt] {$10$} (c1);
\draw[dashed] (a2) to node[fill=white,inner sep=2pt] {$20$} (dp1) to node[fill=white,inner sep=2pt] {$20$} (xy2.center);
\draw[dashed] (a1) to node[fill=white,inner sep=2pt] {$25$} (xy1.center);
\draw[dashed] (a1) to node[fill=white,inner sep=2pt] {$27$} (dp1) to node[fill=white,inner sep=2pt] {$27$} (c2);
\draw[dashed] (a2) to node[fill=white,inner sep=2pt] {$24$} (c1);
\end{tikzpicture}
\caption{Reid's recipe for~$G=\frac{1}{30}(25,2,3)$.}\label{fig:30rr}
\end{figure}
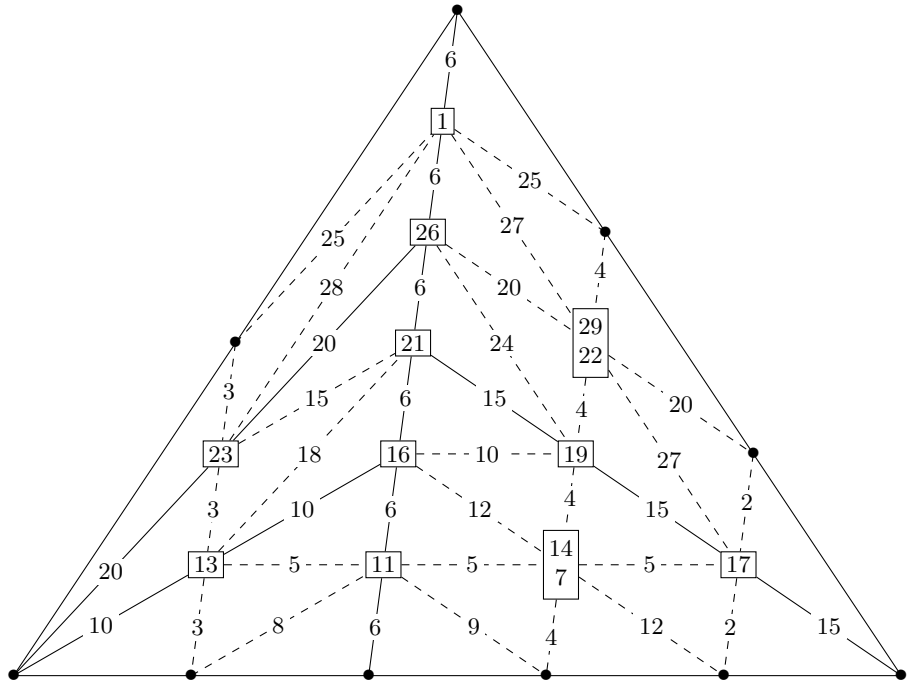

We will demonstrate the unlocking procedure for a~few curves in~$G$-Hilb. Consider the $15$-curve $C_{15}$ shown in~Fig.~\ref{fig:15c}. This curve is of Type \texttt{Ixb} since it is the only $(-1,-1)$-curve marked with $15$ and feeds to the right into boundary edges only. This gives
\begin{gather*}
\Gig(C_{15})=\{15,17,19,21\}.
\end{gather*}

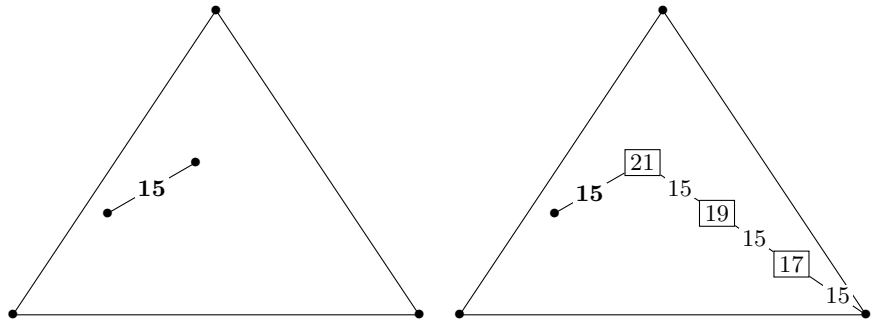
\begin{figure}[h]\centering
\begin{tikzpicture}[scale=0.27]
\footnotesize
\node (e1) at (0,9){$\bullet$};
\node (e2) at (10,-6){$\bullet$};
\node (e3) at (-10,-6){$\bullet$};

\draw (e1.center) to (e2.center) to (e3.center) to (e1.center);

\node (a3) at (-6/6,9-15/2){$\bullet$};
\node (iz1) at (-16/3,-1){$\bullet$};

\draw (iz1.center) to node[fill=white,inner sep=1pt] {$\mathbf{15}$} (a3.center);

\node (e1) at (0+22,9){$\bullet$};
\node (e2) at (10+22,-6){$\bullet$};
\node (e3) at (-10+22,-6){$\bullet$};

\draw (e1.center) to (e2.center) to (e3.center) to (e1.center);

\node[fill=white,draw,inner sep=2pt] (a3) at (-6/6+22,9-15/2){$21$};
\node (iz1) at (-16/3+22,-1){$\bullet$};
\node[draw, fill=white,inner sep=2pt] (c1) at (-1+11/3+22,-1){$19$};
\node[draw, fill=white,inner sep=2pt] (c2) at (-1+22/3+22,-3.5){$17$};

\draw (iz1.center) to node[fill=white,inner sep=1pt] {$\mathbf{15}$} (a3);
\draw (e2.center) to node[fill=white,inner sep=1pt,outer sep=0pt] {$15$} (c2) to node[fill=white,inner sep=1pt,outer sep=0pt] {$15$} (c1) to node[fill=white,inner sep=1pt,outer sep=0pt] {$15$} (a3);
\end{tikzpicture}
\caption{Unlocking for a~$15$-curve.}\label{fig:15c}
\end{figure}

Consider the $5$-curve $C_5$ shown in~Fig.~\ref{fig:5c}. It~passes into the right side of the junior simplex, unlocking the $9$-curve of Type \texttt{Iy} and giving
\begin{gather*}
\Gig(C_5)=\{5,7,9,11\}.
\end{gather*}

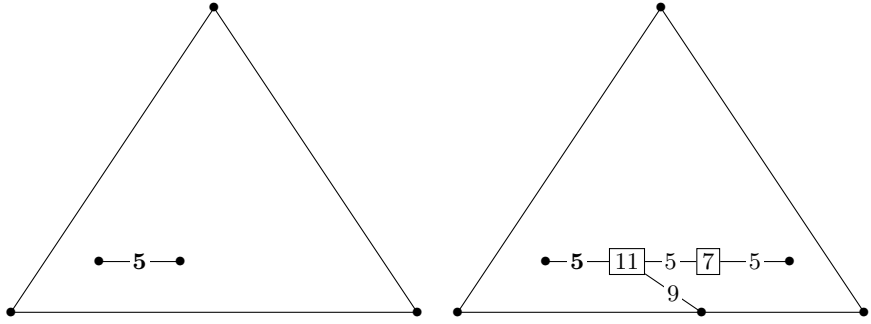
\begin{figure}[h]\centering
\begin{tikzpicture}[scale=0.27]
\footnotesize
\node (e1) at (0,9){$\bullet$};
\node (e2) at (10,-6){$\bullet$};
\node (e3) at (-10,-6){$\bullet$};

\draw (e1.center) to (e2.center) to (e3.center) to (e1.center);

\node (a5) at (-10/6,9-25/2){$\bullet$};
\node (iz2) at (-17/3,-3.5){$\bullet$};
\draw (iz2.center) to node[fill=white,inner sep=1pt] {$\mathbf{5}$} (a5.center);

\node (e1) at (0+22,9){$\bullet$};
\node (e2) at (10+22,-6){$\bullet$};
\node (e3) at (-10+22,-6){$\bullet$};

\draw (e1.center) to (e2.center) to (e3.center) to (e1.center);

\node[draw, fill=white,inner sep=2pt] (a5) at (-10/6+22,9-25/2){$11$};
\node (yz3) at (-10+12+22,-6){$\bullet$};
\node (iz2) at (-17/3+22,-3.5){$\bullet$};
\node (c2) at (-1+22/3+22,-3.5){$\bullet$};
\node[draw, fill=white,inner sep=2pt] (dp2) at (7/3+22,-3.5){$7$};

\draw (iz2.center) to node[fill=white,inner sep=1pt] {$\mathbf{5}$} (a5);
\draw (a5) to node[fill=white,inner sep=1pt] {$5$} (dp2) to node[fill=white,inner sep=1pt] {$5$} (c2.center);
\draw (a5) to node[fill=white,inner sep=1pt] {$9$} (yz3.center);
\end{tikzpicture}
\caption{Unlocking for a~$5$-curve.}\label{fig:5c}
\end{figure}

Consider the $2$-curve $C_2$ shown in~Fig.~\ref{fig:2c}. This is a~curve of Type \texttt{Iz}. We~first get the character $17$ marking the divisor on the $2$-chain, unlocking the $27$-chain. The $27$-chain contains a~del Pezzo divisor contributing the character $22$ in~this case. Hence
\begin{gather*}
\Gig(C_2)=\{2,17,22,27\}.
\end{gather*}

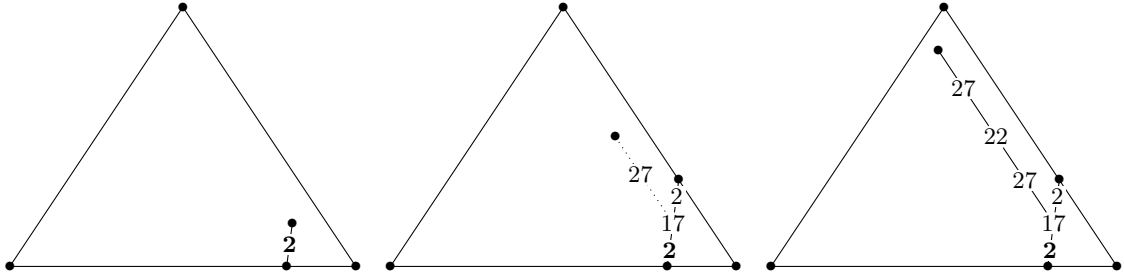
\begin{figure}[h]\centering
\begin{tikzpicture}[scale=0.23]
\footnotesize
\node (e1) at (0,9){$\bullet$};
\node (e2) at (10,-6){$\bullet$};
\node (e3) at (-10,-6){$\bullet$};

\draw (e1.center) to (e2.center) to (e3.center) to (e1.center);

\node (xy2) at (20/3+22,-1){$\bullet$};
\node (yz4) at (-10+16,-6){$\bullet$};
\node (c2) at (-1+22/3,-3.5){$\bullet$};
\draw (c2.center) to node[fill=white,inner sep=1pt] {$\mathbf{2}$} (yz4.center);

\node (e1) at (0+22,9){$\bullet$};
\node (e2) at (10+22,-6){$\bullet$};
\node (e3) at (-10+22,-6){$\bullet$};

\draw (e1.center) to (e2.center) to (e3.center) to (e1.center);

\node (yz4) at (-10+16+22,-6){$\bullet$};
\node[fill=white,inner sep=1pt] (c2) at (-1+22/3+22,-3.5){$17$};
\node (dp1) at (9/3+22,1.5){$\bullet$};
\draw (xy2.center) to node[fill=white,inner sep=1pt] {$2$} (c2) to node[fill=white,inner sep=1pt] {$\mathbf{2}$} (yz4.center);
\draw[dotted] (dp1.center) to node[fill=white,inner sep=1pt] {$27$} (c2);

\node (e1) at (0+44,9){$\bullet$};
\node (e2) at (10+44,-6){$\bullet$};
\node (e3) at (-10+44,-6){$\bullet$};

\draw (e1.center) to (e2.center) to (e3.center) to (e1.center);

\node (a1) at (-2/6+44,9-5/2){$\bullet$};
\node (xy2) at (20/3+44,-1){$\bullet$};
\node (yz4) at (-10+16+44,-6){$\bullet$};
\node[fill=white,inner sep=1pt] (c2) at (-1+22/3+44,-3.5){$17$};
\node[fill=white,inner sep=1pt] (dp1) at (9/3+44,1.5){$22$};

\draw (xy2.center) to node[fill=white,inner sep=1pt] {$2$} (c2)
to node[fill=white,inner sep=1pt] {$\mathbf{2}$} (yz4.center);
\draw (a1.center) to node[fill=white,inner sep=1pt] {$27$} (dp1);
\draw (dp1) to node[fill=white,inner sep=1pt] {$27$} (c2);
\end{tikzpicture}
\caption{Unlocking for a~$2$-curve.}\label{fig:2c}
\end{figure}

Lastly, we will consider the boundary $15$-curve $C_{15}'$ shown in Fig.~\ref{fig:15c'}.

\begin{figure}[h]\centering
\begin{tikzpicture}[scale=0.244]
\footnotesize
\node (e1) at (0,9){$\bullet$};
\node (e2) at (10,-6){$\bullet$};
\node (e3) at (-10,-6){$\bullet$};

\draw (e1.center) to (e2.center) to (e3.center) to (e1.center);

\node (c1) at (-1+11/3,-1){$\bullet$};
\node (c2) at (-1+22/3,-3.5){$\bullet$};
\draw (c2.center) to node[fill=white,inner sep=1pt] {$\mathbf{15}$} (c1.center);

\node (e1) at (0+22,9){$\bullet$};
\node (e2) at (10+22,-6){$\bullet$};
\node (e3) at (-10+22,-6){$\bullet$};

\draw (e1.center) to (e2.center) to (e3.center) to (e1.center);

\node (a2) at (-4/6+22,9-10/2){$\bullet$};
\node[fill=white,draw,inner sep=1.5pt] (a3) at (-6/6+22,9-15/2){$21$};
\node (a4) at (-8/6+22,9-20/2){$\bullet$};
\node (iz1) at (-16/3+22,-1){$\bullet$};
\node (iz2) at (-17/3+22,-3.5){$\bullet$};
\node[draw, fill=white,inner sep=1.5pt] (c1) at (-1+11/3+22,-1){$19$};
\node[draw,fill=white,inner sep=1.5pt] (c2) at (-1+22/3+22,-3.5){$17$};

\draw (iz1.center) to node[fill=white,inner sep=1pt,outer sep=0pt] {$15$} (a3);
\draw[dotted] (iz2.center) to node[fill=white,inner sep=1pt,outer sep=0pt] {$18$} (a3);
\draw (e2.center) to node[fill=white,inner sep=0pt,outer sep=0pt] {$15$} (c2)
to node[fill=white,inner sep=0pt,outer sep=0pt] {$\mathbf{15}$} (c1)
to node[fill=white,inner sep=0pt,outer sep=0pt] {$15$} (a3);
\draw[dotted] (a4.center) to node[fill=white,inner sep=1pt,outer sep=0pt] {$10$} (c1);
\draw[dotted] (a2.center) to node[fill=white,inner sep=0pt,outer sep=0pt] {$24$} (c1);

\node (e1) at (0+44,9){$\bullet$};
\node (e2) at (10+44,-6){$\bullet$};
\node (e3) at (-10+44,-6){$\bullet$};

\draw (e1.center) to (e2.center) to (e3.center) to (e1.center);

\node (a2) at (-4/6+44,9-10/2){$\bullet$};
\node[fill=white,draw,inner sep=2pt] (a3) at (-6/6+44,9-15/2){$21$};
\node[draw, fill=white,inner sep=2pt] (a4) at (-8/6+44,9-20/2){$16$};
\node (iz1) at (-16/3+44,-1){$\bullet$};
\node[fill = white, draw,inner sep=2pt] (iz2) at (-17/3+44,-3.5){$13$};
\node[draw, fill=white,inner sep=2pt] (c1) at (-1+11/3+44,-1){$19$};
\node[draw,fill=white,inner sep=2pt] (c2) at (-1+22/3+44,-3.5){$17$};

\draw (iz1.center) to node[fill=white,inner sep=1pt,outer sep=0pt] {$15$} (a3);
\draw (iz2) to node[fill=white,inner sep=1pt,outer sep=0pt] {$18$} (a3);
\draw(e3.center) to node[fill=white,inner sep=1pt,outer sep=0pt] {$10$} (iz2)
to node[fill=white,inner sep=1pt,outer sep=0pt] {$10$} (a4);
\draw (e2.center) to node[fill=white,inner sep=1pt,outer sep=0pt] {$15$} (c2)
to node[fill=white,inner sep=1pt,outer sep=0pt] {$\mathbf{15}$} (c1)
to node[fill=white,inner sep=1pt,outer sep=0pt] {$15$} (a3);
\draw[dashed] (a4) to node[fill=white,inner sep=1pt,outer sep=0pt] {$10$} (c1);
\draw (a2.center) to node[fill=white,inner sep=1pt,outer sep=0pt] {$24$} (c1);
\end{tikzpicture}
\caption{Unlocking for a~boundary $15$-curve.}\label{fig:15c'}
\end{figure}
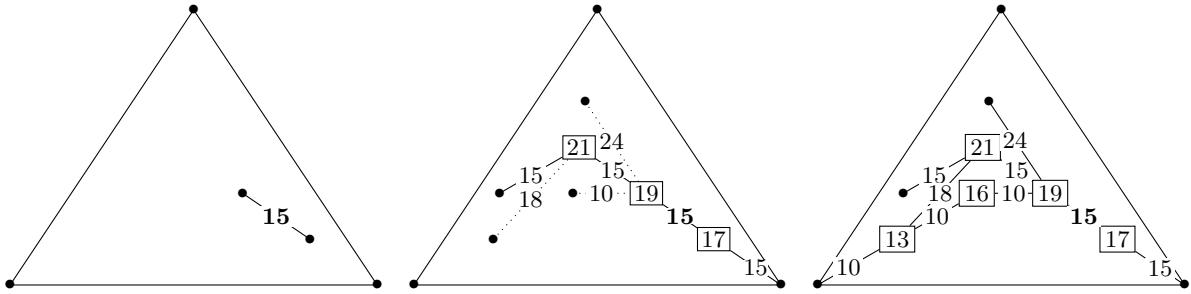

At the first step we include the $15$-chain and the curves of Type \texttt{Ix} and \texttt{Iy} unlocked by it. These curves are marked with characters 10, 24, 18. The $18$-curve and the $24$-curve are of Type~\texttt{Iy} and only contribute their own character to the $G$-igsaw piece. The $10$-curve is of Type~\texttt{Ixb} and so we add the Hirzebruch divisors along the $10$-chain. As~a~result
\begin{gather*}
\Gig(C_{15}')=\{10,13,15,16,17,18,19,21,24\}.
\end{gather*}

\vspace{-2mm}

\subsection[Example: G=1/35(1,3,31)]{Example: $\boldsymbol{G=\frac{1}{35}(1,3,31)}$} \label{ex:35}

We will use the example of $G=\frac{1}{35}(1,3,31)$ to illustrate a~phenomenon implicit, but less clear in~the long side picture. The triangulation for~$G$-Hilb and Reid's recipe are found in~Fig.~\ref{fig:1/35}.

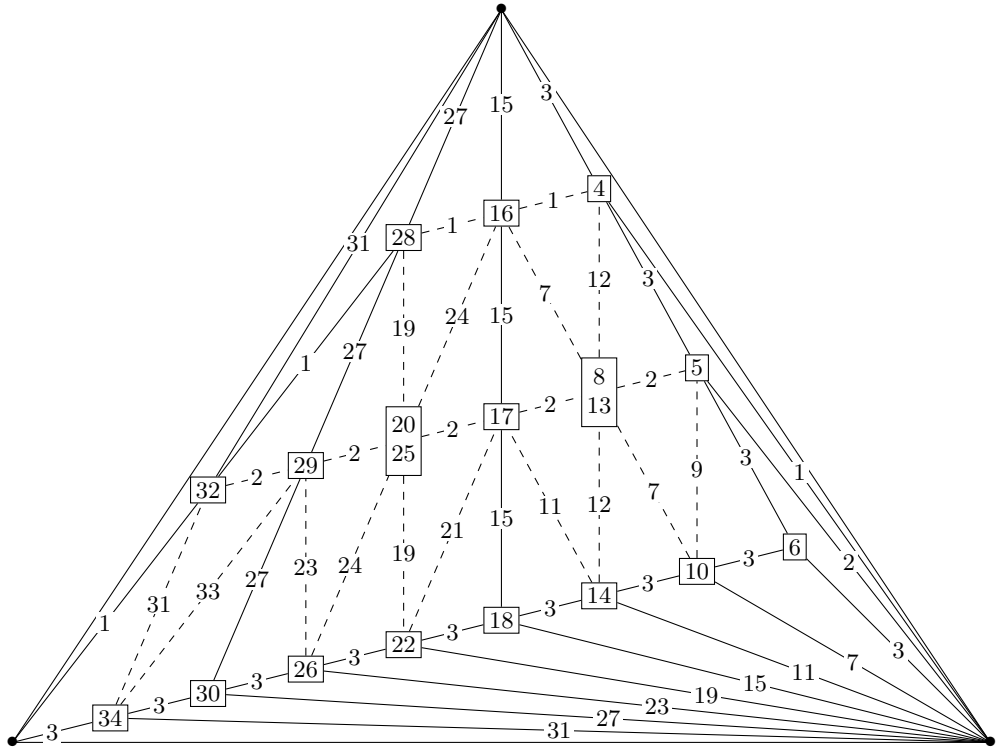
\begin{figure}[h]\centering
\begin{tikzpicture}[scale=0.65]
\small
\node (e1) at (0,9){$\bullet$};
\node (e2) at (10,-6){$\bullet$};
\node (e3) at (-10,-6){$\bullet$};

\draw (e1.center) to (e2.center) to (e3.center) to (e1.center);

\footnotesize

\node[draw, fill=white,inner sep=2pt] (a1) at (-10+2,-6+1/2){$34$};
\node[draw, fill=white,inner sep=2pt] (a2) at (-10+4,-6+2/2){$30$};
\node[draw, fill=white,inner sep=2pt] (a3) at (-10+6,-6+3/2){$26$};
\node[draw, fill=white,inner sep=2pt] (a4) at (-10+8,-6+4/2){$22$};
\node[draw, fill=white,inner sep=2pt] (a5) at (-10+10,-6+5/2){$18$};
\node[draw, fill=white,inner sep=2pt] (a6) at (-10+12,-6+6/2){$14$};
\node[draw, fill=white,inner sep=2pt] (a7) at (-10+14,-6+7/2){$10$};
\node[draw, fill=white,inner sep=2pt] (a8) at (6,-2){$6$};

\node[draw, fill=white,inner sep=2pt] (b1) at (2,9-11/3){$4$};
\node[draw, fill=white,inner sep=2pt] (b2) at (4,9-22/3){$5$};

\node[draw, fill=white,inner sep=2pt] (c1) at (-10+6-2,-6+3/2+11/3){$32$};
\node[draw, fill=white,inner sep=2pt] (c2) at (-10+8-2,-6+4/2+11/3){$29$};
\node[draw, fill=white,inner sep=2pt] (c3) at (-10+10-2,-6+5/2+11/3){$\begin{matrix} 20 \\ 25\end{matrix}$};
\node[draw, fill=white,inner sep=2pt] (c4) at (-10+12-2,-6+6/2+11/3){$17$};
\node[draw, fill=white,inner sep=2pt] (c5) at (-10+14-2,-6+7/2+11/3){$\begin{matrix} 8 \\ 13\end{matrix}$};

\node[draw, fill=white,inner sep=2pt] (d1) at (2-4,9-2/2-11/3){$28$};
\node[draw, fill=white,inner sep=2pt] (d2) at (2-2,9-1/2-11/3){$16$};

\draw(e2.center) to node[fill=white,inner sep=1pt] {$31$} (a1);
\draw(e2.center) to node[fill=white,inner sep=1pt] {$27$} (a2);
\draw(e2.center) to node[fill=white,inner sep=1pt] {$23$} (a3);
\draw(e2.center) to node[fill=white,inner sep=1pt] {$19$} (a4);
\draw(e2.center) to node[fill=white,inner sep=1pt] {$15$} (a5);
\draw(e2.center) to node[fill=white,inner sep=1pt] {$11$} (a6);
\draw(e2.center) to node[fill=white,inner sep=1pt] {$7$} (a7);
\draw(e2.center) to node[fill=white,inner sep=1pt] {$3$} (a8);

\draw(e2.center) to node[fill=white,inner sep=1pt] {$1$} (b1);
\draw(e2.center) to node[fill=white,inner sep=1pt] {$2$} (b2);

\draw(e3.center) to node[fill=white,inner sep=1pt] {$1$} (c1) to node[fill=white,inner sep=1pt] {$1$} (d1);

\draw(e1.center) to node[fill=white,inner sep=1pt] {$31$} (c1);
\draw(e1.center) to node[fill=white,inner sep=1pt] {$27$} (d1)
to node[fill=white,inner sep=1pt] {$27$} (c2)
to node[fill=white,inner sep=1pt] {$27$} (a2);
\draw(e1.center) to node[fill=white,inner sep=1pt] {$15$} (d2)
to node[fill=white,inner sep=1pt] {$15$} (c4)
to node[fill=white,inner sep=1pt] {$15$} (a5);
\draw (e1.center) to node[fill=white,inner sep=1pt] {$3$} (b1)
to node[fill=white,inner sep=1pt] {$3$} (b2)
to node[fill=white,inner sep=1pt] {$3$} (a8);

\draw (e3.center) to node[fill=white,inner sep=1pt] {$3$} (a1)
to node[fill=white,inner sep=1pt] {$3$} (a2)
to node[fill=white,inner sep=1pt] {$3$} (a3)
to node[fill=white,inner sep=1pt] {$3$} (a4)
to node[fill=white,inner sep=1pt] {$3$} (a5)
to node[fill=white,inner sep=1pt] {$3$} (a6)
to node[fill=white,inner sep=1pt] {$3$} (a7)
to node[fill=white,inner sep=1pt] {$3$} (a8);

\draw[dashed] (a1) to node[fill=white,inner sep=1pt] {$31$} (c1);
\draw[dashed] (a1) to node[fill=white,inner sep=1pt] {$33$} (c2);
\draw[dashed] (c1) to node[fill=white,inner sep=1pt] {$2$} (c2);

\draw[dashed] (d1) to node[fill=white,inner sep=1pt] {$1$} (d2);
\draw[dashed] (a3) to node[fill=white,inner sep=1pt] {$23$} (c2);
\draw[dashed] (a4) to node[fill=white,inner sep=1pt] {$21$} (c4);
\draw[dashed] (d1) to node[fill=white,inner sep=1pt] {$19$} (c3)
to node[fill=white,inner sep=1pt] {$19$} (a4);
\draw[dashed] (c2) to node[fill=white,inner sep=1pt] {$2$} (c3)
to node[fill=white,inner sep=1pt] {$2$} (c4);
\draw[dashed] (a3) to node[fill=white,inner sep=1pt] {$24$} (c3)
to node[fill=white,inner sep=1pt] {$24$} (d2);

\draw[dashed] (d2) to node[fill=white,inner sep=1pt] {$1$} (b1);
\draw[dashed] (a7) to node[fill=white,inner sep=1pt] {$9$} (b2);
\draw[dashed] (a6) to node[fill=white,inner sep=1pt] {$11$} (c4);
\draw[dashed] (a6) to node[fill=white,inner sep=1pt] {$12$} (c5) to node[fill=white,inner sep=1pt] {$12$} (b1);
\draw[dashed] (c4) to node[fill=white,inner sep=1pt] {$2$} (c5) to node[fill=white,inner sep=1pt] {$2$} (b2);
\draw[dashed] (a7) to node[fill=white,inner sep=1pt] {$7$} (c5) to node[fill=white,inner sep=1pt] {$7$} (d2);
\end{tikzpicture}
\caption{Reid's recipe for~$G=\frac{1}{35}(1,3,31)$.}\label{fig:1/35}
\end{figure}

Consider the $3$-curve $C_3$ incident to $e_1$. The unlocking procedure for this curve is shown in~Fig.~\ref{fig:1/35c3} giving
\begin{gather*}\begin{split}&
\Gig(C_3)=\{1,2,3,4,5,6,8,9,10,12,13,14,16,17,18,20,21,22,24,25,26,28,29,
\\ & \phantom{\Gig(C_3)=\{}
30,32,33,34\}.\end{split}
\end{gather*}
Notice that every chain meeting the $3$-chain in~a~vertex is broken there. Repeating for the next $3$-curve along the chain produces the same unlocking sequence except that the topmost part including the $1$-chain and the $12$-chain are not included, capturing that the monomials in~the corresponding character spaces are no longer divisible by $r_3$ there.

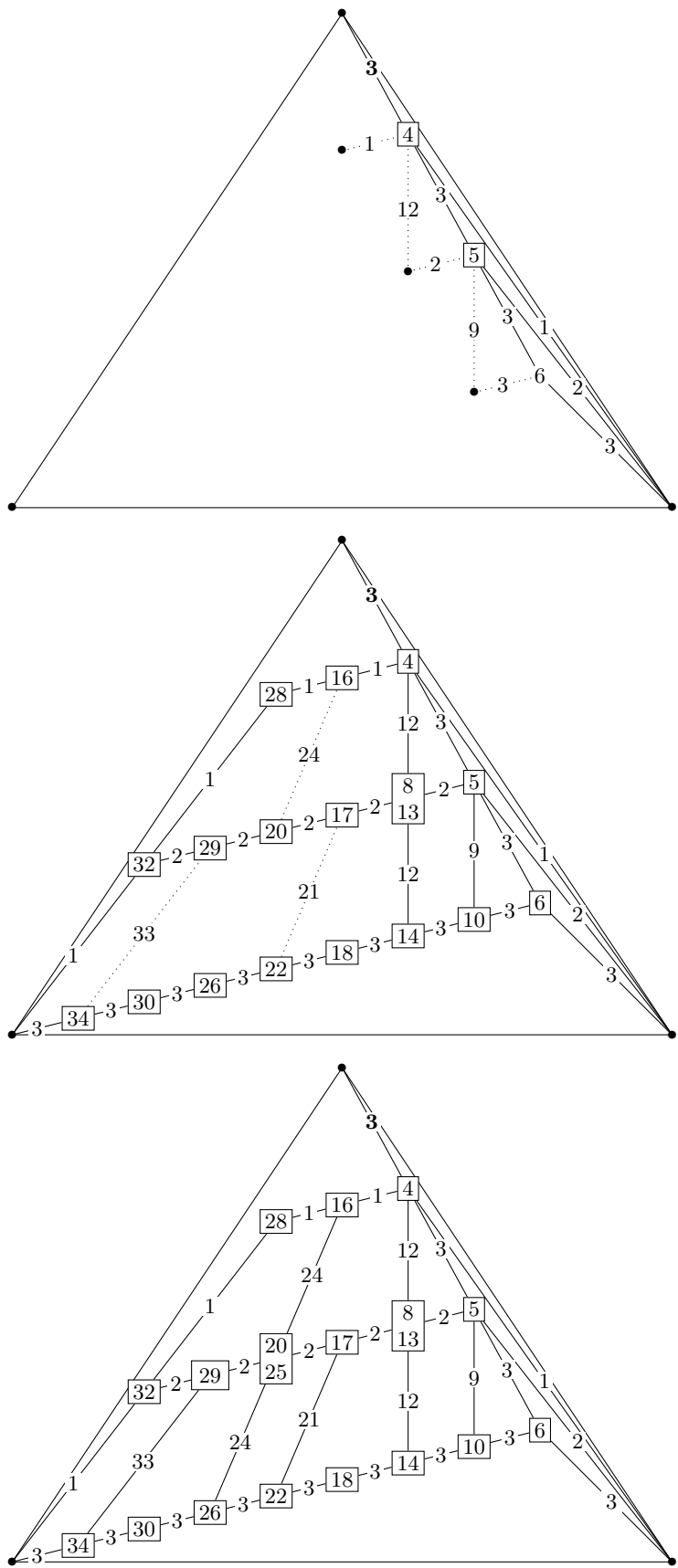
\begin{figure}\centering
\begin{tikzpicture}[scale=0.48]
\footnotesize
\node (e1) at (0,9){$\bullet$};
\node (e2) at (10,-6){$\bullet$};
\node (e3) at (-10,-6){$\bullet$};

\draw (e1.center) to (e2.center) to (e3.center) to (e1.center);

\node (a7) at (-10+14,-6+7/2){$\bullet$};
\node[fill=white,inner sep=1pt] (a8) at (6,-2){$6$};

\node[draw, fill=white,inner sep=2pt] (b1) at (2,9-11/3){$4$};
\node[draw, fill=white,inner sep=2pt] (b2) at (4,9-22/3){$5$};

\node (c5) at (-10+14-2,-6+7/2+11/3){$\bullet$};

\node (d2) at (2-2,9-1/2-11/3){$\bullet$};

\draw(e2.center) to node[fill=white,inner sep=1pt] {$3$} (a8);
\draw(e2.center) to node[fill=white,inner sep=1pt] {$1$} (b1);
\draw(e2.center) to node[fill=white,inner sep=1pt] {$2$} (b2);

\draw (e1.center) to node[fill=white,inner sep=1pt] {$\mathbf{3}$} (b1);
\draw (b1) to node[fill=white,inner sep=1pt] {$3$} (b2)
to node[fill=white,inner sep=1pt] {$3$} (a8);

\draw[dotted] (a7.center) to node[fill=white,inner sep=1pt] {$3$} (a8);

\draw[dotted] (d2.center) to node[fill=white,inner sep=1pt] {$1$} (b1);
\draw[dotted] (a7.center) to node[fill=white,inner sep=1pt] {$9$} (b2);
\draw[dotted] (c5.center) to node[fill=white,inner sep=1pt] {$12$} (b1);
\draw[dotted] (c5.center) to node[fill=white,inner sep=1pt] {$2$} (b2);


\node (e1) at (0,9-16){$\bullet$};
\node (e2) at (10,-6-16){$\bullet$};
\node (e3) at (-10,-6-16){$\bullet$};

\draw (e1.center) to (e2.center) to (e3.center) to (e1.center);

\node[draw, fill=white,inner sep=2pt] (a1) at (-10+2,-6-16+1/2){$34$};
\node[draw, fill=white,inner sep=2pt] (a2) at (-10+4,-16-6+2/2){$30$};
\node[draw, fill=white,inner sep=2pt] (a3) at (-10+6,-16-6+3/2){$26$};
\node[draw, fill=white,inner sep=2pt] (a4) at (-10+8,-16-6+4/2){$22$};
\node[draw, fill=white,inner sep=2pt] (a5) at (-10+10,-16-6+5/2){$18$};
\node[draw, fill=white,inner sep=2pt] (a6) at (-10+12,-16-6+6/2){$14$};
\node[draw, fill=white,inner sep=2pt] (a7) at (-10+14,-16-6+7/2){$10$};
\node[draw, fill=white,inner sep=2pt] (a8) at (6,-2-16){$6$};

\node[draw, fill=white,inner sep=2pt] (b1) at (2,-16+9-11/3){$4$};
\node[draw, fill=white,inner sep=2pt] (b2) at (4,9-16-22/3){$5$};

\node[draw, fill=white,inner sep=2pt] (c1) at (-10+6-2,-16-6+3/2+11/3){$32$};
\node[draw, fill=white,inner sep=2pt] (c2) at (-10+8-2,-16-6+4/2+11/3){$29$};
\node[draw, fill=white,inner sep=2pt] (c3) at (-10+10-2,-16-6+5/2+11/3){$20$};
\node[draw, fill=white,inner sep=2pt] (c4) at (-10+12-2,-16-6+6/2+11/3){$17$};
\node[draw, fill=white,align=center,inner sep=2pt] (c5) at (-10+14-2,-16-6+7/2+11/3){8 \\ 13};

\node[draw, fill=white,inner sep=2pt] (d1) at (2-4,9-16-2/2-11/3){$28$};
\node[draw, fill=white,inner sep=2pt] (d2) at (2-2,9-16-1/2-11/3){$16$};

\draw(e2.center) to node[fill=white,inner sep=1pt] {$3$} (a8);

\draw(e2.center) to node[fill=white,inner sep=1pt] {$1$} (b1);
\draw(e2.center) to node[fill=white,inner sep=1pt] {$2$} (b2);

\draw(e3.center) to node[fill=white,inner sep=1pt] {$1$} (c1)
to node[fill=white,inner sep=1pt] {$1$} (d1);

\draw (e1.center) to node[fill=white,inner sep=1pt] {$\mathbf{3}$} (b1);
\draw (b1) to node[fill=white,inner sep=1pt] {$3$} (b2)
to node[fill=white,inner sep=1pt] {$3$} (a8);

\draw (e3.center) to node[fill=white,inner sep=1pt] {$3$} (a1)
to node[fill=white,inner sep=1pt] {$3$} (a2)
to node[fill=white,inner sep=1pt] {$3$} (a3)
to node[fill=white,inner sep=1pt] {$3$} (a4)
to node[fill=white,inner sep=1pt] {$3$} (a5)
to node[fill=white,inner sep=1pt] {$3$} (a6)
to node[fill=white,inner sep=1pt] {$3$} (a7)
to node[fill=white,inner sep=1pt] {$3$} (a8);

\draw[dotted] (a1) to node[fill=white,inner sep=1pt] {$33$} (c2);
\draw (c1) to node[fill=white,inner sep=1pt] {$2$} (c2);

\draw (d1) to node[fill=white,inner sep=1pt] {$1$} (d2);
\draw[dotted] (a4) to node[fill=white,inner sep=1pt] {$21$} (c4);
\draw (c2) to node[fill=white,inner sep=1pt] {$2$} (c3)
to node[fill=white,inner sep=1pt] {$2$} (c4);
\draw[dotted] (c3) to node[fill=white,inner sep=1pt] {$24$} (d2);

\draw (d2) to node[fill=white,inner sep=1pt] {$1$} (b1);
\draw (a7) to node[fill=white,inner sep=1pt] {$9$} (b2);
\draw (a6) to node[fill=white,inner sep=1pt] {$12$} (c5);
\draw (c5) to node[fill=white,inner sep=1pt] {$12$} (b1);
\draw (c4) to node[fill=white,inner sep=1pt] {$2$} (c5);
\draw (c5) to node[fill=white,inner sep=1pt] {$2$} (b2);


\node (e1) at (0,9-17-15){$\bullet$};
\node (e2) at (10,-17-6-15){$\bullet$};
\node (e3) at (-10,-17-6-15){$\bullet$};

\draw (e1.center) to (e2.center) to (e3.center) to (e1.center);

\node[draw, fill=white,inner sep=2pt] (a1) at (-10+2,-17-6+1/2-15){$34$};
\node[draw, fill=white,inner sep=2pt] (a2) at (-10+4,-17-6+2/2-15){$30$};
\node[draw, fill=white,inner sep=2pt] (a3) at (-10+6,-17-6+3/2-15){$26$};
\node[draw, fill=white,inner sep=2pt] (a4) at (-10+8,-17-6+4/2-15){$22$};
\node[draw, fill=white,inner sep=2pt] (a5) at (-10+10,-17-6+5/2-15){$18$};
\node[draw, fill=white,inner sep=2pt] (a6) at (-10+12,-17-6+6/2-15){$14$};
\node[draw, fill=white,inner sep=2pt] (a7) at (-10+14,-17-6+7/2-15){$10$};
\node[draw, fill=white,inner sep=2pt] (a8) at (6,-17-2-15){$6$};

\node[draw, fill=white,inner sep=2pt] (b1) at (2,-17+9-11/3-15){$4$};
\node[draw, fill=white,inner sep=2pt] (b2) at (4,9-17-22/3-15){$5$};

\node[draw, fill=white,inner sep=2pt] (c1) at (-10+6-2,-17-6-15+3/2+11/3){$32$};
\node[draw, fill=white] (c2) at (-10+8-2,-17-6-15+4/2+11/3){$29$};
\node[draw, fill=white,align=center,inner sep=2pt] (c3) at (-10+10-2,-15-6+5/2-17+11/3){20 \\ 25};
\node[draw, fill=white,inner sep=2pt] (c4) at (-10+12-2,-17-6+6/2-15+11/3){$17$};
\node[draw, fill=white,align=center,inner sep=2pt] (c5) at (-10+14-2,-15-6+7/2+11/3-17){8 \\ 13};

\node[draw, fill=white,inner sep=2pt] (d1) at (2-4,9-17-2/2-11/3-15){$28$};
\node[draw, fill=white,inner sep=2pt] (d2) at (2-2,9-17-1/2-11/3-15){$16$};

\draw(e2.center) to node[fill=white,inner sep=2pt] {$3$} (a8);

\draw(e2.center) to node[fill=white,inner sep=1pt] {$1$} (b1);
\draw(e2.center) to node[fill=white,inner sep=1pt] {$2$} (b2);
\draw(e3.center) to node[fill=white,inner sep=1pt] {$1$} (c1)
to node[fill=white,inner sep=1pt] {$1$} (d1);

\draw (e1.center) to node[fill=white,inner sep=1pt] {$\mathbf{3}$} (b1);
\draw (b1) to node[fill=white,inner sep=1pt] {$3$} (b2)
to node[fill=white,inner sep=1pt] {$3$} (a8);

\draw (e3.center) to node[fill=white,inner sep=1pt] {$3$} (a1)
to node[fill=white,inner sep=1pt] {$3$} (a2)
to node[fill=white,inner sep=1pt] {$3$} (a3)
to node[fill=white,inner sep=1pt] {$3$} (a4)
to node[fill=white,inner sep=1pt] {$3$} (a5)
to node[fill=white,inner sep=1pt] {$3$} (a6)
to node[fill=white,inner sep=1pt] {$3$} (a7)
to node[fill=white,inner sep=1pt] {$3$} (a8);

\draw (a1) to node[fill=white,inner sep=1pt] {$33$} (c2);
\draw (c1) to node[fill=white,inner sep=1pt] {$2$} (c2);

\draw (d1) to node[fill=white,inner sep=1pt] {$1$} (d2);
\draw (a4) to node[fill=white,inner sep=1pt] {$21$} (c4);
\draw (c2) to node[fill=white,inner sep=1pt] {$2$} (c3)
to node[fill=white,inner sep=1pt] {$2$} (c4);
\draw (a3) to node[fill=white,inner sep=1pt] {$24$} (c3);
\draw (c3) to node[fill=white,inner sep=1pt] {$24$} (d2);

\draw (d2) to node[fill=white,inner sep=1pt] {$1$} (b1);
\draw (a7) to node[fill=white,inner sep=1pt] {$9$} (b2);
\draw (a6) to node[fill=white,inner sep=1pt] {$12$} (c5);
\draw (c5) to node[fill=white,inner sep=1pt] {$12$} (b1);
\draw (c4) to node[fill=white,inner sep=1pt] {$2$} (c5);
\draw (c5) to node[fill=white,inner sep=1pt] {$2$} (b2);
\end{tikzpicture}
\caption{Unlocking for a~$3$-curve.}\label{fig:1/35c3}
\end{figure}

\section[Walls of C0]{Walls of $\boldsymbol{\mfk{C}_0}$}

In this section we will compute explicit inequalities carving out $\mfk{C}_0$, and will determine which of these inequalities are necessary and hence define walls of $\mfk{C}_0$.

\subsection[Type I walls]{Type \texttt{I} walls}

We know from \cite[Theorem~9.12]{CrawIshii04} that all flops in~a~single $(-1,-1)$-curve $C$ are achieved by~a~wall-crossing from $\mfk{C}_0$. Moreover, we have $\deg (\mR_\rho|_C)=1$ for all $\rho\in\Gig(C)$ from \mbox{\cite[Corol\-lary~6.3]{CrawIshii04}}. The unlocking procedure hence gives a~combinatorial way of writing down the equations of these walls.

\begin{Proposition} \label{prop:i} Suppose $C\subset G\hilb\mathbb{A}^3$ is an exceptional $(-1,-1)$-curve marked with cha\-rac\-ter $\chi$ by Reid's recipe. Then, the Type \texttt{I} wall corresponding to $C$ is given by
\[
\theta(\ph_{\mfk{C}_0}(\mO_C))=\sum_{\chi\in\Gig(C)}\theta(\chi)>0,
\]
where $\Gig(C)$ is computed by the unlocking procedure.
\end{Proposition}

\subsection[No Type II walls]{No Type \texttt{II} walls}

\begin{Proposition} \label{prop:II} Suppose $C\subset G\hilb\mathbb{A}^3$ is an exceptional $(1,-3)$-curve marked with character~$\chi$ by Reid's recipe. Then, the inequality corresponding to $C$ is given by
\[
\theta\big(\ph_{\mfk{C}_0}(\mO_C)\big)=2\cdot\theta\big(\chi^{\otimes 2}\big)+\sum_{\chi\in\Gig(C)\setminus\{\chi^{\otimes 2}\}}\theta(\chi)>0,
\]
where $\Gig(C)$ is computed by the unlocking procedure.
\end{Proposition}

\begin{proof} Notice that such a~curve $C$ lies inside the exceptional $\pr^2$ in~the meeting of champions case when the meeting of champions triangle has side length $0$. Thus the $\pr^2$ is marked with $\chi^{\otimes 2}$ and lies in~the socle of any torus-invariant $G$-cluster. From Theorem~\ref{thm:craw} $r_{\chi^{\otimes 2}}=r_\chi^2$ and so $r_\chi^2$ is the furthest character from $r_\chi$ in~the $G$-igsaw piece in~some direction. Note that
\[
\deg (\mR_\rho|_C)=\min\big\{k\colon r_\chi^k\,|\, r_\rho\big\}
\]
and so all the characters in~$\Gig(C)$ appear with multiplicity $1$ except for~$r_\chi^2$, which appears with multiplicity $2$. This gives the required formula.
\end{proof}

As a~result we can immediately deduce the conclusion of \cite[Proposition~3.8]{CrawIshii04} for~$\mfk{C}_0$.

\begin{Corollary} \label{cor:ii} $\mfk{C}_0$ has no Type \texttt{II} walls.
\end{Corollary}

\begin{proof} Suppose $C$ is an exceptional $(1,-3)$-curve marked with $\chi$. From the unlocking procedure a~total $G$-igsaw piece for~$C$ consists of $\chi$, $\chi^2$, and the characters marking the (Hirzebruch) divisors along the $\chi$-chain. Let $D'$ be the exceptional $\pr^2$ containing $C$. Consider the inequality for rigid quotients parameterised by $D'$: from Proposition~\ref{prop:rigid} the characters appearing in~this inequality are exactly the characters in~the $G$-igsaw pieces of all three $\chi$-curves converging at~$D'$. These~are
\[
\big\{\chi,\chi^{\otimes 2}\big\}\cup\Hirz(\chi),
\]
which are exactly the characters appearing in~the inequality for~$C$. However, the inequality for rigid quotients parameterised by $D'$ has multiplicities all equal to $1$. When combined with the inequality $\theta\big(\chi^{\otimes 2}\big)>0$ coming from rigid subsheaves parameterised by $D'$~-- note that we can use both inequalities from subsheaves and from quotients since $D'$ is irreducible~-- this implies that the inequality $\ph_{\mfk{C}_0}(\mO_C)>0$ is redundant.
\end{proof}

\subsection[All flops in (-1,-1)-curves]{All flops in $\boldsymbol{(-1,-1)}$-curves}

Using Proposition~\ref{prop:i} and the unlocking procedure one can show directly that every $(-1,-1)$-curve produces a~necessary inequality, recovering \cite[Theorem~9.12]{CrawIshii04} by purely combinatorial means.

In order to test redundancy of inequalities we say that an inequality $\sum_i\alpha_i\theta(\chi_i)>0$ with nonnegative coefficients is a~\textit{summand} of another inequality $\sum_j\beta_j\theta(\rho_j)>0$ with nonnegative coefficients if the difference $\sum_i\alpha_i\theta(\chi_i)-\sum_j\beta_j\theta(\rho_j)$ also has nonnegative coefficients in~the basis~$\on{Irr}{G}$. If an inequality coming from curves or divisors decomposes into other inequalities as summands, then it is redundant and does not define a~wall of $\mfk{C}_0$.

\begin{Proposition} \label{prop:-1nec} Suppose $C$ is an exceptional $(-1,-1)$ curve inside $G\hilb\A^3$. Then the inequality $\theta(\ph_{\mfk{C}_0}(\mO_C))>0$ is necessary and so defines a~wall of $\mfk{C}_0$.
\end{Proposition}

\begin{proof} Suppose $C$ is marked with $\chi$. From the unlocking procedure we can write the inequality corresponding to $C$ in~the form
\begin{gather} \label{eqn:wc1} 
 \theta(\chi)+\sum_i\theta(\psi_i)+\theta(\rho_1)+\sum_i\theta\big(\psi_i^1\big)
 +\dots+\theta(\rho_m)+\sum_i\theta\big(\psi^m_i\big)>0,
\end{gather}
where $\rho_j$ are the characters marking curves $C_j$ unlocked by $C$ and $\psi_i^j$ are the characters in~the $G$-igsaw piece for~$C_j$. Note that curves unlocked by $C$ cannot continue on both sides of the $\chi$-chain, since they meet the $\chi$-chain at a~Hirzebruch divisor found at the intersection of the $\chi$-chain and an edge of a regular triangle, where only two chains can continue. The inequality for the $(-1,-1)$-curve $C_j$ is
\begin{gather*}
\theta\big(\ph_{\mfk{C}_0}(\mO_{C_j})\big)=\theta(\rho_j)+\sum_i\theta\big(\psi_i^j\big)>0.
\end{gather*}
In order to express (\ref{eqn:wc1}) in~terms of other inequalities, we must have an inequality featuring the character $\chi$. These can only arise from other $\chi$-curves or divisors parameterising rigid quotients not featuring $\chi$. Other $\chi$-curves will feature at least one different character in~their $G$-igsaw piece compared to $\Gig(C)$: indeed, other curves in~the same regular triangle will feature a~different collection of del Pezzo divisors, curves in~other regular triangles will either feature different del Pezzo divisors or unlock different curves, and $\chi$-curves along a~boundary edge will have different unlocking behaviour. In~particular, the inequalities from these curves will not be summands of the inequality (\ref{eqn:wc1}). Inequalities from rigid quotients not containing $\chi$ will also not be summands of (\ref{eqn:wc1}) since the unlocking procedure implies that there are no divisors $D_\rho$ along the $\chi$-chain for which all characters marking curves incident to $D_\rho$ are represented in~$\Gig(C)$. It~follows that (\ref{eqn:wc1}) is necessary.
\end{proof}

Proposition~\ref{prop:-1nec} has an analog for Type \texttt{III} walls in~Lemma~\ref{lem:bdry2} where we classify the $(0,-2)$-curves producing those walls in~terms of explicit combinatorics.

\subsection{Irredundant inequalities~-- examples}

The aim of these final sections is to precisely describe all of the walls of $\mfk{C}_0$, which primarily means identifying which curves produce redundant inequalities and subsequently classifying the~walls of Type \texttt{III}. We~start with an example.

\begin{Example} \label{ex:6walls} Consider $G=\frac{1}{6}(1,2,3)$. $G$-Hilb and Reid's recipe are shown in~Fig.~\ref{fig:1/6b}.
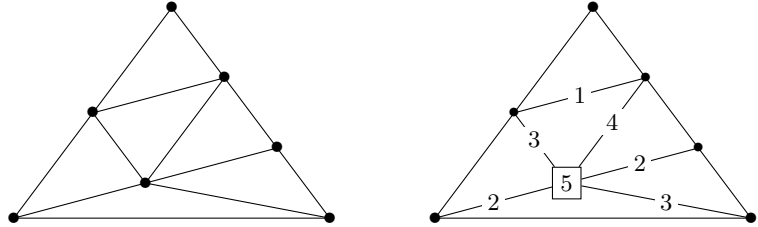
\begin{figure}[h]\centering
\begin{tikzpicture}[scale=0.7]
\node(e1) at (0,2){$\bullet$};
\node(e2) at (3,-2){$\bullet$};
\node(e3) at (-3,-2){$\bullet$};

\node(123) at (-0.5,-4/3){$\bullet$};
\node(240) at (2,2-8/3){$\bullet$};
\node(420) at (1,2-4/3){$\bullet$};
\node(303) at (-1.5,0){$\bullet$};

\draw[-] (e1.center) to (e2.center) to (e3.center) to (e1.center);

\draw[-] (e2.center) to (123.center);
\draw[-] (e3.center) to (240.center);
\draw[-] (303.center) to (420.center);
\draw[-] (123.center) to (303.center);
\draw[-] (123.center) to (420.center);

\node(e1) at (0+8,2){$\bullet$};
\node(e2) at (3+8,-2){$\bullet$};
\node(e3) at (-3+8,-2){$\bullet$};

\footnotesize

\node[draw,fill=white] (123) at (-0.5+8,-4/3){$5$};
\node(240) at (2+8,2-8/3){$\bullet$};
\node(420) at (1+8,2-4/3){$\bullet$};
\node(303) at (-1.5+8,0){$\bullet$};

\draw[-] (e2.center) to node[fill=white,inner sep=2pt] {$3$} (123);
\draw[-] (e3.center) to node[fill=white,inner sep=2pt] {$2$} (123) to node[fill=white,inner sep=2pt] {$2$} (240.center);
\draw[-] (303.center) to node[fill=white,inner sep=2pt] {$1$} (420.center);
\draw[-] (123) to node[fill=white,inner sep=2pt] {$3$} (303.center);
\draw[-] (123) to node[fill=white,inner sep=2pt] {$4$} (420.center);

\draw[-] (e1.center) to (e2.center) to (e3.center) to (e1.center);
\end{tikzpicture}
\caption{$G$-Hilb and Reid's recipe for~$\frac{1}{6}(1,2,3)$.}\label{fig:1/6b}
\end{figure}

We compute the inequalities coming from curves and divisors that define $\mfk{C}_0$ via the unlocking procedure:
\begin{gather}
\label{6:a1} \tag{$A_1$}
\theta(\chi_1)>0, \\
\label{6:a2} \tag{$A_2$}
\theta(\chi_2)+\theta(\chi_5)>0 \\
\label{6:b2} \tag{$B_2$}
\theta(\chi_2)+\theta(\chi_3)+2\theta(\chi_4)+2\theta(\chi_5)>0, \\
\label{6:a3} \tag{$A_3$}
\theta(\chi_3)+\theta(\chi_5)>0, \\
\label{6:b3} \tag{$B_3$}
\theta(\chi_3)+\theta(\chi_4)+\theta(\chi_5)>0, \\
\label{6:a4} \tag{$A_4$}
\theta(\chi_4)>0, \\
\label{6:a5} \tag{$A_5$}
\theta(\chi_5)>0, \\
\label{6:b5} \tag{$B_5$}
\theta(\chi_2)+\theta(\chi_3)+\theta(\chi_4)+\theta(\chi_5)>0.
\end{gather}
(\ref{6:a1}) is from the curve marked with the essential character $1$. Similarly for (\ref{6:a4}). We~then have two inequalities (\ref{6:a2}) and (\ref{6:b2}) coming from the two $2$-curves, and two (\ref{6:a3}) and (\ref{6:b3}) from the two $3$-curves. The $5$-divisor gives two inequalities (\ref{6:a5}) and (\ref{6:b5}) for rigid subsheaves and quotients it parameterises.

We can see that (\ref{6:b2}) is redundant by expressing it as a~combination of (\ref{6:a2}), (\ref{6:a3}) and (\ref{6:a4}). Similarly, (\ref{6:b3}) can be expressed in~terms of (\ref{6:a3}) and (\ref{6:a4}). No further reductions are possible, and so the walls of $\mfk{C}_0$ (with their types) in~this example are
\begin{gather}
\tag{\texttt{I}} \theta(\chi_1)=0, \\
\tag{\texttt{III}} \theta(\chi_2)+\theta(\chi_5)=0, \\
\tag{\texttt{I}} \theta(\chi_3)+\theta(\chi_5)=0, \\
\tag{\texttt{I}} \theta(\chi_4)=0, \\
\tag{\texttt{0}} \theta(\chi_5)=0, \\
\tag{\texttt{0}} \theta(\chi_2)+\theta(\chi_3)+\theta(\chi_4)+\theta(\chi_5)=0.
\end{gather}
\end{Example}

\begin{Example} \label{ex:30walls} We continue with a~more detailed example for~$G=\frac{1}{30}(25,2,3)$. Continuing the calculations in~Section~\ref{ex:30}, we find that the inequalities from curves in~$G$-Hilb are
\begin{gather}
\label{30:a2} \tag{$A_2$} \theta_2+\theta_{27}+\theta_{22}+\theta_{17}>0, \\
\label{30:b2} \tag{$B_2$} \theta_2+\theta_5+\theta_8+\theta_{11}+\theta_{14}>0, \\
\label{30:a3} \tag{$A_3$} \theta_3+\theta_{13}+\theta_{18}+\theta_{23}+\theta_{28}>0, \\
\label{30:b3} \tag{$B_3$} \theta_3+\theta_5+\theta_7+\theta_9+\theta_{11}+\theta_{13}+\theta_{23}+\theta_{28}>0, \\
\label{30:c3} \tag{$C_3$} \theta_3+\theta_5+\theta_7+\theta_9+\theta_{11}+\theta_{13}+\theta_{15}+\theta_{17}+\theta_{19}+\theta_{21}>0, \\
\label{30:a4} \tag{$A_4$} \theta_4+\theta_{29}+\theta_{24}+\theta_{19}+\theta_{14}>0, \\
\label{30:b4} \tag{$B_4$} \theta_4+\theta_7+\theta_{29}+\theta_{24}+\theta_{19}>0, \\
\label{30:c4} \tag{$C_4$} \theta_4+\theta_7+\theta_{10}+\theta_{13}+\theta_{16}+\theta_{19}+\theta_{29}>0, \\
\label{30:d4} \tag{$D_4$} \theta_4+\theta_7+\theta_{10}+\theta_{13}+\theta_{16}+\theta_{19}+\theta_{22}>0, \\
\label{30:a5} \tag{$A_5$} \theta_5+\theta_7+\theta_9+\theta_{11}>0, \\
\label{30:b5} \tag{$B_5$} \theta_5+\theta_7+\theta_8+\theta_{11}>0, \\
\label{30:c5} \tag{$C_5$} \theta_5+\theta_8+\theta_{11}+\theta_{14}>0, \\
\theta_6+\theta_8+\theta_9+\theta_{10}+\theta_{11}+\theta_{13}+2\theta_{12}
+2\theta_{14}+2\theta_{16}+2\theta_{15}+2\theta_{17}+2\theta_{19}+3\theta_{18}\nonumber
\\ \phantom{\theta_6}
{}+3\theta_{20}\,{+}\,3\theta_{22}\,{+}\,3\theta_{21}\,{+}\,3\theta_{23}\,{+}\,3\theta_{25}\,{+}\,4\theta_{24}
\,{+}\,4\theta_{26}\,{+}\,4\theta_{28}\,{+}\,4\theta_{27}\,{+}\, 4\theta_{29}\,{+}\, 4\theta_1>0,
\label{30:a6} \tag{$\mathbf{A_6}$}
\\
\theta_6+\theta_8+\theta_{10}+2\theta_{12}+2\theta_{14}+2\theta_{16}
+3\theta_{18}+\theta_9+\theta_{11}+\theta_{13}+2\theta_{15}+2\theta_{17} \nonumber
\\ \phantom{\theta_6}
{}+2\theta_{19}+3\theta_{21}+\theta_1+4\theta_{26}+2\theta_{16}+3\theta_{23}
+3\theta_{20}+3\theta_{22}+4\theta_{24}+4\theta_{26}>0,\label{30:b6} \tag{$\mathbf{B_6}$}
\\
\theta_6+\theta_8+\theta_{10}+2\theta_{12}+2\theta_{14}+2\theta_{16}
+3\theta_{18}+\theta_9+\theta_{11}+\theta_{13}\nonumber
\\ \phantom{\theta_6}
{}+2\theta_{15}+2\theta_{17}+2\theta_{19}+3\theta_{21}+\theta_1
+\theta_{26}+\theta_{21}+\theta_{16}>0,\label{30:c6} \tag{$\mathbf{C_6}$}
\\
\label{30:d6} \tag{$\mathbf{D_6}$} \theta_6+\theta_8+\theta_{10}+2\theta_{12}+2\theta_{14}+2\theta_{16}+\theta_1
+\theta_{26}+\theta_{21}+\theta_{16}+\theta_9+\theta_{11}+\theta_{13}>0, \\
\label{30:e6} \tag{$\mathbf{E_6}$} \theta_6+\theta_1+\theta_{26}+\theta_{21}+\theta_{16}+\theta_{11}+\theta_8+\theta_9>0, \\
\label{30:f6} \tag{$\mathbf{F_6}$} \theta_6+\theta_1+\theta_{26}+\theta_{21}+\theta_{16}+\theta_{11}>0, \\
\label{30:a8} \tag{$A_8$} \theta_8>0, \\
\label{30:a9} \tag{$A_9$} \theta_9>0, \\
\label{30:a10} \tag{$A_{10}$} \theta_{10}+\theta_{13}+\theta_{16}>0, \\
\label{30:b10} \tag{$\mathbf{B_{10}}$} \theta_{10}+\theta_{12}+\theta_{14}+\theta_{16}+\theta_{18}+\theta_5
+\theta_7+\theta_9+\theta_{11}+\theta_{13}>0, \\
\label{30:c10} \tag{$\mathbf{C_{10}}$} \theta_{10}
+\theta_{13}+\theta_{12}+\theta_{14}+\theta_{16}>0, \\
\label{30:a12} \tag{$A_{12}$} \theta_{12}+\theta_7>0, \\
\label{30:b12} \tag{$B_{12}$} \theta_{12}+\theta_{14}>0, \\
\label{30:a15} \tag{$A_{15}$} \theta_{15}+\theta_{17}+\theta_{19}+\theta_{21}>0, \\
\label{30:b15} \tag{$\mathbf{B_{15}}$} \theta_{15}+\theta_{17}+\theta_{19}+\theta_{18}+\theta_{21}>0, \\
\label{30:c15} \tag{$\mathbf{C_{15}}$} \theta_{15}+\theta_{17}+\theta_{18}+\theta_{21}+\theta_{24}
+\theta_{10}+\theta_{13}+\theta_{16}+\theta_{19}>0,
\\
\theta_{15}+\theta_{18}+\theta_{21}+\theta_{24}+\theta_{27}+\theta_{10}+\theta_{13}+\theta_{16} +\theta_{19}+\theta_{22}+\theta_5+\theta_8\nonumber
\\ \phantom{\theta_{15}}
+\theta_{11}+\theta_{14}+\theta_{17}>0, \label{30:d15} \tag{$\mathbf{D_{15}}$}
\\
\label{30:a18} \tag{$A_{18}$} \theta_{18}>0, \\
\label{30:a20} \tag{$A_{20}$} \theta_{20}+\theta_{23}+\theta_{26}+\theta_{29}>0, \\
\label{30:b20} \tag{$B_{20}$} \theta_{20}+\theta_{22}+\theta_{23}+\theta_{26}>0, \\
\label{30:c20} \tag{$\mathbf{C_{20}}$} \theta_{20}+\theta_{23}+\theta_{22}+\theta_{24}+\theta_{26}>0, \\
\label{30:d20} \tag{$\mathbf{D_{20}}$} \theta_{20}+\theta_{15}+\theta_{17}+\theta_{19}
+\theta_{21}+\theta_{22}+\theta_{24}+\theta_{26}+\theta_{28}>0, \\
\label{30:a24} \tag{$A_{24}$} \theta_{24}>0, \\
\label{30:a25} \tag{$A_{25}$} \theta_{25}+\theta_{27}+\theta_{29}+\theta_1>0, \\
\label{30:b25} \tag{$B_{25}$} \theta_{25}+\theta_{28}+\theta_1>0, \\
\label{30:a27} \tag{$A_{27}$} \theta_{27}+\theta_{22}>0, \\
\label{30:b27} \tag{$B_{27}$} \theta_{27}+\theta_{29}>0, \\
\label{30:a28} \tag{$A_{28}$} \theta_{28}>0.
\end{gather}

The bolded inequalities correspond to curves $C$ with $\mathcal{N}_C$ not of type $(-1,-1)$. We~know by \cite[Theorem~9.12]{CrawIshii04} that the other inequalities are necessary and define Type \texttt{I} walls of $\mfk{C}_0$. The inequalities from divisors parameterising rigid subsheaves are
\begin{gather}
\label{30:a1} \tag{$A_1$} \theta_1>0, \\
\label{30:a7} \tag{$A_7$} \theta_7>0, \\
\label{30:a11} \tag{$A_{11}$} \theta_{11}>0, \\
\label{30:a13} \tag{$A_{13}$} \theta_{13}>0, \\
\label{30:a14} \tag{$A_{14}$} \theta_{14}>0, \\
\label{30:a16} \tag{$A_{16}$} \theta_{16}>0, \\
\label{30:a17} \tag{$A_{17}$} \theta_{17}>0, \\
\label{30:a19} \tag{$A_{19}$} \theta_{19}>0, \\
\label{30:a21} \tag{$A_{21}$} \theta_{21}>0, \\
\label{30:a22} \tag{$A_{22}$} \theta_{22}>0, \\
\label{30:a23} \tag{$A_{23}$} \theta_{23}>0, \\
\label{30:a26} \tag{$A_{26}$} \theta_{26}>0, \\
\label{30:a29} \tag{$A_{29}$} \theta_{29}>0.
\end{gather}

We record the redundancies for the bold (or~potentially redundant) inequalities:
\begin{gather*}
\begin{array}{lcl}
\text{(\ref{30:f6})}+\text{(\ref{30:a8})}+\text{(\ref{30:a9})}+\text{(\ref{30:a10})}
+\text{(\ref{30:b12})}+\text{(\ref{30:a15})}
+\text{(\ref{30:a18})}+\text{(\ref{30:a20})}+\text{(\ref{30:a24})}
\\ \phantom{\text{(\ref{30:f6})}}
{}+\text{(\ref{30:a25})}+\text{(\ref{30:b27})}+\text{(\ref{30:a28})}
&\Longrightarrow&\text{(\ref{30:a6})}, \\
\text{(\ref{30:f6})}+\text{(\ref{30:a8})}+\text{(\ref{30:a9})}+\text{(\ref{30:a10})}+\text{(\ref{30:b12})} +\text{(\ref{30:a15})}+\text{(\ref{30:a18})}+\text{(\ref{30:a20})}+\text{(\ref{30:a24})}&\Longrightarrow&\text{(\ref{30:b6})}, \\
\text{(\ref{30:f6})}+\text{(\ref{30:a8})}+\text{(\ref{30:a9})}+\text{(\ref{30:a10})}+\text{(\ref{30:b12})}+\text{(\ref{30:a15})} +\text{(\ref{30:a18})}&\Longrightarrow&\text{(\ref{30:c6})}, \\
\text{(\ref{30:f6})}+\text{(\ref{30:a8})}+\text{(\ref{30:a9})}+\text{(\ref{30:a10})}+\text{(\ref{30:b12})}& \Longrightarrow&\text{(\ref{30:d6})}, \\
\text{(\ref{30:f6})}+\text{(\ref{30:a8})}+\text{(\ref{30:a9})}&\Longrightarrow&\text{(\ref{30:e6})}, \\
\text{(\ref{30:a5})}+\text{(\ref{30:b12})}+\text{(\ref{30:a18})}&\Longrightarrow&\text{(\ref{30:b10})}, \\
\text{(\ref{30:a10})}+\text{(\ref{30:b12})}&\Longrightarrow&\text{(\ref{30:c10})}, \\
\text{(\ref{30:a15})}+\text{(\ref{30:a18})}&\Longrightarrow&\text{(\ref{30:b15})}, \\
\text{(\ref{30:a15})}+\text{(\ref{30:a18})}+\text{(\ref{30:a10})}+\text{(\ref{30:a24})}&\Longrightarrow&\text{(\ref{30:c15})}, \\
\text{(\ref{30:a15})}+\text{(\ref{30:a18})}+\text{(\ref{30:a10})}+\text{(\ref{30:a24})}+\text{(\ref{30:a27})}+\text{(\ref{30:c5})} &\Longrightarrow&\text{(\ref{30:d15})}, \\
\text{(\ref{30:b20})}+\text{(\ref{30:a24})}&\Longrightarrow&\text{(\ref{30:c20})}, \\
\text{(\ref{30:a15})}+\text{(\ref{30:b20})}+\text{(\ref{30:a24})}&\Longrightarrow&\text{(\ref{30:d20})}.
\end{array}
\end{gather*}
We have killed off the inequalities from all curves except for the $(-1,-1)$-curves and one curve~(\ref{30:f6}) from the long side.
\end{Example}

\subsection{Redundant inequalities from curves}

Observe that the vast majority of inequalities in~Examples \ref{ex:6walls}--\ref{ex:30walls} define walls of Type \texttt{I}. We~should be unsurprised by the cancellation of all except one bolded inequality in~Example~\ref{ex:30walls} due to the following result from~\cite{CrawIshii04}.

\begin{Lemma}[{\cite[Corollaries 6.3 and 6.5]{CrawIshii04}}] \label{lem:coeff} Suppose $w=(\sum\alpha_i\theta_i=0)$ is a~Type \texttt{I} or \texttt{III} wall of $\mfk{C}_0$. Then all $\alpha_i\in\{0,1\}$.
\end{Lemma}

Chambers other than $\mfk{C}_0$ can have coefficients $\alpha_i=-1$, however since the trivial representation does not appear in~$\Gig(C)$ for any curve $C$ we can exclude this possibility.

\begin{Corollary} Suppose $G\hilb\A^3$ has a~meeting of champions of side length $0$. Then the inequality for any curve along one of the three champions is redundant.
\end{Corollary}

\begin{proof} Suppose $\chi$ is the character marking each of the champions. Then, by Theorem~\ref{thm:craw}, $r_\chi^2=r_{\chi^2}$ globally on~$G$-Hilb and so $\deg (\mR_{\chi^2}|_C)=2$ for all $\chi$-curves $C$. It~follows from Lemma~\ref{lem:coeff} that none of these inequalities can be strict.
\end{proof}

We can also show this directly via unlocking. This reproves Corollary~\ref{cor:ii}.

\begin{Lemma} \label{lem:sq} Suppose $C$ is a~$\chi$-curve. If the unlocking procedure for~$C$ doesn't unlock a~curve or divisor marked with $\chi^2$ then all the coefficients in~the inequality $\theta(\ph_{\mfk{C}_0}(\mO_C))>0$ are equal to~$0$ or~$1$.
\end{Lemma}

\begin{proof} This is because if some $\rho$ has $\deg (\mR_\rho|_C)\geq 2$ then $r_\chi^2\,|\, r_\rho$ and so $r_\chi^2$ must feature in~the $G$-igsaw piece for~$C$ and is hence equal to $r_{\chi^2}$ near $C$.
\end{proof}

\begin{Lemma} \label{lem:mult} Suppose a~curve $C_0$ unlocks a~curve $C_1$ of character $\rho$. Let $\psi\in\Gig(C_1)$. If $C$ is a~curve that unlocks $C_0$, then $\deg (\mR_\psi|_C)\geq\deg (\mR_\rho|_C)$.
\end{Lemma}

\begin{proof} As used previously, $\deg (\mR_\rho|_C)=\max \big\{k\in\Z_{\geq0}\colon r_\chi^k\,|\, r_\rho \big\}$. From this formulation, clearly if $r_\rho\,|\, r_\psi$ then $\deg (\mR_\psi|_C)\geq\deg (\mR_\rho|_C)$, but this is the case by definition of $G$-igsaw piece.
\end{proof}

\begin{Lemma} \label{lem:bdry} Suppose $C$ is a~curve on the boundary of a regular triangle marked with a~character $\chi$. Suppose the $\chi$-chain contains a~$(-1,-1)$-curve. Then the inequality $\theta(\ph_{\mfk{C}_0}(\mO_C))>0$ is redundant.
\end{Lemma}

\begin{proof} Suppose $C$ is marked with character $\chi$. Let $C_0$ be the first $(-1,-1)$-curve in~the $\chi$-chain moving inwards from $C$. Then the $G$-igsaw piece for~$C$ consists of exactly the characters in~the $G$-igsaw piece for~$C_0$ along with the characters in~the $G$-igsaw pieces for any curves $C_1,\dots,C_n$ unlocked by $C$ at Hirzebruch divisors before $C_0$. Let the character marking $C_i$ be $\chi_i$. The inequality for~$C$ decomposes as
\begin{gather} \label{eqn:wc0} 
\theta\big(\ph_{\mfk{C}_0}(\mO_C)\big)=\sum_{\rho\in\Gig(C_0)}\alpha_\rho\theta(\rho)
+\sum_{i=1}^n\sum_{\rho\in\Gig(C_i)}\beta_\rho^i\theta(\rho),
\end{gather}
where $\alpha_\rho$ and $\beta_\rho^i$ are nonnegative multiplicities given by the appropriate calculation of $\deg (\mR_\rho|_?)$, possibly computing the degree of $\mR_\rho$ on multiple curves. Note that $\alpha_\chi=1$. One can thus write
\begin{gather*}
\theta\big(\ph_{\mfk{C}_0}(\mO_C)\big)=\theta\big(\ph_{\mfk{C}_0}(\mO_{C_0})\big)
\\ \hphantom{\theta(\ph_{\mfk{C}_0}(\mO_C))=}
+\!\!\sum_{\rho\in\Gig(C_0)}\!\!(\alpha_\rho-1)\theta(\rho) +\sum_{i=1}^m\Bigg(\beta_{\chi_i}^i\theta\big(\ph_{\mfk{C}_0}(\mO_{C_i})\big)
+\!\!\sum_{\rho\in\Gig(C_i)}\!\!(\beta_\rho^i-\beta_{\chi_i}^i)\theta(\rho)\!\Bigg).
\end{gather*}
From Lemma~\ref{lem:mult}, $\alpha_\rho-1$ and $\beta_\rho-\beta_{\chi'}$ are both nonnegative. If all the remaining $\rho$ in~these sums with nonzero coefficients after this reduction are characters marking divisors then one can express each term $\gamma_\rho\theta(\rho)=\gamma_\rho\theta\big(\ph_{\mfk{C}_0}\big(\mR_\psi^{-1}|_D\big)\big)$ for some divisor $D$, thus evidencing that~(\ref{eqn:wc0}) is redundant. Suppose instead that some $\rho=\rho_1$ marks a~curve unlocked by $C_0$ or some $C_i$. We~assume the latter; the former is treated identically. Denote this new curve by $C_{i,1}$. Then
\begin{gather*}
\sum_{\rho\in\Gig(C_i)}\big(\beta_\rho^i-\beta_{\chi_i}^i\big)\theta(\rho)
=\big(\beta_{\rho_1}^i-\beta_{\chi_i}^i\big)\theta\big(\ph_{\mfk{C}_0}(\mO_{C_{i,1}})\big)
\\ \hphantom{\sum_{\rho\in\Gig(C_i)}\big(\beta_\rho^i-\beta_{\chi_i}^i\big)\theta(\rho)=}
{}+\sum_{\rho\in\Gig(C_{i,1})}\big(\beta_\rho^i-\beta_{\rho_1}^i\big)\theta(\rho)
+\sum_{\rho\notin\Gig(C_{i,1})}\big(\beta_\rho^i-\beta_{\chi_i}^i\big)\theta(\rho),
\end{gather*}
where again each coefficient is nonnegative by Lemma~\ref{lem:mult} applied to $C_{i,1}$. Observe that there are strictly fewer nonzero coefficients in~this expression than before, since at the least we removed the term for~$\rho_1$. Continuing in~this way for each character appearing that marks a~curve, we can reduce to the situation where the only characters with nonzero coefficients in~the error term are those that mark divisors. At that point we have already seen how to express the error term in~terms of inequalities coming from divisors, and so we have shown that (\ref{eqn:wc0}) is redundant.
\end{proof}

\subsection[Classifying Type III walls]{Classifying Type \texttt{III} walls}
We provide a~combinatorial classification of the Type \texttt{III} walls for~$\mfk{C}_0$. We~start with the following definition.

\begin{Definition} \label{def:gls} Let $\chi$ be a~character marking a~curve in~$G$-Hilb. We~say that the $\chi$-chain is a~\textit{generalised long side} if it starts and ends on the boundary of the junior simplex, and all the edges along the $\chi$-chain are boundary edges of regular triangles. We~exclude the lines meeting at a~trivalent vertex if there is a~meeting of champions of side length $0$ from this definition.
\end{Definition}

For example, any long side is a~generalised long side. The $15$-chain for~$\frac{1}{35}(1,3,31)$ is a~ge\-neralised long side as can be seen in~Fig.~\ref{fig:1/35}.

\begin{Example} We compute the inequalities for curves along the $15$-chain in~$G$-Hilb for~$G=\frac{1}{35}(1,3,31)$. From the unlocking procedure or computing $G$-igsaw pieces directly, the inequalities for the $15$-curves starting from $e_1$ and moving downwards are
\begin{gather}
\label{35:a15} \tag{$A_{15}$} \theta_{15}+\theta_{18}+\theta_{21}+\theta_{24}+\theta_7+\theta_{10}+\theta_{13}+\theta_{16}
+\theta_{11}+\theta_{14}+\theta_{17}+\theta_{20}>0 ,
\\[-.3ex]
\label{35:b15} \tag{$B_{15}$} \theta_{15}+\theta_{18}+\theta_{21}+\theta_{16}+\theta_{11}+\theta_{14}+\theta_{17}>0 ,
\\[-.3ex]
\label{35:c15} \tag{$C_{15}$} \theta_{15}+\theta_{16}+\theta_{17}+\theta_{18}>0 ,
\\[-.3ex]
\label{35:d15} \tag{$D_{15}$} \theta_{15}+\theta_{16}+\theta_{17}+\theta_{18}>0.
\end{gather}
Clearly (\ref{35:c15}) and (\ref{35:d15}) depend on each other; the inequality is the same since they are fibres of the $\pr^1$-bundle structure on the Hirzebruch surface marked with $18$, and so contracting one must contract the other. We~consider some of the additional inequalities coming from $(-1,-1)$-curves:
\begin{gather}
\label{35:a7} \tag{$A_7$} \theta_7+\theta_{10}+\theta_{13}>0, \\[-.3ex]
\label{35:a11} \tag{$A_{11}$} \theta_{11}+\theta_{14}>0, \\[-.3ex]
\label{35:a21} \tag{$A_{21}$} \theta_{21}>0, \\[-.3ex]
\label{35:a24} \tag{$A_{24}$} \theta_{24}+\theta_{20}>0.
\end{gather}
We can deduce
\begin{gather*}
\text{(\ref{35:c15})}+\text{(\ref{35:a7})}+\text{(\ref{35:a11})}
+\text{(\ref{35:a21})}+\text{(\ref{35:a24})}\Longrightarrow\text{(\ref{35:a15})},
\\
\text{(\ref{35:c15})}+\text{(\ref{35:a11})}+\text{(\ref{35:a21})}\Longrightarrow\text{(\ref{35:b15})},
\end{gather*}
so that (\ref{35:a15}) and (\ref{35:b15}) are redundant.
\end{Example}

\begin{Definition} Consider a~generalised long side marked with character $\chi$. Recall that each $\chi$-chain consists of potentially several straight line segments. We~call a~curve in~the $\chi$-chain \textit{final} if it is the furthest curve along the $\chi$-chain away from a~vertex along such a~line segment.
\end{Definition}

For example, for~$G=\frac{1}{35}(1,3,31)$, the bolded curves in~Fig.~\ref{fig:final} are final.

\begin{figure}[h]\centering
\begin{tikzpicture}[scale=0.4]
\node (e1) at (0,9){$\bullet$};
\node (e2) at (10,-6){$\bullet$};
\node (e3) at (-10,-6){$\bullet$};

\draw (e1.center) to (e2.center) to (e3.center) to (e1.center);

\node (a1) at (-10+2,-6+1/2){$\bullet$};
\node (a2) at (-10+4,-6+2/2){$\bullet$};
\node (a3) at (-10+6,-6+3/2){$\bullet$};
\node (a4) at (-10+8,-6+4/2){$\bullet$};
\node (a5) at (-10+10,-6+5/2){$\bullet$};
\node (a6) at (-10+12,-6+6/2){$\bullet$};
\node (a7) at (-10+14,-6+7/2){$\bullet$};
\node (a8) at (6,-2){$\bullet$};

\node (b1) at (2,9-11/3){$\bullet$};
\node (b2) at (4,9-22/3){$\bullet$};

\node (c1) at (-10+6-2,-6+3/2+11/3){$\bullet$};
\node (c2) at (-10+8-2,-6+4/2+11/3){$\bullet$};
\node (c3) at (-10+10-2,-6+5/2+11/3){$\bullet$};
\node (c4) at (-10+12-2,-6+6/2+11/3){$\bullet$};
\node (c5) at (-10+14-2,-6+7/2+11/3){$\bullet$};

\node (d1) at (2-4,9-2/2-11/3){$\bullet$};
\node (d2) at (2-2,9-1/2-11/3){$\bullet$};

\draw(e2.center) to (a1.center);
\draw[line width = 1.2pt](e2.center) to (a2.center);
\draw(e2.center) to (a3.center);
\draw(e2.center) to (a4.center);
\draw[line width = 1.2pt](e2.center) to (a5.center);
\draw(e2.center) to (a6.center);
\draw(e2.center) to (a7.center);
\draw(e2.center) to (a8.center);

\draw(e2.center) to (b1.center);
\draw(e2.center) to (b2.center);

\draw(e3.center) to (c1.center) to (d1.center);

\draw(e1.center) to (c1.center);
\draw(e1.center) to (d1.center) to (c2.center);
\draw[line width = 1.2pt] (c2.center) to (a2.center);
\draw(e1.center) to (d2.center) to (c4.center);
\draw[line width = 1.2pt] (c4.center) to (a5.center);

\draw (e3.center) to (a1.center) to (a2.center) to (a3.center) to (a4.center) to (a5.center) to (a6.center) to (a7.center) to (a8.center);

\draw (e1.center) to (b1.center) to (b2.center) to (a8.center);

\draw (a1.center) to (c1.center);
\draw (a1.center) to (c2.center);
\draw (c1.center) to (c2.center);

\draw (d1.center) to (d2.center);
\draw (a3.center) to (c2.center);
\draw (a4.center) to (c4.center);
\draw (d1.center) to (c3.center) to (a4.center);
\draw (c2.center) to (c3.center) to (c4.center);
\draw (a3.center) to (c3.center) to (d2.center);

\draw (d2.center) to (b1.center);
\draw (a7.center) to (b2.center);
\draw (a6.center) to (c4.center);
\draw (a6.center) to (c5.center) to (b1.center);
\draw (c4.center) to (c5.center) to (b2.center);
\draw (a7.center) to (c5.center) to (d2.center);
\end{tikzpicture}
\caption{Final curves for~$G=\frac{1}{35}(1,3,31)$.}\label{fig:final}
\end{figure}
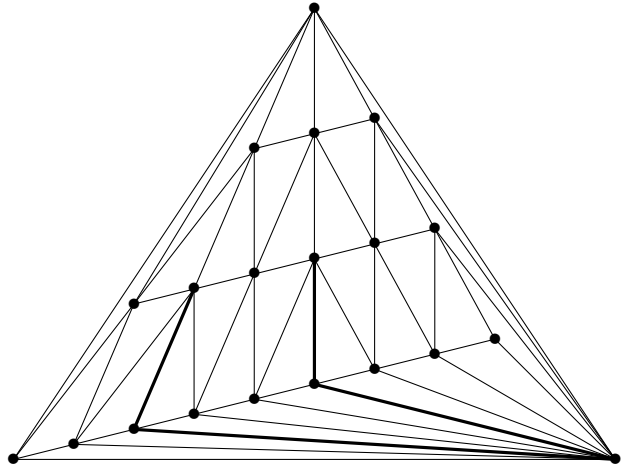

Final curves not along a~long side are also those contained in~an exceptional Hirzebruch surface (with no blowups) or, equivalently, those corresponding to edges incident to a~$4$-valent vertex. There can be at most two final curves for each generalised long side, with exactly one when the generalised long side is actually a~long side.

\begin{Lemma} \label{lem:bdry2} Suppose $\chi$ is a~character marking a~curve and that the $\chi$-chain is a~generalised long side. Then, the inequality for each non-final curve $C$ in~the $\chi$-chain is redundant. The final curves all produce the same inequality:
\[
\theta(\chi)+\sum_{\psi\in\Hirz(\chi)}\theta(\psi)>0,
\]
which is a~necessary inequality defining a~Type \texttt{III} wall of $\mfk{C}_0$.
\end{Lemma}

\begin{proof} First, the inequality for a~final $\chi$-curve $C$ features only the Hirzebruch divisors along the $\chi$-chain by the unlocking procedure. It~has all nonzero coefficients equal to $1$ for the following reason. $\chi^2$ cannot mark a~Hirzebruch divisor along the $\chi$-chain because to do so one would require another chain, say with character $\rho$, to cross the $\chi$-chain and have $\chi\otimes\rho=\chi^2$. Of course, this would mean that $\rho=\chi$, but chains do not self-intersect. Hence, $\chi^2$ does not appear in~the $G$-igsaw piece for~$C$ and so all multiplicities must be equal to $1$ by Lemma~\ref{lem:sq}. This is clearly a~necessary inequality, as $\chi$ is the only character in~the inequality coming from a~curve and there is no divisor that contains only $\chi$-curves~-- in~contrast to the case of a~trivalent vertex.

To see that the other inequalities coming from curves along a~generalised long side are redundant, we will decompose these inequalities similarly to before. Let $C$ be such a~curve and write
\[
\theta\big(\ph_{\mfk{C}_0}(\mO_C)\big)=\theta(\chi)+\sum_{\psi\in\Hirz(C)}\alpha_\psi\theta(\psi)
+\sum_{i=1}^n\sum_{\rho\in\Gig(C_i)}\beta^i_\rho\theta(\rho),
\]
where $C_1,\dots,C_n$ are the curves unlocked by $C$. By exactly the same methods as in~the proof of Lemma~\ref{lem:bdry}, one can express the final term as a~sum of inequalities from curves and divisors. The first two terms are equal to
\[
\theta(\chi)+\sum_{\psi\in\Hirz(C)}\alpha_\psi\theta(\psi)
=\theta\big(\ph_{\mfk{C}_0}(\mO_{C'})\big)+\sum_{\psi\in\Hirz(\chi)}(\alpha_\psi-1)
\theta\big(\ph_{\mfk{C}_0}\big(\mR_\psi^{-1}|_{D_\psi}\big)\big),
\]
where $C'$ is a~final $\chi$-curve and $D_\psi$ is the divisor marked with $\psi$. Of course $\alpha_\psi\geq1$ and so we have shown that the inequality from $C$ is redundant.
\end{proof}

We consider the example $G=\frac{1}{25}(1,3,21)$, which has a~meeting of champions of side length~$2$.

\begin{Example} We show the triangulation for~$G$-Hilb and Reid's recipe for~$G=\frac{1}{25}(1,3,21)$ in~Fig.~\ref{fig:25rr}. Observe that of the three champions, the $3$-chain and $9$-chain are generalised long sides whilst the $1$-chain contains a~$(-1,-1)$-curve. We~hence obtain two Type \texttt{III} walls from the champions and another for the $21$-chain that is also a~generalised long side, with inequalities
\begin{gather}
\tag{$F_3$} \theta_3+\theta_4+\theta_8+\theta_{12}+\theta_{16}+\theta_{20}+\theta_{24}>0,
\\
\tag{$C_9$} \theta_9+\theta_{10}+\theta_{11}+\theta_{12}>0,
\\
\tag{$C_{21}$} \theta_{21}+\theta_{22}+\theta_{23}+\theta_{24}>0.
\end{gather}

\begin{figure}[h]\centering
\begin{tikzpicture}[scale=0.6]
\node (e1) at (0,9){$\bullet$};
\node (e2) at (10,-6){$\bullet$};
\node (e3) at (-10,-6){$\bullet$};

\draw (e1.center) to (e2.center) to (e3.center) to (e1.center);

\footnotesize

\node[fill=white,draw,inner sep=2pt] (a1) at (-10+1*14/5,-6+1*3/5){$24$};
\node[fill=white,draw,inner sep=2pt] (a2) at (-10+2*14/5,-6+2*3/5){$20$};
\node[fill=white,draw,inner sep=2pt] (a3) at (-10+3*14/5,-6+3*3/5){$16$};
\node[fill=white,draw,inner sep=2pt] (a4) at (-10+4*14/5,-6+4*3/5){$12$};
\node[fill=white,draw,inner sep=2pt] (a5) at (-10+5*14/5,-6+5*3/5){$8$};
\node[fill=white,draw,inner sep=2pt] (a6) at (-10+6*14/5,-6+6*3/5){$4$};

\node[fill=white,draw,inner sep=2pt] (a12) at (-10+5*14/5-2/5,-6+5*3/5+21/5){$7$};

\node[fill=white,draw,inner sep=2pt] (a11) at (-10+4*14/5-2/5,-6+4*3/5+21/5){$11$};
\node[fill=white,draw,inner sep=2pt] (a18) at (-10+4*14/5-2*2/5,-6+4*3/5+2*21/5){$10$};

\node[fill=white,draw,inner sep=2pt] (a9) at (-10+2*14/5-2/5,-6+2*3/5+21/5){$23$};
\node[fill=white,draw,inner sep=2pt] (a10) at (-10+3*14/5-2/5,-6+3*3/5+21/5){$\begin{matrix} 14 \\ 19 \end{matrix}$};

\node[fill=white,draw,inner sep=2pt] (a17) at (-10+3*14/5-2*2/5,-6+3*3/5+2*21/5){$22$};

\draw(e1.center) to node[fill=white,inner sep=2pt] {$9$} (a18) to node[fill=white,inner sep=2pt] {$9$} (a11) to node[fill=white,inner sep=2pt] {$9$} (a4);
\draw(e2.center) to node[fill=white,inner sep=1pt] {$1$} (a6) to node[fill=white,inner sep=2pt] {$1$} (a12) to node[fill=white,inner sep=2pt] {$1$} (a18);
\draw(e1.center) to node[fill=white,inner sep=2pt] {$21$} (a17) to node[fill=white,inner sep=2pt] {$21$} (a9) to node[fill=white,inner sep=2pt] {$21$} (a1);
\draw(e3.center) to node[fill=white,inner sep=1pt] {$3$} (a1) to node[fill=white,inner sep=2pt] {$3$} (a2) to node[fill=white,inner sep=2pt] {$3$} (a3) to node[fill=white,inner sep=2pt] {$3$} (a4) to node[fill=white,inner sep=2pt] {$3$} (a5) to node[fill=white,inner sep=2pt] {$3$} (a6);
\draw(e3.center) to node[fill=white,inner sep=2pt] {$2$} (a9);
\draw(e3.center) to node[fill=white,inner sep=1pt] {$1$} (a17);
\draw(e2.center) to node[fill=white,inner sep=2pt] {$21$} (a1);
\draw(e2.center) to node[fill=white,inner sep=2pt] {$17$} (a2);
\draw(e2.center) to node[fill=white,inner sep=2pt] {$13$} (a3);
\draw(e2.center) to node[fill=white,inner sep=2pt] {$9$} (a4);
\draw(e2.center) to node[fill=white,inner sep=2pt] {$5$} (a5);
\draw(e1.center) to node[fill=white,inner sep=2pt] {$6$} (a12);
\draw(e1.center) to node[fill=white,inner sep=1pt] {$3$} (a6);

\draw[dashed](a17) to node[fill=white,inner sep=2pt] {$13$} (a10) to node[fill=white,inner sep=2pt] {$13$} (a3);
\draw[dashed](a17) to node[fill=white,inner sep=2pt] {$1$} (a18);
\draw[dashed](a9) to node[fill=white,inner sep=2pt] {$2$} (a10) to node[fill=white,inner sep=2pt] {$2$} (a11);
\draw[dashed](a2) to node[fill=white,inner sep=2pt] {$18$} (a10) to node[fill=white,inner sep=2pt] {$18$} (a18);
\draw[dashed](a3) to node[fill=white,inner sep=2pt] {$15$} (a11);
\draw[dashed](a2) to node[fill=white,inner sep=2pt] {$17$} (a9);

\draw[dashed](a12) to node[fill=white,inner sep=2pt] {$6$} (a5);
\draw[dashed](a11) to node[fill=white,inner sep=2pt] {$2$} (a12);
\draw[dashed](a11) to node[fill=white,inner sep=2pt] {$5$} (a5);
\end{tikzpicture}
\caption{Reid's recipe for~$G=\frac{1}{25}(1,3,21)$.}\label{fig:25rr}
\end{figure}
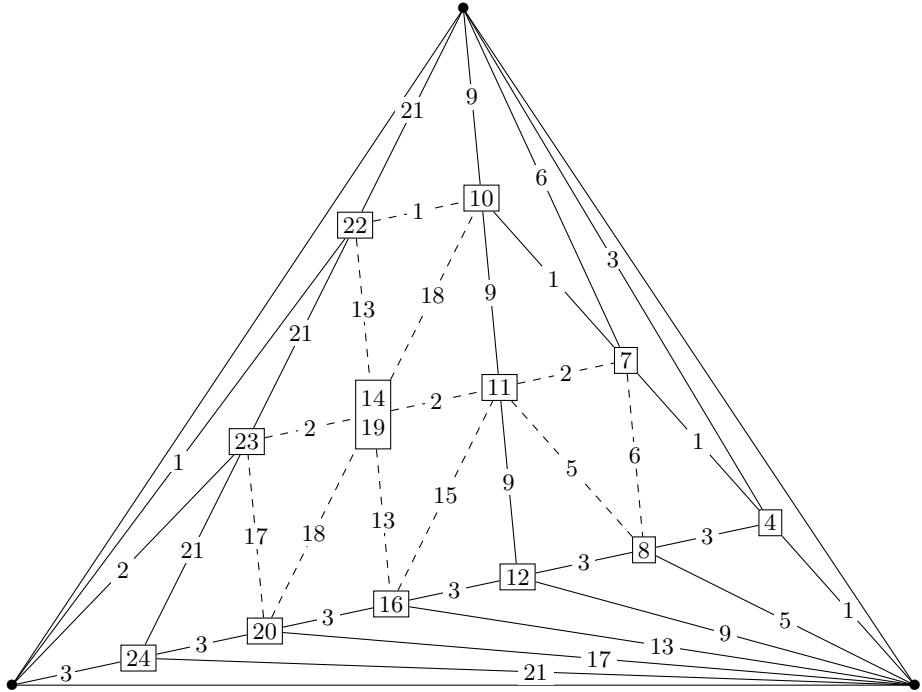
\end{Example}

\subsection{Summary} \label{sec:sum}

We compile the main results~-- Corollary~\ref{cor:ii}, Proposition~\ref{prop:-1nec}, Lemmas~\ref{lem:bdry} and~\ref{lem:bdry2}~-- of this section.

\begin{Theorem} \label{thm:main} Suppose $G\subset\on{SL}_3(\C)$ is a~finite abelian subgroup. The walls of the chamber $\mfk{C}_0$ for~$G\hilb\A^3$ and their types are as follows:
\begin{itemize}\itemsep=0pt
\item a~Type \texttt{I} wall for each exceptional $(-1,-1)$-curve,
\item a~Type \texttt{III} wall for each generalised long side,
\item a~Type \texttt{0} wall for each irreducible exceptional divisor,
\item each remaining wall is of Type \texttt{0} and comes from a~divisor parameterising a~rigid quotient.
\end{itemize}
Moreover, for every contraction of Type \texttt{I} or \texttt{III} for~$G\hilb\A^3$ there is a~wall of the correspon\-ding type that induces the contraction by VGIT.
\end{Theorem}

Proposition~\ref{prop:rigid} describes how to recover the unstable locus or the corresponding reducible divisor $D'$ for each wall of Type \texttt{0} from a~rigid quotient. Let $\mfk{w}$ be a~wall of $\mfk{C}_0$. Denote by~$E(\mfk{w})$ the set of edges in~the Craw--Reid triangulation corresponding to curves $C$ for which all characters in~$\Gig(C)$ appear in~the equation of the wall. The desired divisor $D'$ inducing~$\mfk{w}$ is then the union of the divisors corresponding to vertices for which all incident edges are in~$E(\mfk{w})$. We~observe that the unlocking procedure allows the check of which walls from rigid quotients are necessary to be performed combinatorially.

\section{Future directions}

There are several natural avenues of further study opened up by the results of this paper, three of which are
\begin{itemize}\itemsep=0pt
\item attuning the results here with the derived interpretation of Reid's recipe~\cite{CautisCrawLogvinenko17},
\item exploring any relations between analogs to Reid's recipe in~other settings and walls in~stability (for instance, dimer models~\cite{BocklandtCrawVelez15,IshiiUeda16}),
\item reverse-engineering a~partial Reid's recipe for other resolutions $\M_\mfk{C}$ from an explicit des\-crip\-tion of the walls of $\mfk{C}$, and examining whether this has any categorical content.
\end{itemize}

\subsection*{Acknowledgements}

The author would like to thank Yukari Ito and Nagoya University for hosting him as this research began. He would also like to thank Alastair Craw, \'Alvaro Nolla de Celis, and David Nadler for many fruitful and enjoyable conversations about this project, as well as the referees for their thoughtful suggestions on how to improve its exposition.

\pdfbookmark[1]{References}{ref}
\LastPageEnding

\end{document}